\documentclass[12pt]{article}   
\usepackage{mathrsfs}
\usepackage{amsmath,amsthm}
\usepackage[usenames,dvipsnames]{pstricks}
\usepackage{epsfig}
\usepackage{pst-grad} 
\usepackage{pst-plot} 
\usepackage[space]{grffile} 
\usepackage{etoolbox}
\usepackage{epstopdf}
\usepackage{amssymb}
\usepackage{latexsym}
\usepackage{color}
\usepackage[usenames,dvipsnames]{pstricks}
 \usepackage{pst-grad} % For gradients
\usepackage{pst-plot}
\usepackage[latin1]{inputenc}
\usepackage{psfrag}
\usepackage{graphics}
\usepackage{graphicx}
\usepackage{amssymb}
\usepackage{latexsym}
\usepackage{amsthm}
\newcommand{\pb}{\rho_2}
\newcommand{\pr}{\rho_r}
\newcommand{\lin}{\mathbb{L}}
\newcommand{\x}{\mathbf{x}}

\newcommand{\kk}{\mathbf{k}}

\newtheorem{theo}{\bf Theorem}
\newtheorem{defi}{\bf Definition}
\newtheorem{lem}{\bf Lemma}
\newtheorem{cor}{\bf Corollary}
\newtheorem{prop}{\bf Proposition}
\newtheorem{rem}{\bf Remark}
\newtheorem{ex}{\bf Example}

\newcommand{\be}{\begin{equation}}
\newcommand{\eeq}{\end{equation}}

\newcommand{\xx}{\mathbf{x}}
\newcommand{\nn}{\mathbf{n}}

\begin{document}

\title{Combinatorial differential operators in: Fa\`a di Bruno formula, enumeration of ballot paths, enriched rooted trees and increasing rooted trees.}

\author{M. A. M\'endez\\ Facultad de Ciencias\\Universidad Antonio Nari\~no\\Bogot\'a, Colombia}
\date{}

\maketitle

\begin{abstract} We obtain a differential equation for the enumeration of the path length of general increasing trees. By using differential operators and their combinatorial interpretation we give a bijective proof of a version of Fa\`a di Bruno formula, and  model the generation of ballot and Dyck paths. We get formulas for its enumeration according with the height of their lattice points. Recursive formulas for the enumeration of enriched increasing trees and forests with respect to the height of their internal and external vertices are also obtained. Finally we present a generalized form of all those results using one-parameter groups in the general context of formal power series in an arbitrary number of variables. 
\end{abstract}

\section{Introduction and preliminaries} 
 Through this article a rooted tree will be thought of as a directed graph by orienting the edges toward the root. Given a rooted tree $T$ on the set of vertices $[n]$, denote by $\mathrm{d}_T(k)$, $k\in[n]$, the indegree  of vertex $k$.

Let $\phi(x)$ be formal power series $\phi(x)=\sum_{k=0}^{\infty}\phi_k\frac{x^k}{k!}$, with $\phi_0\neq 0$. We also denote  the $\mathrm{kth}$ coefficient $\phi_k$  by $\phi[k]$. Define the $\phi$-weight of a rooted tree $T$ on the set of vertices $[n]$ as the product
\begin{equation}\label{weight1}w_{\phi}(T)=\prod_{k=1}^n\phi[\mathrm{d}_T(k)].\end{equation}
Recall that an {\em increasing rooted tree} is one whose vertices (with labels in a totally ordered set)
increase along any path from the root to the leaves. 
When $\phi$ enumerates a family of structures, $w_{\phi}(T)$ enumerates all the trees obtained by enriching the vertices of $T$ with those structures.  For example, if we choose $\phi(t)=E(t)=e^t,$ since its coefficients as a exponential generating function are all equal to one, the corresponding trees are enriched with a trivial structure. They are called in the literature \emph{recursive trees}. If instead we choose $\phi(t)=\mathbb{L}(t)=\frac{1}{1-t}=\sum_{n=0}^{\infty}n!\frac{x^n}{n!}$, the exponential generating function of linear orders, it gives rise to the plane increasing trees. Sometimes they are also called plane recursive trees.  For $\phi(t)=1+t^2$ we get the complete plane binary trees. More generally,  $\phi(t)=1+t^r$ gives us the complete plane $r$-ary trees. The binary  plane trees (non necessarily complete) are obtained from $\phi(t)=(1+t)^2$. 
Increasing trees have a variety of applications. Recursive trees have been used in models of the spread of infections \cite{Moondistance}, pyramid schemes \cite{Gastwirth}, and as a simplified model for de world wide web. Plane recursive trees are a special instance of Albert-Barab\'asi free of scale networks model \cite {Barabasi}. Increasing binary trees are closely related to binary search trees and sorting algorithms.

Consider the autonomous differential equation
\begin{equation}\label{autonomous}
\begin{cases}y'=\phi(y)\\y(0)=0.\end{cases}
\end{equation}
It is well known that its solution $y(t)=\mathcal{A}_{\phi}^\uparrow(t)$ has the following interesting combinatorial interpretation in terms of increasing rooted trees. The coefficient $\mathcal{A}_{\phi}^\uparrow[n]$ of the formal power series 
$$\mathcal{A}_{\phi}^\uparrow(t)=\sum_{n=1}^{\infty}\mathcal{A}_{\phi}^\uparrow[n]\frac{t^n}{n!}$$ is the inventory of the set of $\phi$-enriched increasing rooted trees in $n$ vertices \cite{Boucher1, Leroux-Viennot, BerFlaSal, BLL}.
$$\mathcal{A}_{\phi}^\uparrow[n]=\sum_T w_{\phi}(T)$$
\begin{defi}\normalfont
The \emph{height} of a vertex in rooted tree is defined to be its  distance from the root, i.e., the number of edges in the unique path from the vertex to the root. The \emph{path length} of a tree is the sum of the heights of all its vertices.
\end{defi}
The height of vertices and the path length are very important statistics,  and their distribution in random trees have been studied by a number of authors (see for example \cite{Devroye1, Devroye2, Devroye3, Devroye4,Drmota, Kuba, Moon, Mahamoudlimit, Mahamoudsurvey, Munsonius, Neininger, Pittel} and references therein). The path length is of much importance for the analysis of algorithms since it is frequently related to the execution time \cite{Knut}. However, the solution of differential equations giving rise to the generating functions of families of increasing rooted trees  does not give information about those items. We found here a differential equation with an additional parameter $q$, whose generating function give us the number of trees classified with respect to their path length. The path length polynomial $p_n^{\phi}(q)$ so obtained is a $q$-anolog of the coefficient $\mathcal{A}_{\phi}^\uparrow[n]$.

A very simple formula for the chain rule for higher derivatives (Fa\`a di Bruno type formula) was found in \cite{stefa}. We present a  proof of it by using the species combinatorial interpretation of differential operators \cite{BLL,Joyal,Labelle,Labelle1,Leroux-Viennot}. In particular we use sums of operators of the form $x_{i+1}\partial_i$ (dart operators) acting on formal power series in a finite number of variables $x_0, x_1, x_2, \dots, x_n$. This lead us to go further and find many other applications to this kind of operators. In this way, by considering sum of dart operators $x_{i+1}\partial_i$ and of their adjoints $x_{i}\partial_{i+1}$ we show that they model the generation of upper-diagonal (ballot) paths, and the special case of Dyck paths. It allows us to get  generating functions that give information about the heights at different stages of the path, and to the associated generalized Bell polynomials, that count assemblies of upper diagonal lattice paths, according with their heights. Other versions of lattice paths can be modeled by similar operators.

By considering operators of the from $\phi(x_{i+1})\partial_i$  (Shift corolla operators)  we get recursive formulas for the enumeration of increasing trees and forests according with the height (depth) of their vertices. The framework we use here is equivalent to the formalism of weighted species over totally ordered sets. However for simplicity we avoid the use of categorical language, unnecessary in this context.

After obtaining the results of Sections \ref{sec:pathlength}, \ref{Sec:chain}, and \ref{Sec:shiftcorolla} we realized that the operators considered here are closely related the grammar approach of Chen \cite{Chen, Dumont}. However, they are not equivalent. Differential operators that are not derivations, i.e., those that involve higher order partial derivatives,  do not have free of context grammar counterpart. In Section \ref{chen} we translate, reinterpret and give new proofs in this context  of Chen interesting results about enumeration of rooted trees, and Fa\`a di Bruno grammar. Inspired on those, we give a new formula for the enumeration of Cayley's (non rooted) trees.  In Section \ref{sec:onepgroup} we present a generalized form of all those results using one-parameter groups in the general context of formal power series in an arbitrary number of variables. In this way we obtain a wide scope generalization of Chen's formula (\ref{Eq:cayley1}), that allows us to enumerate enriched rooted trees associated to differential operators (see Corollary \ref{Cor:generalchen}).

\begin{defi}\normalfont Let $F(t)\in\mathcal{R}[[t]]$ be an exponential formal power series
$$F(t):=\sum_{k=0}^{\infty} f_n\frac{t^n}{n!}$$
with coefficients in a ring $\mathcal{R}$ that contains $\mathbb{Q}$. 
For a finite totally ordered set of cardinal $n$, $V=\{v_1,v_2,\dots,v_n\}$, we define $F[V]$ to be the coefficient $f_{n}$ of the series, \begin{equation}\label{coeficiente}F[V]:=f_{|V|}.
\end{equation}
\end{defi}

\begin{rem}\normalfont   \label{rem:combinatorialconfig}We have to be careful, in the usual notation of species $F[V]$ denotes a set of combinatorial structures in a family $F$. Here, as a coefficient of the series $F(t)$, it is an element of the ring $\mathcal{R}$. This coefficient  is  thought of as the total weight (inventory) of a set of {\em combinatorial structures}, or combinatorial {\em configurations} belonging to a family $\mathcal{W}_{F}$ and having labels in the set $V$. More precisely, we assume that for every set $V$, there exists a weighted set $\mathcal{W}_F[V]$ such that $$F[V]=|\mathcal{W}_F[V]|_w=\sum_{\mathcal{O}\in \mathcal{W}_F[V]}w(\mathcal{O})=f_{|V|}.$$ There could possible be many families of weighted sets giving rise to different combinatorial interpretations to the same  series $F(t)$. But we are sure that there exists at least one, the trivial configurations, i.e., the family of unitary sets $$\mathcal{W}_F[V]=\{V\}$$ having as weight the coefficient of the series itself
$$w(V)=f_{|V|}.$$ More interesting configuration come into play  by defining them directly, or by the use of combinatorial operations on the trivial ones. We use G. 
Labelle corollas \cite{Labelle, BLL} to represent the coefficient
$F[V]$.
\end{rem}
The coefficients of the sum, product, derivative and substitution of  formal power series are given by the recipes
\begin{eqnarray}(F+G)[V] &=& F[V]+G[V]\\
\label{combinatorialproduct} (F.G)[V]&=&\sum_{V_1+V_2=V} F[V_1]G[V_2]\\
\label{combinatorialderivative}F^\prime[V]&=& F[\{\ast\}\uplus V]\\
\label{combinatorialsubstitution} F(G)[V]&=&\sum_{\pi\in\Pi[V]}F[\pi]\prod_{B\in\pi}G[B]
\end{eqnarray}
\noindent The sum in Eq. (\ref{combinatorialproduct}) is over all decompositions of $V$ into a pair of disjoint sets $(V_1,V_2)$, $V_1\uplus V_2=V$. In Eq. (\ref{combinatorialderivative}), $\ast$ is a ghost vertex that we add to $V$ as first element. In Eq. (\ref{combinatorialsubstitution}) the sum is over all the set partitions of the vertices $V$ and we assume that the series $G(t)$ has zero constant term. 

With the set theoretical notation for the coefficients of a formal power series we can rewrite the weight of a  $\phi$-enriched rooted tree, Eq. (\ref{weight1}). Recall we assume that the arcs of a rooted tree $T$ are oriented towards the root. In a rooted tree $T$ with vertices in $V$, denote by $T^{-1}(v)$ the (possible empty) set of immediate predecessors of $v$ (fiber of $v$).  The weight is now defined by
 \begin{equation}\label{weight2}
 w_{\phi}(T)=\prod_{v\in V}\phi[T^{-1}(v)]
 \end{equation}

\section{Path length generating function} \label{sec:pathlength}
 Let $T$ be an increasing tree with vertices labels in a totally ordered set set $V$. We  define the weight $\nu_{\phi}$, in order to include information about the height of their vertices,
$$\nu_{\phi}(T):=\prod_{v\in V}\phi[\mathrm{T^{-1}(v)}]q^{\mathrm{ht}(v)}=q^{\mathrm{pl}(T)}\prod_{v\in V}\phi[\mathrm{T^{-1}(v)}].$$
where $\mathrm{pl}(T)$ is the path length of $T$ and $\mathrm{ht}(v)$ is the height of vertex $v$ in $T$. 
We define  $P_{\phi}(t,q)$ as the exponential generating function of the inventory of the $\phi$-enriched trees according with their path length, which is a polynomial in $q$, 
$$p^{\phi}_n(q):=P_{\phi}[n]=\sum_T \nu_{\phi}(T).$$
 The sum over all increasing trees with vertices in $[n].$ The polynomial $p_n^{\phi}(q)$ is a $q$-analog of the inventory of $\phi$-enriched increasing trees $\mathcal{A}^{\uparrow}_{\phi}[n]$.

\begin{theo}
The generating function $P_{\phi}(t,q)$ of the $\phi$-enriched increasing trees with respect to the path length satisfies the differential equation 
\begin{equation}\label{autonomouspl}
\begin{cases} \frac{ \partial P_{\phi}(t,q)} {\partial t } =\phi(P_{\phi}(tq,q))\\ P_{\phi}(0,q)=0. \end{cases}
\end{equation}
\end{theo}
\begin{proof}
The coefficient of $\frac{\partial P_{\phi}}{\partial t}[V]$ on a totally ordered set $V$ is equal to 
$P_{\phi}[\{\ast\}+V].$ Since the configurations of   $\mathcal{W}_{P_{\phi}}[V]$ are increasing  $\phi$-enriched trees, for a tree $T$ of $\mathcal{W}_{P_{\phi}}[\{\ast\}+V]$ the ghost vertex takes the place of the root. The remaining vertices constitute a $\phi$-forest of weighted increasing trees. Their weight is equal to
$$\phi[\pi]\times\prod_{B\in\pi}\widehat{\nu}(T_B),$$  where $\widehat{\nu_{\phi}}(T_B)=q^{|B|}\nu_{\phi}(T_B)$, since the height of each vertex in $T_B$ is one less its original height on $T$. The total weight of the configurations in $\mathcal{W}_{P_{\phi}}[\{\ast\}+V]$ is  $$q^{|V|} \sum_{\pi\in\Pi[V]}\phi[\pi]\times\prod_{B\in \pi}P_{\phi}[B].$$
By Eq. (\ref{combinatorialsubstitution}), it is equal to $$q^{|V|}\phi(P_{\phi})[V],$$ and  we have
$$\frac{\partial P_{\phi}}{\partial t}(t,q)=\sum_{n=0}^{\infty}q^n\phi(P_{\phi})[n]\frac{t^n}{n!}=\sum_{n=0}^{\infty}\phi(P_{\phi})[n]\frac{(qt)^n}{n!}=\phi(P_{\phi}(tq,q)).$$ 
\end{proof}
As an immediate consequence we get
\begin{cor}\normalfont  
 Denote by  $$F_{\phi}(t,q)=\phi(P_{\phi}(t,q))=\sum_{n=0}^{\infty}f_n^{\phi}(q)\frac{t^n}{n!}$$   the generating function  of $\phi$-forests of $\phi$-enriched increasing trees enumerated according with their pathlength. We have
\begin{equation*}
F_{\phi}(t,q)=\frac{\partial P_{\phi}(u,q)}{\partial u}\left | _{u=\frac{t}{q}}\right . .
\end{equation*}
Equivalently, the respective coefficients satisfy the identity
\begin{equation*}
f_n^{\phi}(q)=q^{-n}p_{n+1}^{\phi }(q),\, n\geq 0.
\end{equation*}
\end{cor}

\begin{ex}\normalfont  
 The path length generating function of the recursive trees ($\phi=\exp$) satisfies the equation
\begin{equation}\label{pathlengthrecursive}
\frac{\partial P_{\exp}(t,q)}{\partial t}=e^{ P_{\exp}(tq,q)}.
\end{equation}

Taking derivatives in Eq. (\ref{pathlengthrecursive}) we obtain
\begin{equation*}
\frac{\partial^2 P_{\exp}(t,q)}{\partial t^2}=\frac{\partial P_{\exp}(t,q)}{\partial t}\frac{\partial P_{\exp}(u,q)}{\partial u} \left |_{u=tq}\right .\times q,
\end{equation*}
and from that, the recursive formula 
\begin{equation*}
\begin{cases}p_{n+1}^{\exp}(q)=\sum_{k=0}^{n-1}\binom{n-1}{k}p_{k+1}^{\exp}(q)p_{n-k}^{\exp}(q)q^{n-k}\\ \, p^{\exp}_1(q)=1,\, p^{\exp}_2(q)=q. \end{cases}
\end{equation*}
$F_{\exp}(t,q)=e^{ P_{\exp}(t,q)}$  is the path length generating function of forests of recursive trees. 
It satisfies the equation
\begin{equation*}
\frac{\partial F_{\exp}(t,q)}{\partial t}=e^{P_{\exp}(t,q)}\frac{\partial P_{\exp}(t,q)}{\partial t}=e^{P_{\exp}(t,q)}e^{P_{\exp}(tq,q)}=F_{\exp}(t,q)F_{\exp}(tq,q). 
\end{equation*} 
\end{ex}
\begin{ex}\normalfont   The path length generating function of the plane increasing trees ($\phi=\lin $) satisfies the differential equation \begin{equation}\label{pathlengthplannar}
\frac{\partial P_{\lin}(t,q)}{\partial t}=\frac{1}{1-P_{\lin}(tq,q)}
\end{equation}
From that we get
\begin{equation*}
\frac{\partial P_{\lin}(t,q)}{\partial t}=1+\frac{\partial P_{\lin}(t,q)}{\partial t}P_{\lin}(tq,q)
\end{equation*}
and we obtain the recursive formula
\begin{equation*}
\begin{cases} p_{n+1}^{\lin}(q)=\sum_{k=0}^{n-1}\binom{n}{k}p_{k+1}^{\lin}(q)p_{n-k}^{\lin}(q)q^{n-k}\\
p_{1}^{\lin}(q)=1.\end{cases}
\end{equation*}
$F_{\lin}(t,q)=\frac{1}{1- P_{\lin}(t,q)}$ is the path length generating function of the linearly ordered forests of increasing plane trees.  It clearly satisfies the differential equation
\begin{equation*}
\frac{\partial F_{\lin}(t,q)}{\partial t}=F^2_{\lin}(t,q)\frac{\partial P_{\lin}(t,q)}{\partial t}=F^2_{\lin}(t,q)F_{\lin}(tq,q).
\end{equation*}
\end{ex}
The polynomials $f^{\exp}_n(q)$ and $f^{\lin}_n(q)$  are respectively those in sequences  A126470 and  A232433 of the OEIS.  We have provided here  combinatorial interpretations for both of them.
\begin{enumerate}{}
\item The polynomial $f^{\exp}_n(q)$, sequence A126470, enumerates forests of recursive trees over $n$ vertices, according with their pathlength.
\item The polynomial $f^{\lin}_n(q)$, sequence A232433, enumerates  linearly ordered forests of plane increasing trees on $n$ vertices, according with their pathlength. 
\end{enumerate}
\begin{ex}\normalfont  
The plane binary increasing trees are obtained as the solution of Eq. (\ref{autonomous}) for $\phi(x)=(1+x)^2.$ More generally, the $r$-ary plane increasing trees are obtained from $\phi(x)=(1+x)^r=:\pr(x)$. Their path length generating function then satisfies\begin{equation*} \frac{\partial P_{\pr}(t,q)}{\partial t}=(1+P_{\pr}(tq,q))^r.
\end{equation*}
For the binary case we obtain
\begin{equation*}
\begin{cases}p_{n+1}^{\pb}(q)=q^n(2p_n^{\pb}(q)+\sum_{k=1}^{n-1}\binom{n}{k}p_k^{\pb}(q)p_{n-k}^{\pb}(q))\\
p_1^{\pb}(q)=1.
\end{cases}
\end{equation*}
%By using the following Mathematica program\\
%\\
%\verb|plbin[n_]:= q^(n - 1)*(2*plbin[n - 1] +|\\ 
%\verb|Sum[Binomial[n - 1, k]*plbin[k]*plbin[n - 1 - k], {k, 1, n - 2}])|\\
%\verb|plbin[1]:= 1|\\
%\verb|plbin[0]:=0|\\
%\verb|Sum[Expand[plbin[k]]*t^k/k!, {k, 1, 6}]|\\
\\With this recursive formula, we compute
\begin{eqnarray*}
P_{\pb}(t, q)&=& t+2q \frac{t^2}{2}+\left(4 q^3+2 q^2\right)\frac{ t^3}{6} + \left(8 q^6+4 q^5+12 q^4\right)\frac{t^4}{24}\\
  && \hspace{10pt} +\left(16 q^{10}+8 q^9+24 q^8+32 q^7+40 q^6\right) \frac{t^5}{120}\\&&\hspace{8pt}+ \left(32 q^{15}+16 q^{14}+48 q^{13}+64 q^{12}+160 q^{11}+40 q^{10}+280
   q^9+80 q^8\right)\frac{t^6}{720} +\dots 
\end{eqnarray*}
\end{ex}
\begin{ex}\normalfont  
The  plane (complete) binary trees are obtained by enriching with the function $\beta(t)=1+t^2$. We have $$ \frac{\partial P_{\beta}(t,q)}{\partial t}=1+P_{\beta}^2(tq,q).$$
From that we get the recursive formula for the path length polynomials
\begin{equation*}
\begin{cases} p_{n+1}^{\beta}(q)=q^n\sum_{k=1}^{n-1} \binom{n}{k}p_k^{\beta}(q)p_{n-k}^{\beta}(q)\\
p_1^{\beta}(q)=1.
\end{cases}
\end{equation*}
The function $\tan(t)$ is the solution of the original autonomous differential equation $y'(t)=1+y^2$, $y(0)=0$, which enumerates the alternating permutations of odd length \cite{Andre}. The polynomials $p_{n}^{\beta}(q)$ are a new kind of $q$-analogs  of the tangent numbers.
\end{ex}
\section{Chain rule for higher derivatives and enumeration of lattice paths}\label{Sec:chain}
Let $W$ and $Z$ be Banach spaces over $\mathbb{K}$, $\mathbb{K}$  being either $\mathbb{R}$ or $\mathbb{C}$. 
Assume that we have two functions $f:\Psi\rightarrow \Omega$ and $g:\Omega\rightarrow Z,$  $\Psi$ and $\Omega$ being open subsets of $\mathbb{K}$ and $W$  respectively. Assume also  that $f$ and $g$ have both $n$ continuous Fr\'echet derivatives, for some $n\in \mathbb{N}$. The following formula, alternative to the classical Fa\`a di Bruno's \cite{Faa}, was obtained in \cite{stefa}
\begin{equation}\label{HMY}
(f\circ g)^{(n)}(t)=\Delta_n f(g(t),g'(t),\dots,g^{(n)}(t)),
\end{equation}
\noindent where $\Delta_n$ is defined recursively as follows
\begin{eqnarray*}
\Delta_0f(x_0)&=&f(x_0)\\
\Delta_{j+1}f(x_0,x_1,\dots,x_{j+1})&=&\lim_{t\rightarrow 0}\frac{1}{t}[\Delta_j f(x_0+t x_1,x_1+tx_2,\dots, x_j+tx_{j+1} )].
\end{eqnarray*}
Observe that $\Delta_j f$ goes from $\Omega\times X^j$ to $Y$. The operator $\Delta_j$ is called the higher-order directional derivative. To the best of our knowledge the notion of  higher-order directional derivatives and Eq. (\ref{HMY}) together with  two other versions of it, were introduced in \cite{stefa}. For the classical version of Fa\`a di Bruno and historical remarks see for example  \cite{Comtet},\cite{Jonson}. For multivariate  generalizations see \cite{Const},\cite{Mishkov}). We will call Eq. (\ref{HMY}) the HMY formula. The HMY formula was required in order to study  the spectral Carath\'eodory-Fej\'er problem (see \cite{stefa1}). This problem consists in determining whether there exists an analytical matrix  function $F$ on the unit disc $\mathbb{D}$ with prescribed derivatives $F^{(j)}(0)$ $(1\leq j\leq n)$ and such that the eigenvalues of $F(\lambda)$ lie in $\mathbb{D}$ for all $\lambda\in\mathbb{D}$.  
In spite of its simplicity, the HMY formula is not well known to the combinatorialist audience. We first translate it into the context of formal power series. Let $\mathcal{R}$ be a ring that contains $\mathbb{Q}$. Define for each $n\geq 1$ the operator $\mathcal{D}_{n}:\mathcal{R}[[x_0,x_1,\dots,x_{n-1}]]\rightarrow\mathcal{R}[[x_0,x_1,\dots,x_{n}]]$, $\mathcal{D}_{n}=\sum_{j=0}^{n-1} x_{j+1} \partial_{j}$.  For example,  
\begin{eqnarray*}
\mathcal{D}_1&=&x_1\partial_0,\\
\mathcal{D}_2&=&x_2\partial_1+x_1\partial_0,\\
\mathcal{D}_3&=&x_3\partial_2+x_2\partial_1+x_1\partial_0.
\end{eqnarray*}
The reader may check that the correct translation of the operator $\Delta_n$ to this context is as follows
\begin{equation}
\Delta_n:=\mathcal{D}_n\mathcal{D}_{n-1}\dots\mathcal{D}_1.
\end{equation} 
Let $F(x)$ and $G(x)$ be two formal power series with $G(0)=0$. Observe that if we apply the operator $\Delta_n$ to the formal power series $F(x_0)$, we get a formal power series in the variables, $x_0, x_1,\dots, x_n$, $(\Delta_nF)(x_0,x_1,\dots,x_n)$.  HMY formula then becomes
\begin{equation}\label{delta}
F(G)^{(n)}(x)=(\Delta_nF)(G(x),G'(x),\dots,G^{(n)}(x)). 
\end{equation}
Closely related to this, Chen  \cite {Chen} defined a grammar to obtain  the nth. coefficient $z_n$ of the composition of two formal power series $P(Q(t))$, $P(t)=\sum_{k=0}^{\infty}y_k\frac{t^k}{k!}$, $Q(t)=\sum_{k=1}^{\infty}x_k\frac{t^k}{k!}$.  By translating that to the differential operators language we get that
\begin{equation}\label{chenfaa}z_n=\Gamma_n y_0,\end{equation}
\noindent where $$\Gamma_n:=\mathcal{F}_n \mathcal{F}_{n-1}\dots \mathcal{F}_1$$
$\mathcal{F}_n$ being the differential operator
$$\mathcal{F}_n=\mathcal{D}^x_n+x_1\mathcal{D}^y_n=\sum_{i=0}^{n-1}(x_{i+1} \partial_{x_i}+x_1y_{i+1} \partial_{y_i}).$$

We give a combinatorial (bijective) proof of (\ref{delta}) by reducing it to the classical Fa\`a di Bruno formula by means of a combinatorial device that we called \emph {dart operators}. The dart operator are a special type of corolla operators \cite{Leroux-Viennot}. They can be readily adapted to give a direct visual proof of (\ref{chenfaa}), as well as to several other identities in \cite{Chen} and in \cite{Dumont}. By using shift corolla operators,  we construct a family of generating functions that enumerates trees and forests of increasing  trees according with the height of their vertices. A special kind of them are the $r$-ary plane increasing trees. For $r=1$ we get back the classical Bell polynomials.\\
We begin by giving the combinatorial description of operations among  exponential generating function in a finite set of variables, following the colored species conventions as in \cite{Nava-yo}.
\subsection{The combinatorics of formal power series in many variables}\label{Sec:combinaformal}
We shall deal with exponential formal power series in a set of variables $x_0, x_1,x_2\dots, x_n$, over a ring $\mathcal{R}$ that contains $\mathbb{Q}$:
\begin{equation}\label{variasvariables}
    F(x_0,x_1,x_2,\dots,x_n)=\sum_{k_0,k_1,k_2\dots,k_n}f_{k_0,k_1,k_2,\dots,k_n}\frac{x_0^{k_0}x_1^{k_1}x_2^{k_2}\dots x_n^{k_n}}{k_0!k_1!k_2!\dots, k_n!}
\end{equation}
In order to write in a more compact way this kind of formal power series we shall use the following notation. Bold letters like $\kk$ will represent vectors of non-negative integers $(k_0, k_1, k_2,\dots, k_n)$. The vector of variables $(x_0,x_1,x_2,\dots,x_n)$ will be denoted by $\x$, and the monomial $x_0^{n_0}x_1^{n_1}x_2^{n_2}\dots x_n^{k_n}$ by $\x^{\kk}.$ The symbol $\kk!$ will denote the (finite) product of factorials $k_0!k_1!k_2!\dots k_n!$.

\noindent The series (\ref{variasvariables}) looks as follows:
\begin{equation}\label{compactgen}
    F(\x)=\sum_{\kk}f_{\kk}\frac{ \x^{\kk}}{\kk!}.
\end{equation}
A colored set, is a pair $(V,\kappa)$, where $V$ is a finite totally ordered set, and $\kappa$ is a function from $V$ to the set
  $\mathbb{N}$.
The {\em type or cardinal} of $(V,\kappa)$ is defined to be the vector of the cardinalities of the pre-images by $\kappa$ of each of the colors,
$$|(V,\kappa)|=(|\kappa^{-1}(0)|,|\kappa^{-1}(1)|,|\kappa^{-1}(2)|,\dots, |\kappa^{-1}(n)|).$$

\begin{defi}\normalfont For $F(\x)\in\mathcal{R}[[x_0,x_1,x_2,\dots,x_n]]$ we define $F[V,\kappa]$ to be the coefficient $f_{\kk}$, where $\kk$ is the type of $(V,\kappa)$, \begin{equation}\label{coeficiente}F[V,\kappa]:=f_{|(V,\kappa)|}.
\end{equation}
\end{defi}
Similarly as in Remark \ref{rem:combinatorialconfig}, the coefficient is interpreted as the weight of a family of combinatorial structures, or combinatorial configurations on the set of colored vertices $(V,\kappa)$.

It is trivial to check that
$$(F+G)[V,\kappa] = F[V,\kappa]+G[V,\kappa].$$
 At the set theoretical level, the coefficient of the series of the partial derivative is the coefficient of the original series on the colored set plus a `ghost' element colored with the color corresponding to the variable of the partial derivative

\begin{equation}
\nonumber \partial_iF[V,\kappa] = F[\{\ast\}+V,k^{+i}] 
 \end{equation}
 Here, $\{\ast\}+V$ is the totally ordered set obtained by adding $\ast$ as first element, and  $\kappa^{+i}$ is the extension of $\kappa$ that colors $\ast$ with the color $i$.
 the combinatorial formula for the coefficient of the product is given by
\begin{equation}
 \label{combinatorialproductc} (F.G)[V,\kappa]= \sum_{V_1+V_2=V} F[V_1,\kappa_{V_1}]G[V_2,\kappa_{V_2}],
\end{equation}
\noindent where the sum in Eq. (\ref{combinatorialproductc}) is over all decompositions of $V$ into a pair of disjoint sets $(V_1,V_2)$, $V_1\uplus V_2=V$, and $\kappa_{V_i}$ is the restriction of $\kappa$ to  $V_i$, $i=1,2.$
 \begin{figure}
\begin{center}
% Generated with LaTeXDraw 2.0.8
% Tue Oct 20 13:27:35 VET 2015
% \usepackage[usenames,dvipsnames]{pstricks}
% \usepackage{epsfig}
% \usepackage{pst-grad} % For gradients
% \usepackage{pst-plot} % For axes
\scalebox{1} % Change this value to rescale the drawing.
{\begin{pspicture}(0,-2.0200002)(11.2,2.14)
\definecolor{color2111557}{rgb}{0.6,0.0,0.6}
\pscircle[linewidth=0.04,dimen=outer,fillstyle=solid,fillcolor=yellow](0.33,-0.17000015){0.23}
\pscircle[linewidth=0.04,dimen=outer,fillstyle=solid,fillcolor=yellow](2.31,-0.13000014){0.23}
\psarc[linewidth=0.04](1.25,-1.2900001){0.85}{37.146687}{138.46822}
\psline[linewidth=0.04,fillstyle=solid,fillcolor=black](2.1809757,-0.30878055)(1.28,-1.3600001)
\psline[linewidth=0.04,fillstyle=solid,fillcolor=black](1.3,-0.03999995)(1.3,-1.32)
\psline[linewidth=0.04,fillstyle=solid,fillcolor=black](0.46,-0.32000005)(1.28,-1.34)
\pscircle[linewidth=0.04,dimen=outer,fillstyle=solid,fillcolor=yellow](1.31,0.16999982){0.23}
\usefont{T1}{ptm}{m}{n}
\rput(0.77140623,-0.20999995){$2$}
\usefont{T1}{ptm}{m}{n}
\rput(1.5914062,-0.14999995){$2$}
\usefont{T1}{ptm}{m}{n}
\rput(2.5514061,-0.46999994){$1$}
\psline[linewidth=0.04,fillstyle=solid,fillcolor=black](3.56,-0.43999994)(4.66,-1.4000001)
\psline[linewidth=0.04,fillstyle=solid,fillcolor=black](4.36,-0.06000005)(4.66,-1.4000001)
\psline[linewidth=0.04,fillstyle=solid,fillcolor=black](4.92,0.09999985)(4.66,-1.4199998)
\psline[linewidth=0.04,fillstyle=solid,fillcolor=black](5.68,-0.28000006)(4.68,-1.4000001)
\psdots[dotsize=0.28,dotstyle=asterisk](4.32,0.15999988)
\pscircle[linewidth=0.04,dimen=outer](4.32,0.16000001){0.24}
\pscircle[linewidth=0.04,dimen=outer,fillstyle=solid,fillcolor=yellow](3.47,-0.37000015){0.23}
\pscircle[linewidth=0.04,dimen=outer,fillstyle=solid,fillcolor=yellow](4.99,0.30999985){0.23}
\pscircle[linewidth=0.04,dimen=outer,fillstyle=solid,fillcolor=yellow](5.71,-0.17000015){0.23}
\pscircle[linewidth=0.04,dimen=outer,fillstyle=solid,fillcolor=yellow](6.15,-0.77000016){0.23}
\psline[linewidth=0.04,fillstyle=solid,fillcolor=black](4.66,-1.4199998)(6.0,-0.9)
\psarc[linewidth=0.04](4.73,-1.1700001){0.65}{-22.249012}{178.91907}
\usefont{T1}{ptm}{m}{n}
\rput(3.9314063,-0.42999995){$3$}
\usefont{T1}{ptm}{m}{n}
\rput(4.641406,-0.04999995){$i$}
\usefont{T1}{ptm}{m}{n}
\rput(6.0514064,-0.32999995){$2$}
\usefont{T1}{ptm}{m}{n}
\rput(6.411406,-1.1099999){$0$}
\usefont{T1}{ptm}{m}{n}
\rput(5.311406,0.01000005){$2$}
\psframe[linewidth=0.04,linecolor=color2111557,framearc=0.18213224,dimen=outer](6.68,1.4799999)(0.0,-2.0200002)
\psline[linewidth=0.04,linestyle=dashed,dash=0.16cm 0.16cm,fillstyle=solid,fillcolor=black](2.96,1.44)(2.96,-2.0)
\psframe[linewidth=0.04,linecolor=color2111557,framearc=0.18432,dimen=outer](11.2,2.14)(7.2,-2.0200002)
\psline[linewidth=0.04,fillstyle=solid,fillcolor=black](7.88,-0.6399999)(8.98,-1.6000001)
\psline[linewidth=0.04,fillstyle=solid,fillcolor=black](8.68,-0.26000005)(8.98,-1.6000001)
\psline[linewidth=0.04,fillstyle=solid,fillcolor=black](9.24,-0.10000015)(8.98,-1.6199999)
\psline[linewidth=0.04,fillstyle=solid,fillcolor=black](10.0,-0.48000005)(9.0,-1.6000001)
\psdots[dotsize=0.28,dotstyle=asterisk](8.64,-0.04000013)
\pscircle[linewidth=0.04,dimen=outer](8.64,-0.03999999){0.24}
\pscircle[linewidth=0.04,dimen=outer,fillstyle=solid,fillcolor=yellow](7.79,-0.5700002){0.23}
\pscircle[linewidth=0.04,dimen=outer,fillstyle=solid,fillcolor=yellow](9.31,0.10999985){0.23}
\pscircle[linewidth=0.04,dimen=outer,fillstyle=solid,fillcolor=yellow](10.03,-0.37000015){0.23}
\pscircle[linewidth=0.04,dimen=outer,fillstyle=solid,fillcolor=yellow](10.47,-0.97000015){0.23}
\psline[linewidth=0.04,fillstyle=solid,fillcolor=black](8.98,-1.6199999)(10.32,-1.0999999)
\psarc[linewidth=0.04](9.05,-1.37){0.65}{-22.249012}{178.91907}
\usefont{T1}{ptm}{m}{n}
\rput(8.251407,-0.62999994){$3$}
\usefont{T1}{ptm}{m}{n}
\rput(8.961407,-0.24999996){$i$}
\usefont{T1}{ptm}{m}{n}
\rput(10.371407,-0.53){$2$}
\usefont{T1}{ptm}{m}{n}
\rput(10.731406,-1.31){$0$}
\usefont{T1}{ptm}{m}{n}
\rput(9.631406,-0.18999995){$2$}
\pscircle[linewidth=0.04,dimen=outer,fillstyle=solid,fillcolor=yellow](7.67,1.3499999){0.23}
\pscircle[linewidth=0.04,dimen=outer,fillstyle=solid,fillcolor=yellow](9.65,1.3899999){0.23}
\psarc[linewidth=0.04](8.59,0.22999994){0.85}{37.146687}{138.46822}
\psline[linewidth=0.04,fillstyle=solid,fillcolor=black](9.520976,1.2112194)(8.62,0.15999985)
\psline[linewidth=0.04,fillstyle=solid,fillcolor=black](8.64,1.48)(8.64,0.19999994)
\psline[linewidth=0.04,fillstyle=solid,fillcolor=black](7.8,1.1999999)(8.62,0.17999995)
\pscircle[linewidth=0.04,dimen=outer,fillstyle=solid,fillcolor=yellow](8.65,1.6899998){0.23}
\usefont{T1}{ptm}{m}{n}
\rput(8.111406,1.3100001){$2$}
\usefont{T1}{ptm}{m}{n}
\rput(8.931406,1.37){$2$}
\usefont{T1}{ptm}{m}{n}
\rput(9.891406,1.0500001){$1$}
\usefont{T1}{ptm}{m}{n}
\rput(4.021406,-1.41){$F$}
\usefont{T1}{ptm}{m}{n}
\rput(0.51140624,-1.01){$\phi$}
\usefont{T1}{ptm}{m}{n}
\rput(7.751406,0.53000003){$\phi$}
\usefont{T1}{ptm}{m}{n}
\rput(8.341406,-1.55){$F$}
\end{pspicture}}
\end{center}
\caption{Corolla operator $\phi(\x)\partial_i$ applied to $F.$}\label{Fig:corola}
\end{figure}
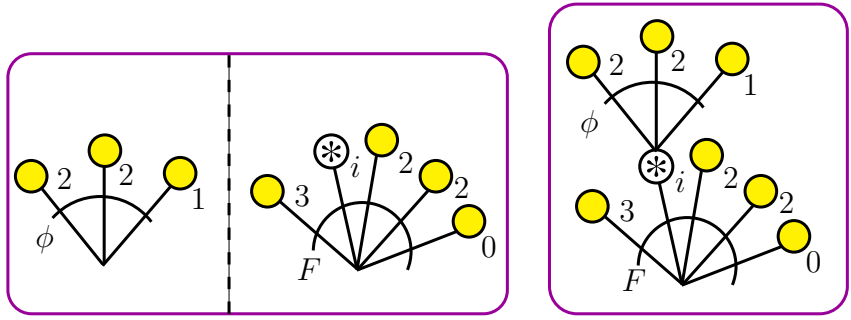
\subsection{The Fa\`a di Bruno formula}
Let $F(x)$ and $G(x)$ be two exponential formal power series $F(x)=\sum_{k=0}^{\infty}f_k
\frac{x^k}{k!}$ and $G(x)=\sum_{k=0}^{\infty}g_k \frac{x^k}{k!},$ with $g_0=0$.  Then,
the nth derivative of $F(G(x)) $ is given by  
\begin{equation}\label{faadibruno}
F(G)^{(n)}(x)=\sum_{k=1}^n F^{(k)}(G(x)).B_{n,k}(G'(x),G''(x),\dots, G^{(n)}(x)), 
\end{equation}
where $$B_{n,k}(x_1,x_2,\dots,x_n)=\sum_{k_1+k_2+\dots+k_n=k,k_1+2k_2+\dots+nk_n=n }\frac{n!}{1!^{k_1}k_1!2!^{k_2!}k_2!\dots}\x^\kk,$$ is the partial Bell polynomial. 
The type of a set partition $\pi$ is defined to be the tuple $\kk$, $k_i$ being the number blocks of size $i$ in $\pi$. Since $\frac{n!}{1!^{k_1}k_1!2!^{k_2!}k_2!\dots}$ counts the number of partitions of $[n]$ of type $\kk$, the Bell polynomials can be written in a set theoretical way as
$$B_{n,k}(\xx)=\sum_{\pi\in \Pi[n], |\pi|=k} \prod_{B\in\pi}x_{|B|}.$$
Expressing the Bell polynomials as an exponential power series,
 $$B_{n,k}(x_1,x_2,\dots,x_{n})=\sum_{\kk}\frac{n!}{1!^{k_1}k_1!2!^{k_2!}k_2!\dots}\kk!\frac{\x}{\kk!},$$
  the coefficient $B_{n,k}[V,\kappa]$ counts the pairs $(\pi,h)$, where $\pi$ is a partition of $[n]$ having exactly $k$ blocks, and $h$ is a bijection  from $\pi$ to $V$ such that the color of the image of each block equals its size: $$\kappa(h(B))=|B|,\; B\in \pi.$$
Fa\`a di Bruno formula can be rewritten  in a set theoretical way as 
\begin{equation}\label{faadibrunopart}
F(G)^{(n)}(x)=\sum_{\pi\in\Pi[n]} F^{(|\pi|)}(G(x))
\prod_{B\in\pi}G^{(|B|)}(x).
\end{equation}

\subsection{Corolla and dart operators} 

\begin{defi}\normalfont Let $\phi(\x)$ be a formal power series in a finite number of variables. An  operator of the form  $\phi(\x)\partial_i$, for some $i=0,1,\dots$ is called a {\em corolla operator}.  It 
acts on the coefficients of a formal power series $F$
as follows
\begin{equation}(\phi(\x)\cdot\partial_i F)[V,\kappa]=\sum_{V_1\uplus
V_2=V}\phi[V_1,\kappa_1]\times F[\{\ast\}\uplus V_2,\kappa_2^{+i}]\label{corolla-operator}.
\end{equation} \noindent

The configurations of $\phi(\x)\cdot\partial_i F$ are then pairs: the first a configuration of $\phi$, and the second a configuration $F$ over a colored set, the first of them a ghost element of color $i$. This can be represented in a more pictorial way as a $\phi$-enriched corolla standing over the ghost element $\ast$ (of color $i$) in the $F$-configuration (see Fig. \ref{Fig:corola}).

As a special case, the operator  $x_j.\partial_i$ is called the {\em dart operator} of type $(i,j)$. Since 
$$x_j[V_1,\kappa_1]= \begin{cases}   1 &  \mbox{if $(V_1,\kappa_1)$ is a singleton of color $j$ } \\  0 & \mbox{otherwise} \end{cases} ,$$
the structures of $(x_j\cdot\partial_i)F$ can be
represented  by drawing an edge (dart) from a singleton vertex 
(of color $j$) to the ghost element $\ast$ (of color $i$) in a configuration of $F$.
(see Fig. \ref{dart}). Such dart is said to be of type $(i,j)$.
\end{defi}

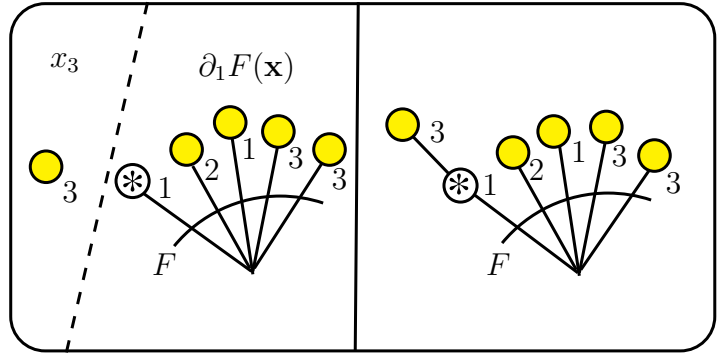
\begin{figure}
\begin{center}
% Generated with LaTeXDraw 2.0.8
% Mon Oct 19 21:03:39 VET 2015
% \usepackage[usenames,dvipsnames]{pstricks}
% \usepackage{epsfig}
% \usepackage{pst-grad} % For gradients
% \usepackage{pst-plot} % For axes
% User packages:
%\usepackage{amssymb, amsmath}

\scalebox{1} % Change this value to rescale the drawing.
{
\begin{pspicture}(0,-1.57)(9.4,3.13)
\psline[linewidth=0.04cm,fillcolor=black](7.299369,1.3093687)(7.5793686,-0.5306313)
\psline[linewidth=0.04cm,fillcolor=black](7.9393687,1.3293687)(7.5793686,-0.4906313)
\psline[linewidth=0.04cm,fillcolor=black](5.88,0.79)(5.34,1.35)
\psline[linewidth=0.04cm,fillcolor=black](6.18,0.55)(7.5593686,-0.4906313)
\rput{-270.0}(7.8074756,-5.588738){\pscircle[linewidth=0.04,dimen=outer,fillstyle=solid,fillcolor=yellow](6.698107,1.109369){0.23}}
\rput{-270.0}(8.627475,-5.8487377){\pscircle[linewidth=0.04,dimen=outer,fillstyle=solid,fillcolor=yellow](7.2381063,1.3893687){0.23}}
\rput{-270.0}(9.387475,-6.488737){\pscircle[linewidth=0.04,dimen=outer,fillstyle=solid,fillcolor=yellow](7.9381065,1.449369){0.23}}
\psline[linewidth=0.04cm,fillcolor=black](6.779369,0.90936875)(7.5393686,-0.4506312)
\psline[linewidth=0.04cm,fillcolor=black](8.519369,0.8893687)(7.5993686,-0.4506312)
\pscircle[linewidth=0.04,dimen=outer,fillstyle=solid,fillcolor=yellow](5.23,1.48){0.23}
\psframe[linewidth=0.04,framearc=0.19761379,dimen=outer](9.4,3.11)(0.0,-1.55)
\psline[linewidth=0.04cm,fillcolor=black](2.9593687,1.3293687)(3.2393687,-0.5106313)
\psline[linewidth=0.04cm,fillcolor=black](3.5993688,1.3493687)(3.2393687,-0.4706312)
\psline[linewidth=0.04cm,fillcolor=black](1.7793688,0.5693688)(3.2193687,-0.4706312)
\rput{-270.0}(4.967475,-2.1887374){\pscircle[linewidth=0.04,dimen=outer,fillstyle=solid,fillcolor=yellow](3.5781062,1.3893688){0.23}}
\rput{-270.0}(4.447475,-1.4287373){\pscircle[linewidth=0.04,dimen=outer,fillstyle=solid,fillcolor=yellow](2.9381063,1.5093689){0.23}}
\rput{-270.0}(5.4074755,-3.1087377){\pscircle[linewidth=0.04,dimen=outer,fillstyle=solid,fillcolor=yellow](4.2581067,1.149369){0.23}}
\rput{-270.0}(3.507475,-1.2087376){\pscircle[linewidth=0.04,dimen=outer,fillstyle=solid,fillcolor=yellow](2.3581061,1.1493686){0.23}}
\psline[linewidth=0.04cm,fillcolor=black](2.4393687,0.92936873)(3.1993687,-0.4306313)
\psline[linewidth=0.04cm,fillcolor=black](4.16,0.97)(3.2593687,-0.4306313)
\pscircle[linewidth=0.04,dimen=outer,fillstyle=solid,fillcolor=yellow](0.49,0.92){0.23}
\psline[linewidth=0.04cm,fillcolor=black](4.64,3.11)(4.6,-1.53)
\psline[linewidth=0.04cm,fillcolor=black,linestyle=dashed,dash=0.16cm 0.16cm](1.84,3.11)(0.76,-1.55)
\psdots[dotsize=0.28,dotstyle=asterisk](1.64,0.72999984)
\pscircle[linewidth=0.04,dimen=outer](1.64,0.72999996){0.24}
\pscircle[linewidth=0.04,dimen=outer,fillstyle=solid](6.0,0.66999996){0.24}
\psdots[dotsize=0.28,dotstyle=asterisk](5.98,0.6899998)
\rput{-270.0}(9.647474,-7.508736){\pscircle[linewidth=0.04,dimen=outer,fillstyle=solid,fillcolor=yellow](8.578105,1.069369){0.23}}
\psarc[linewidth=0.04](7.96,-1.25){1.82}{71.626175}{140.63731}
\psarc[linewidth=0.04](3.6,-1.29){1.82}{71.626175}{140.63731}
\usefont{T1}{ptm}{m}{n}
\rput(2.0614061,-0.38){$F$}
\usefont{T1}{ptm}{m}{n}
\rput(6.501406,-0.38){$F$}
\usefont{T1}{ptm}{m}{n}
\rput(0.81140625,0.58){$3$}
\usefont{T1}{ptm}{m}{n}
\rput(0.74140626,2.3){$x_3$}
\usefont{T1}{ptm}{m}{n}
\rput(3.1614063,2.24){$\partial_1 F(\mathbf{x})$}
\usefont{T1}{ptm}{m}{n}
\rput(2.0714064,0.64){$1$}
\usefont{T1}{ptm}{m}{n}
\rput(6.391406,0.68){$1$}
\usefont{T1}{ptm}{m}{n}
\rput(7.0114064,0.88){$2$}
\usefont{T1}{ptm}{m}{n}
\rput(3.1914062,1.16){$1$}
\usefont{T1}{ptm}{m}{n}
\rput(8.111406,1.06){$3$}
\usefont{T1}{ptm}{m}{n}
\rput(2.6914062,0.88){$2$}
\usefont{T1}{ptm}{m}{n}
\rput(7.5714064,1.12){$1$}
\usefont{T1}{ptm}{m}{n}
\rput(3.8114061,1.04){$3$}
\usefont{T1}{ptm}{m}{n}
\rput(4.391406,0.76){$3$}
\usefont{T1}{ptm}{m}{n}
\rput(8.831407,0.7){$3$}
\usefont{T1}{ptm}{m}{n}
\rput(5.6914062,1.38){$3$}
\end{pspicture} 
}

\end{center}
\caption{Structure of the dart operator $x_3\partial_1$ acting on the generic combinatorial family enumerated by 
$F(\x)$} \label{dart}
\end{figure}

For a formal power series on many variables, the configurations of $$\mathcal{D}_nF(\x)=\sum_{i=0}^{n-1}x_{i+1}\partial_iF(\x)$$
are configurations of $F(\x)$ with (at most) one dart of type $(i,i+1)$ for some $i$ from $0$ to
$n-1$.

\noindent We have the following

\begin{theo}\label{stefa} Let $F(x_0)$ be a formal power series in one variable.
Then we have the identity
\begin{equation*}
(\Delta_n F)(x_0,x_1,\dots,x_n)=\sum_{k=1}^n
F^{(k)}(x_0)B_{n,k}(x_1,x_2,\dots,x_n).
\end{equation*}
\end{theo}
\begin{proof} The operator $\mathcal{D}_1$ adds to the configurations of $F$ a dart of type
$(0,1)$. The operator $\mathcal{D}_2=x_2\partial_1+x_1\partial_0$ adds to the structures of
$\mathcal{D}_1 F$ or well a dart of type $(1,2)$, or one of type $(0,1)$ whose
ghost element takes the place of a label of the original structure in $F$. Since configurations of $F(x_0)$ 
has only vertices of color $0$, the ghost element of darts of type $(1,2)$ has to be on
the top of a dart of type $(0,1)$ already in the structure of $\mathcal{D}_1 F$, forming a
tower of darts of size two. Iterating this procedure it can be easily seen that the
structures of $\Delta_nF[V,\kappa]$ are configurations of $F$ with $k$ towers of darts, for some
$k$ between $1$ and $n$ ($k$-towered $F$-configurations, see Fig. \ref{Bell}.(a)). The
labels of the ghost elements in each tower increase from bottom to top. The number
of ghost elements between bottom and top of each tower is equal to the color of the label
on the top. The total number of ghost elements in each structure is equal to $n$. Let
$(V_2,\kappa_2)$ be the colored set given by the labels on the top of the towers. The towers
clearly define a partition on the set of ghost elements, plus a type-preserving bijection
from the partition to $(V_2,\kappa_2)$ (see Fig. \ref{Bell}.(b)). In other words, one of the structures counted by
$B_{n,k}[V_2,\kappa_2]$. Over the colored set of the remaining vertices
$(V_1,\kappa_1)$, $V_1=V-V_2$, $\kappa_1(V_1)=0$, we have a structure of the
$k^{\underline{\mathrm{th}}}$ derivative of $F$,
$F^{(k)}[V_1,\kappa_1]=\partial_0^kF(x_0)[V_1,\kappa_1]$. This correspondence between k-towered
$F$-structures associated to $\Delta_n F$ and configurations of $F^{(k)}(x_0)\cdot B_{n,k}(x_1,x_2,\dots,x_n)$
is clearly reversible. 
\end{proof}
The following corollary (HMY formula) follows directly from the previous theorem and the Fa\`a di Bruno
formula.
\begin{cor}\normalfont  Let $F$ and $G$ be two formal power series such that $G(0)=0$.
Then the  $n$-derivative of the formal power series $F(G(x))$ is given by the formula:
\begin{equation*} (F(G(x)))^{(n)}=\Delta_n F(G(x), G'(x),G''(x),G'''(x),\dots,G^{(n)}(x)).\end{equation*}
\end{cor}
\begin{cor}\normalfont  \label{Cor:functions}
The polynomials $\Delta_k x_0^n$ counts the functions $f:[k]\rightarrow [n]$ according with the sizes of their preimages. Setting the weight $w(f)$ as the product  $\prod_{i=1}^n x_{|f^{-1}(i)|}$, we have\begin{equation}\label{Eq:functions}
\Delta_k x_0^{n}=\sum_{f:[k]\rightarrow [n]}w(f).
\end{equation}

Equivalently, the coefficient $\Delta_k x_0^n[V,\kappa]$ is equal to the number of pairs $(f,h)$, where $f$ is as above, and $h$ is a bijection $h:[n]\rightarrow V$ such that the color of $h(j)$ equals the size of its preimage by $f$,
$$\kappa(h(j))=|f^{-1}(j)|.$$
\begin{proof}
By Theorem \ref{stefa} \begin{equation*} \Delta_k x_0^n=\sum_{j=1}^k(n)_j x_0^{n-j}B_{k,j}(x_1,x_2,\dots,x_k)=\sum_{j=1}^k (n)_j\sum_{\pi\in\Pi[k], |\pi|=j}x_0^{n-j}\prod_{B\in\pi}x_{|B|}.\end{equation*}
The monomial $x_0^{n-j}\prod_{B\in\pi}x_{|B|}$ is the weight of any function having as preimages the blocks of the partition $\pi$, $|\pi|=j$. There are $(n)_j$ of such  functions, which proves Eq. (\ref{Eq:functions}). The coefficient $\Delta_kx_0^n[V,\kappa]$  is $\kk!$ times the number of functions that have weight $\x^\kk$ ($\kk$ being the type of $(V,\kappa)$). The coefficient  $\kk!$ is exactly the number of  bijections $h$. 
\end{proof} 
\end{cor}

\begin{figure}\begin{center}\scalebox{0.8}{
\begin{pspicture}(0,-4.93)(15.22,4.93)
\definecolor{color16689b}{rgb}{1.0,1.0,0.2}
\definecolor{color16698}{rgb}{0.4,0.4,0.0}
\pscircle[linewidth=0.04,dimen=outer,fillstyle=solid,fillcolor=color16689b](13.4,0.59){0.3}
\pscircle[linewidth=0.04,dimen=outer,fillstyle=solid,fillcolor=color16689b](5.46,2.15){0.3}
\psline[linewidth=0.04cm](0.94,-1.11)(1.02,-2.11)
\pscircle[linewidth=0.04,dimen=outer,fillstyle=solid,fillcolor=color16689b](0.92,-0.89){0.3}
\psline[linewidth=0.04cm](2.22,-0.53)(2.22,-1.39)
\psline[linewidth=0.04cm](4.08,-0.51)(4.12,-1.27)
\pscircle[linewidth=0.04,linecolor=color16698,dimen=outer,fillstyle=solid](2.273281,-1.61){0.3}
\pscircle[linewidth=0.04,linecolor=color16698,dimen=outer,fillstyle=solid](1.0332813,-2.33){0.3}
\pscircle[linewidth=0.04,linecolor=color16698,dimen=outer,fillstyle=solid](4.14,-1.47){0.3}
\pscircle[linewidth=0.04,linecolor=color16698,dimen=outer](2.08,-3.39){0.3}
\pscircle[linewidth=0.04,linecolor=color16698,dimen=outer](3.34,-2.65){0.3}
\pscircle[linewidth=0.04,linecolor=color16698,dimen=outer](5.38,-2.57){0.3}
\pscircle[linewidth=0.04,linecolor=color16698,dimen=outer](4.1,-3.71){0.3}
\usefont{T1}{ptm}{m}{n}
\rput(5.3942184,-2.58){$\ast_1$}
\usefont{T1}{ptm}{m}{n}
\rput(2.2542186,-1.5903125){$\ast_2$}
\psline[linewidth=0.04cm](5.44,-1.05)(5.38,-2.29)
\psline[linewidth=0.04cm](5.48,0.41)(5.48,-0.63)
\psline[linewidth=0.04cm](5.44,1.87)(5.46,0.79)
\psline[linewidth=0.04cm](4.06,2.49)(4.08,1.41)
\psline[linewidth=0.04cm](4.08,0.97)(4.08,-0.11)
\psline[linewidth=0.04cm](1.24,-2.53)(1.84,-3.23)
\psline[linewidth=0.04cm](2.34,-3.45)(3.82,-3.71)
\psline[linewidth=0.04cm](3.48,-2.43)(4.02,-1.71)
\psline[linewidth=0.04cm](4.34,-1.65)(5.18,-2.37)
\psline[linewidth=0.04cm](1.34,-2.37)(3.06,-2.61)
\psline[linewidth=0.04cm](1.26,-2.13)(2.02,-1.73)
\psline[linewidth=0.04cm](2.56,-1.57)(3.84,-1.43)
\psline[linewidth=0.04cm](3.5,-2.89)(3.9,-3.49)
\psline[linewidth=0.04cm](4.36,-3.55)(5.24,-2.81)
\usefont{T1}{ptm}{m}{n}
\rput(13.380938,0.6){$c$}
\usefont{T1}{ptm}{m}{n}
\rput(3.2642186,-2.62){$e$}
\usefont{T1}{ptm}{m}{n}
\rput(2.0442185,-3.3954687){$f$}
\usefont{T1}{ptm}{m}{n}
\rput(4.0642185,-3.6903124){$g$}
\usefont{T1}{ptm}{m}{n}
\rput(0.8542187,-0.8954688){$h$}
\usefont{T1}{ptm}{m}{n}
\rput(1.0142188,-2.3103125){$\ast_4$}
\usefont{T1}{ptm}{m}{n}
\rput(4.1742187,-1.44){$\ast_5$}
\psline[linewidth=0.04cm](2.24,-0.13)(2.22,0.85)
\usefont{T1}{ptm}{m}{n}
\rput(10.834218,1.74){$\{\ast_5, \ast_7, \ast_9\}$}
\pscircle[linewidth=0.04,linecolor=color16698,dimen=outer](10.18,-1.57){0.3}
\pscircle[linewidth=0.04,linecolor=color16698,dimen=outer](8.95328,-2.29){0.3}
\pscircle[linewidth=0.04,linecolor=color16698,dimen=outer](12.06,-1.43){0.3}
\pscircle[linewidth=0.04,linecolor=color16698,dimen=outer](10.0,-3.35){0.3}
\pscircle[linewidth=0.04,linecolor=color16698,dimen=outer](11.26,-2.61){0.3}
\pscircle[linewidth=0.04,linecolor=color16698,dimen=outer](13.293281,-2.53){0.3}
\pscircle[linewidth=0.04,linecolor=color16698,dimen=outer](12.02,-3.67){0.3}
\usefont{T1}{ptm}{m}{n}
\rput(13.274219,-2.52){$\ast_1$}
\usefont{T1}{ptm}{m}{n}
\rput(10.15422,-1.5503125){$\ast_2$}
\psline[linewidth=0.04cm](9.16,-2.49)(9.76,-3.19)
\psline[linewidth=0.04cm](10.26,-3.39)(11.74,-3.65)
\psline[linewidth=0.04cm](11.4,-2.39)(11.94,-1.67)
\psline[linewidth=0.04cm](12.26,-1.61)(13.1,-2.33)
\psline[linewidth=0.04cm](9.26,-2.33)(10.98,-2.57)
\psline[linewidth=0.04cm](9.18,-2.09)(9.94,-1.69)
\psline[linewidth=0.04cm](10.48,-1.55)(11.76,-1.39)
\psline[linewidth=0.04cm](11.42,-2.85)(11.82,-3.45)
\psline[linewidth=0.04cm](12.28,-3.51)(13.16,-2.77)
\usefont{T1}{ptm}{m}{n}
\rput(11.240937,-2.6154687){$e$}
\usefont{T1}{ptm}{m}{n}
\rput(9.980938,-3.3554688){$f$}
\usefont{T1}{ptm}{m}{n}
\rput(12.000937,-3.6503124){$g$}
\usefont{T1}{ptm}{m}{n}
\rput(8.934218,-2.2703125){$\ast_4$}
\usefont{T1}{ptm}{m}{n}
\rput(12.034219,-1.4){$\ast_5$}
\usefont{T1}{ptm}{m}{n}
\rput(10.964219,2.7){$\{\ast_2, \ast_8\}$}
\usefont{T1}{ptm}{m}{n}
\rput(11.07422,3.72){$\{\ast_4\}$}
\usefont{T1}{ptm}{m}{n}
\rput(10.814219,0.6){$\{\ast_1, \ast_3, \ast_6\}$}
\psframe[linewidth=0.04,framearc=0.25,dimen=outer](15.22,4.93)(7.84,-4.89)
\psframe[linewidth=0.04,framearc=0.25,dimen=outer](7.38,4.89)(0.0,-4.93)
\psline[linewidth=0.04cm,linestyle=dashed,dash=0.16cm 0.16cm](7.84,-0.31)(15.18,-0.23)
\psline[linewidth=0.04cm,arrowsize=0.05291667cm 2.0,arrowlength=1.4,arrowinset=0.4]{->}(11.72,3.75)(12.74,3.77)
\psline[linewidth=0.04cm,arrowsize=0.05291667cm 2.0,arrowlength=1.4,arrowinset=0.4]{->}(11.88,2.67)(12.88,2.69)
\psline[linewidth=0.04cm,arrowsize=0.05291667cm 2.0,arrowlength=1.4,arrowinset=0.4]{->}(12.02,1.71)(13.04,1.71)
\psline[linewidth=0.04cm,arrowsize=0.05291667cm 2.0,arrowlength=1.4,arrowinset=0.4]{->}(11.92,0.57)(13.12,0.55)
\usefont{T1}{ptm}{m}{n}
\rput(9.404219,-0.9){$G_c^{(4)}(x_0)$}
\usefont{T1}{ptm}{m}{n}
\rput(8.764218,4.42){$B_{9,4}$}
\usefont{T1}{ptm}{m}{n}
\rput(1.3042186,3.9){$\Delta_9 G_c$}
\pscircle[linewidth=0.04,dimen=outer,fillstyle=solid,fillcolor=color16689b](13.04,3.77){0.3}
\usefont{T1}{ptm}{m}{n}
\rput(13.014218,3.7645311){$h$}
\pscircle[linewidth=0.04,dimen=outer,fillstyle=solid](2.2332811,-0.31){0.3}
\pscircle[linewidth=0.04,dimen=outer,fillstyle=solid,fillcolor=color16689b](2.24,1.09){0.3}
\pscircle[linewidth=0.04,dimen=outer,fillstyle=solid](4.06,-0.29){0.3}
\pscircle[linewidth=0.04,dimen=outer,fillstyle=solid](5.44,0.65){0.3}
\pscircle[linewidth=0.04,dimen=outer,fillstyle=solid](4.0932813,1.23){0.3}
\pscircle[linewidth=0.04,dimen=outer,fillstyle=solid](5.46,-0.81){0.3}
\pscircle[linewidth=0.04,dimen=outer,fillstyle=solid,fillcolor=color16689b](13.16,2.75){0.3}
\usefont{T1}{ptm}{m}{n}
\rput(13.11422,2.7445312){$d$}
\pscircle[linewidth=0.04,dimen=outer,fillstyle=solid,fillcolor=color16689b](13.34,1.75){0.3}
\usefont{T1}{ptm}{m}{n}
\rput(13.320937,1.7445313){$a$}
\pscircle[linewidth=0.04,dimen=outer,fillstyle=solid,fillcolor=color16689b](4.06,2.73){0.3}
\usefont{T1}{ptm}{m}{n}
\rput(4.0242186,2.7245312){$a$}
\usefont{T1}{ptm}{m}{n}
\rput(5.3842187,2.1445312){$c$}
\usefont{T1}{ptm}{m}{n}
\rput(1.0942189,-2.78){$0$}
\usefont{T1}{ptm}{m}{n}
\rput(2.4742188,-3.68){$0$}
\usefont{T1}{ptm}{m}{n}
\rput(4.554219,-3.88){$0$}
\usefont{T1}{ptm}{m}{n}
\rput(5.754219,-2.9){$0$}
\usefont{T1}{ptm}{m}{n}
\rput(3.8342187,-2.88){$0$}
\usefont{T1}{ptm}{m}{n}
\rput(4.3142185,-1.92){$0$}
\usefont{T1}{ptm}{m}{n}
\rput(2.6942186,-1.86){$0$}
\usefont{T1}{ptm}{m}{n}
\rput(1.2942188,-1.22){$1$}
\usefont{T1}{ptm}{m}{n}
\rput(2.5142188,-0.7){$1$}
\usefont{T1}{ptm}{m}{n}
\rput(4.354219,-0.64){$1$}
\usefont{T1}{ptm}{m}{n}
\rput(5.8142185,-1.14){$1$}
\usefont{T1}{ptm}{m}{n}
\rput(5.8742185,0.32){$2$}
\usefont{T1}{ptm}{m}{n}
\rput(4.454219,0.9){$2$}
\usefont{T1}{ptm}{m}{n}
\rput(2.6142187,0.76){$2$}
\usefont{T1}{ptm}{m}{n}
\rput(13.654219,2.66){$2$}
\usefont{T1}{ptm}{m}{n}
\rput(13.454219,3.64){$1$}
\usefont{T1}{ptm}{m}{n}
\rput(13.874219,1.6){$3$}
\usefont{T1}{ptm}{m}{n}
\rput(13.8542185,0.52){$3$}
\usefont{T1}{ptm}{m}{n}
\rput(4.4142184,2.4){$3$}
\usefont{T1}{ptm}{m}{n}
\rput(5.8542185,1.86){$3$}
\usefont{T1}{ptm}{m}{n}
\rput(8.694219,-2.72){$0$}
\usefont{T1}{ptm}{m}{n}
\rput(9.474218,-3.64){$0$}
\usefont{T1}{ptm}{m}{n}
\rput(12.07422,-4.2){$0$}
\usefont{T1}{ptm}{m}{n}
\rput(13.33422,-3.0){$0$}
\usefont{T1}{ptm}{m}{n}
\rput(12.114219,-1.9){$0$}
\usefont{T1}{ptm}{m}{n}
\rput(11.014218,-2.16){$0$}
\usefont{T1}{ptm}{m}{n}
\rput(10.274219,-2.06){$0$}
\usefont{T1}{ptm}{m}{n}
\rput(0.4754688,-4.28){(a)}
\usefont{T1}{ptm}{m}{n}
\rput(8.305469,-4.24){(b)}
\usefont{T1}{ptm}{m}{n}
\rput(4.0742188,1.2496876){$\ast_9$}
\usefont{T1}{ptm}{m}{n}
\rput(4.054219,-0.28){$\ast_7$}
\usefont{T1}{ptm}{m}{n}
\rput(5.454219,0.66){$\ast_6$}
\usefont{T1}{ptm}{m}{n}
\rput(5.454219,-0.82){$\ast_3$}
\usefont{T1}{ptm}{m}{n}
\rput(2.2142186,-0.2903125){$\ast_8$}
\usefont{T1}{ptm}{m}{n}
\rput(2.1742187,1.0845313){$d$}
\end{pspicture} }

\end{center}
\caption{A configuration  of 
$\Delta_9 G_c=\sum_{k=1}^9 G_c^{(k)}(x_0)B_{9,k}(\x)$, $G_c(x)$ being the 
series that  enumerates simple and connected graphs.}\label{Bell}
\end{figure}
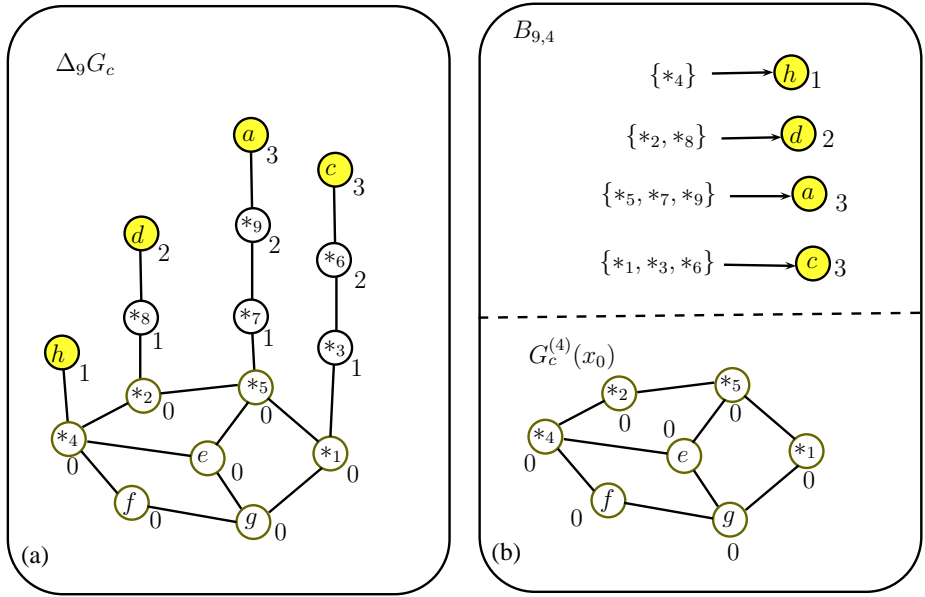
 \subsection{Ballot and Dyck paths}\label{Dyck}
 Suppose that we have an election with two candidates. Assume that during the counting of $n$ ballots one of the candidates always stays ahead of the other. Such kind of ballot counting will be called good. Formally a {\em good ballot counting} is a function $$g:V\rightarrow \{-1,1\}$$ such that for every $k$, $1\leq k\leq n$, $S_m=\sum_{j=1}^{m}g(v_j)\geq 0.$ A good ballot counting will be identified with the signed set \begin{equation}\label{signedset}g\equiv \{v_1^+,v_2^{g(v_2)},\dots,v_n^{g(v_n)} \}.\end{equation}
 The {\em ballot cardinal} of the signed set $\{v_1^+,v_2^{g(v_2)},\dots,v_n^{g(v_n)} \}$ is defined to be the sum $S_n$, \begin{equation}\label{ballotcardinal}
 |\{v_1^+,v_2^{g(v_2)},\dots,v_n^{g(v_n)} \}|_b=\sum_{j=1}^{n}g(v_j)\geq 0
 \end{equation}
 \begin{rem}\label{rem:kindsofpahs}\normalfont
 	 The partial sums $S_m$ form a one dimensional lattice walk that never crosses zero. There are two equivalent ways of representing these one dimensional lattice walks.
 
 \begin{enumerate}
 	\item \label{aboveaxis}By the  graphics $\{(m,S_m)|m=0, 1,\dots,n\}$, drew as a two dimensional lattice path with north-east, $(1,1)$, and south-east, $(1,-1)$, steps. It stays above the $x$-axis (see Fig. \ref{lattice1} (1)).
 	\item \label{abovediagonal} By codifying each positive ballot as a north step, $(0,1)$, and the negatives as an east step, $(1,0)$. This is the path associated to the graphic of the set $\{(\frac{m-S_m}{2},\frac{S_m+m}{2})|m=0,1,2,\dots,n\}$. These lattice paths stay above the diagonal. The partial sum $S_m$ is the height of the point $(\frac{m-S_m}{2},\frac{S_m+m}{2})$ with respect to the diagonal (see Fig. \ref{lattice1} (2)).
 \end{enumerate}
 If $S_n=0$, the resulting path is called a {\em Dyck path}.
\end{rem}
   \begin{figure}
   	\begin{center}% \usepackage[usenames,dvipsnames]{pstricks}
   		% \usepackage{epsfig}
   		% \usepackage{pst-grad} % For gradients
   		% \usepackage{pst-plot} % For axes
   		% \usepackage[space]{grffile} % For spaces in paths
   		% \usepackage{etoolbox} % For spaces in paths
   		% \makeatletter % For spaces in paths
   		% \patchcmd\Gread@eps{\@inputcheck#1 }{\@inputcheck"#1"\relax}{}{}
   		% \makeatother
   		% \psscalebox{1.0 1.0} % Change this value to rescale the drawing.
   		{
   			\begin{pspicture}(0,-2.8267713)(14.160842,2.8267713)
   			\definecolor{colour0}{rgb}{0.5019608,0.5019608,0.5019608}
   			\definecolor{colour2}{rgb}{0.6,0.0,0.6}
   			\definecolor{colour1}{rgb}{0.2,0.0,1.0}
   			\definecolor{colour3}{rgb}{1.0,0.6,0.0}
   			\rput(0.1516111,-1.2067713){\psgrid[gridwidth=0.028222222, subgridwidth=0.00412, gridlabels=5.4pt, subgriddiv=1, gridlabelcolor=white, gridcolor=colour2, subgridcolor=colour0](0,0)(8,0)(0,3)}
   			\psline[linecolor=colour1, linewidth=0.052, arrowsize=0.05291667cm 2.1,arrowlength=1.4,arrowinset=0.0]{->}(6.1516113,0.8132287)(7.1516113,1.7732288)
   			\psline[linecolor=colour1, linewidth=0.05, arrowsize=0.05291667cm 2.1,arrowlength=1.4,arrowinset=0.0]{->}(5.1716113,1.7732288)(6.1716113,0.7932287)
   			\psline[linecolor=colour1, linewidth=0.052, arrowsize=0.05291667cm 2.1,arrowlength=1.4,arrowinset=0.0]{->}(4.1516113,0.8132287)(5.1516113,1.7732288)
   			\psline[linecolor=colour1, linewidth=0.052, arrowsize=0.05291667cm 2.1,arrowlength=1.4,arrowinset=0.0]{->}(3.151611,-0.18677129)(4.1516113,0.7932287)(4.2116113,0.8532287)
   			\psline[linecolor=colour1, linewidth=0.052, arrowsize=0.05291667cm 2.1,arrowlength=1.4,arrowinset=0.0]{->}(2.171611,0.7732287)(3.151611,-0.20677128)
   			\psline[linecolor=colour1, linewidth=0.052, arrowsize=0.05291667cm 3.42,arrowlength=1.4,arrowinset=0.0]{->}(1.1516111,-0.20677128)(2.151611,0.7732287)
   			\psline[linecolor=colour1, linewidth=0.052, arrowsize=0.05291667cm 3.44,arrowlength=1.4,arrowinset=0.0]{->}(0.1516111,-1.1867713)(1.1516111,-0.20677128)
   			\rput(11.160842,-2.193438){\psgrid[gridwidth=0.028222222, subgridwidth=0.014111111, gridlabels=0.0pt, subgriddiv=1, gridcolor=colour2, subgridcolor=colour0](0,0)(-2,0)(3,5)}
   			\psline[linecolor=colour1, linewidth=0.052, arrowsize=0.05291667cm 2.1,arrowlength=1.4,arrowinset=0.0]{->}(9.164945,-2.1401045)(9.144944,-1.0934379)
   			\psline[linecolor=colour1, linewidth=0.052, arrowsize=0.05291667cm 2.0,arrowlength=1.4,arrowinset=0.0]{->}(9.185174,-0.1899111)(10.197686,-0.20010206)
   			\psline[linecolor=colour1, linewidth=0.052, arrowsize=0.05291667cm 2.0,arrowlength=1.4,arrowinset=0.0]{->}(9.1582775,-1.1288226)(9.1582775,-0.09497642)
   			\psline[linecolor=colour1, linewidth=0.052, arrowsize=0.05291667cm 2.0,arrowlength=1.4,arrowinset=0.0]{->}(10.182983,1.81166)(11.224944,1.7998954)
   			\psline[linecolor=colour1, linewidth=0.052, arrowsize=0.05291667cm 2.0,arrowlength=1.4,arrowinset=0.0]{->}(10.154265,0.8219467)(10.164462,1.8598953)
   			\psline[linecolor=colour1, linewidth=0.052, arrowsize=0.05291667cm 2.0,arrowlength=1.4,arrowinset=0.0]{->}(11.167508,1.8378441)(11.1695595,2.8232286)
   			\psline[linecolor=colour1, linewidth=0.052, arrowsize=0.05291667cm 2.0,arrowlength=1.4,arrowinset=0.0]{->}(11.161612,2.7947974)(12.219454,2.801464)
   			\psline[linecolor=colour1, linewidth=0.052, arrowsize=0.05291667cm 2.0,arrowlength=1.4,arrowinset=0.0]{->}(10.171855,-0.18908781)(10.160085,0.89144266)
   			\psline[linecolor=colour3, linewidth=0.04, linestyle=dashed, dash=0.17638889cm 0.10583334cm](9.171611,-2.173438)(14.151611,2.8132286)
   			\rput[bl](3.891611,-2.7667713){$(1)$}
   			\rput[bl](10.831611,-2.8267713){$(2)$}
   			\psline[linecolor=colour1, linewidth=0.048, arrowsize=0.05291667cm 2.1,arrowlength=1.4,arrowinset=0.0]{->}(7.1716113,1.7732288)(8.151611,0.8132287)
   			\end{pspicture}}
   		 \end{center}\caption{Equivalent lattice paths}\label{lattice1}
   \end{figure}
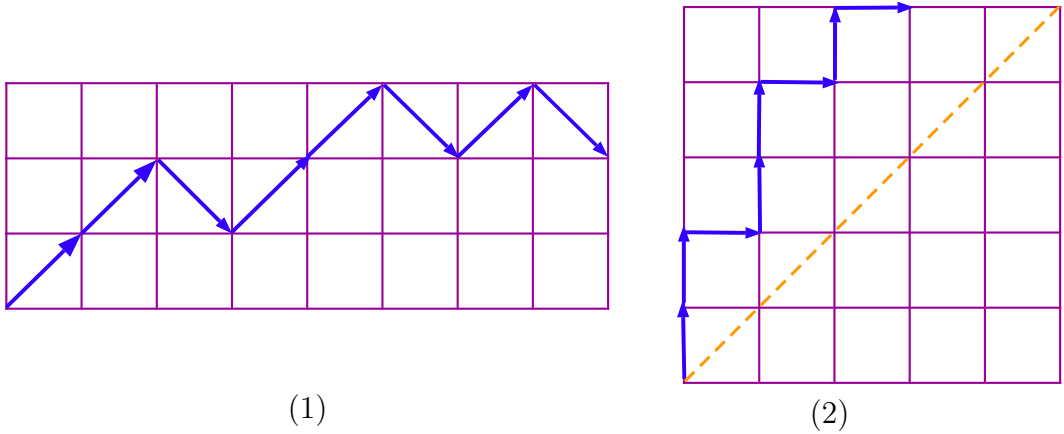
   Consider the adjoints of  dart operators $(x_{i+1}\partial_i)^*=x_i\partial_{i+1}$, $i=0,\dots,n-1$, and their sum
   $\mathcal{D}_n^*=\sum_{i=0}^{n-1}x_{i}\partial_{i+1}.$ The ballot operator $\mathcal{B}_n$  is defined as the sum
   \begin{equation*}
   \mathcal{B}_n=\mathcal{D}_n+\mathcal{D}_n^*,
   \end{equation*}
   The product
   \begin{equation}\label{lattice}
   \mathcal{P}_n=\mathcal{B}_n\mathcal{B}_{n-1}\dots \mathcal{B}_1,
   \end{equation}
   has very interesting properties.
   If the variable $x_{k+1}$ represents a vertex of height $k+1$ in a combinatorial structure, the adjoint (decreasing) dart operator $x_k\partial_{k+1}$ put a ghost vertex in its place and adds a vertex whose height is decreased by one. The operator $\mathcal{P}_n$ applied to $x_0$ may represent any of the two equivalent kinds of lattice paths in Remark \ref{rem:kindsofpahs}. We choose the above-diagonal representation, because it is more directly related to the $q$-analog of Catalan numbers in terms of Dyck paths. The indeterminate $x_k$ appearing at each stage of the walk gives us the height of the lattice point with respect to the diagonal $y=x$ (see Fig. \ref{lattice}). Specifically, $$\mathcal{P}_n x_0=b^{[0]}_n(x_0,x_1,\dots,x_n)=\sum_{k=0}^n b^{[0]}_{n,k} x_k,$$ where $b^{[0]}_{n,k}$ is the number of $n$-steps above-diagonal paths from $(0,0)$ and having final height $k=S_n$. Or equivalently, above-diagonal paths from $(0,0)$ to $(\frac{n-k}{2},\frac{n+k}{2})$. The coefficient of $x_0$, $C_n=b^{[0]}_{n,0}$ is the number of Dyck paths (Catalan number) \cite{Carlitz1}.

   In a similar way we get that 
   $$\mathcal{P}_n x_j=b^{[j]}_n(x_r,x_{r+1},\dots,x_{n+j})=\sum_{k=r}^{n+j} b^{[j]}_{n,k} x_k,\; r=\mathrm{min}\{0, n-j,\}$$ 
   where $b^{[j]}_{n,k}$ is the number of above-diagonal $n$-steps paths from $(0,j)$ and having final height $k$.  
   
   We now modify the lattice operators in order to keep track of the information about the height of some points of the path. We set
    $$\mathcal{D}_n^{q*}:=\sum_{i=0}^{n-1} q_ix_i\partial_{i+1},$$
    
    $$\mathcal{B}_n^q:=\mathcal{D}+\mathcal{D}_n^{q*},$$
    and 
    \begin{eqnarray}\label{qlattice1}\mathcal{P}^q_n&:=&\mathcal{B}_n^q\mathcal{B}_{n-1}^q\dots \mathcal{B}_1^q\\\label{qlattice2}
    b^{[j]}_n(\mathbf{x},\mathbf{q})&:=&\mathcal{P}^q_n x_j.
    \end{eqnarray}
    
    The parameters $q_i$ in the polynomials $b^{[j]}_n(\mathbf{x},\mathbf{q})$ weight the paths according with the heights of the points after horizontal steps. For example, the weight of the path in Fig. \ref{lattice} is equal to $x_2 q_1q_2^2$. Making the substitution $q_i\leftarrow q^i$ in $b^{[j]}_n(\mathbf{x},\mathbf{q})$ weights each path according with the area between the path and the diagonal. In particular,  the coefficient of $x_0$ in $b^{[0]}_n(\mathbf{x},\mathbf{q})$ after subindex rising in the $q$ parameters,  is the area $q$-analog of the Catalan number $C_n(q)$ \cite{Carlitz, Carlitz1}.
    From its definition in Eq. (\ref{qlattice2}), $b^{[j]}_n(\mathbf{x},\mathbf{q})$ satisfies the recursive formula:
    \begin{equation}
    b^{[j]}_n(\mathbf{x},\mathbf{q})=\mathcal{B}_n^q b^{[j]}_{n-1}(\mathbf{x},\mathbf{q})
    \end{equation}
    \subsubsection{Ballot Partitions}
    
    \begin{defi}\normalfont A ballot partition over a (totally ordered) set $V$, is an ordinary partition on $V$ together with a good ballot on each block of it. Equivalently, each block of the partition is a signed set as in Eq. (\ref{signedset}). 
    \end{defi}
    Enumerating the ballot partition according to the ballot cardinals, Eq. (\ref{ballotcardinal}), we obtain the following generalization of partial and total Bell polynomials:
    \begin{eqnarray*}
    	B_{n,k}^{\mathcal{B}}(\mathbf{x})&:=&\sum_{\pi\in\Pi[n], |\pi|=k}\prod_{B\in\pi} x_{|B|_b}\\
        Y_n^{\mathcal{B}}(\mathbf{x})&:=&\sum_{k=1}^{n}B_{n,k}^{\mathcal{B}}(\mathbf{x}).
    \end{eqnarray*}
    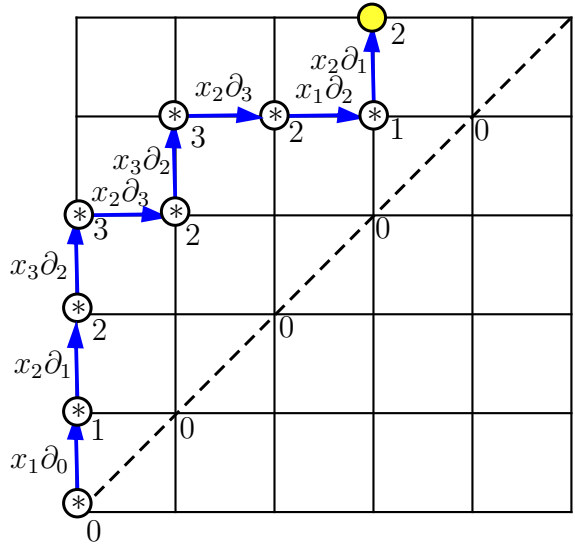
\begin{figure}
 	\begin{center}

\psscalebox{1.0 1.0} % Change this value to rescale the drawing.
{
	\begin{pspicture}(0,-3.579192)(7.466022,3.579192)
	\definecolor{colour0}{rgb}{0.5019608,0.5019608,0.5019608}
	\definecolor{colour2}{rgb}{1.0,1.0,0.2}
	\psline[linecolor=black, linewidth=0.04, linestyle=dashed, dash=0.17638889cm 0.10583334cm](0.9283983,-3.1832323)(7.448398,3.3634343)
	\rput(0.88693136,-3.1987784){\psgrid[gridwidth=0.028222222, subgridwidth=0.00412, gridlabels=0.0pt, subgriddiv=1, gridlabelcolor=magenta, unit=1.3158181cm, subgridcolor=colour0](0,0)(5,0)(0,5)
		\psset{unit=1.0cm}}
	\rput[bl](2.255065,-2.2080808){$0$}
	\rput[bl](4.855065,0.45191926){$0$}
	\rput[bl](3.575065,-0.86808074){$0$}
	\rput[bl](6.175065,1.7519194){$0$}
	\rput[b](1.2076576,0.429697){$3$}
	\rput[bl](0.004415584,-2.6771717){$x_1\partial_0$}
	\rput[bl](0.05220779,-1.4122367){$x_2\partial_1$}
	\rput[bl](0.0,-0.0935353){$x_3\partial_2$}
	\rput[bl](1.0823377,0.87789327){$x_2\partial_3$}
	\rput[bl](1.3657142,1.3025686){$x_3\partial_2$}
	\rput[b](5.164156,3.0319192){$2$}
	\rput[bl](2.458961,2.2726984){$x_2\partial_3$}
	\rput[bl](3.991794,2.6383262){$x_2\partial_1$}
	\rput[bl](3.7917652,2.248889){$x_1\partial_2$}
	\rput[bl](1.015065,-3.579192){$0$}
	\rput[b](1.1676576,-2.2458584){$1$}
	\rput[b](1.1928427,-0.90956223){$2$}
	\rput[b](2.4335835,0.35043776){$2$}
	\rput[b](3.806917,1.7282156){$2$}
	\rput[bl](2.4036362,1.6604908){$3$}
	\psline[linecolor=blue, linewidth=0.05, arrowsize=0.013cm 4.26,arrowlength=1.61,arrowinset=0.0]{->}(3.4331198,2.067818)(4.4025574,2.0812383)(4.7192335,2.042603)
	\pscircle[linecolor=black, linewidth=0.04, fillstyle=solid, dimen=outer](3.517655,2.079192){0.20277837}
	\rput[b](3.5361803,1.9865375){$\ast$}
	\psline[linecolor=blue, linewidth=0.05, arrowsize=0.013cm 4.26,arrowlength=1.61,arrowinset=0.0]{->}(4.842618,2.0959425)(4.817322,3.0940588)(4.8250575,3.2852705)
	\pscircle[linecolor=black, linewidth=0.04, fillstyle=solid, dimen=outer](4.8376546,2.079192){0.20277837}
	\rput[b](4.85618,1.9865375){$\ast$}
	\pscircle[linecolor=black, linewidth=0.04, fillstyle=solid,fillcolor=colour2, dimen=outer](4.815065,3.3791919){0.2}
	\psline[linecolor=blue, linewidth=0.05, arrowsize=0.013cm 4.26,arrowlength=1.61,arrowinset=0.0]{->}(0.9026179,-3.0640574)(0.8773216,-2.0659413)(0.8850577,-1.8747294)
	\pscircle[linecolor=black, linewidth=0.04, fillstyle=solid, dimen=outer](0.8976548,-3.0808082){0.20277837}
	\rput[b](0.9161803,-3.1734626){$\ast$}
	\psline[linecolor=blue, linewidth=0.05, arrowsize=0.013cm 4.26,arrowlength=1.61,arrowinset=0.0]{->}(0.9026179,-0.46405745)(0.8773216,0.5340587)(0.8850577,0.72527057)
	\pscircle[linecolor=black, linewidth=0.04, fillstyle=solid, dimen=outer](0.8976548,-0.48080814){0.20277837}
	\rput[b](0.9161803,-0.5734626){$\ast$}
	\psline[linecolor=blue, linewidth=0.05, arrowsize=0.013cm 4.26,arrowlength=1.61,arrowinset=0.0]{->}(0.83311975,0.7478179)(1.8025572,0.7612382)(2.1192336,0.7226029)
	\pscircle[linecolor=black, linewidth=0.04, fillstyle=solid, dimen=outer](0.9176548,0.7591919){0.20277837}
	\rput[b](0.93618035,0.6665374){$\ast$}
	\psline[linecolor=blue, linewidth=0.05, arrowsize=0.013cm 4.26,arrowlength=1.61,arrowinset=0.0]{->}(2.202618,0.8159425)(2.1773217,1.8140587)(2.1850576,2.0052705)
	\pscircle[linecolor=black, linewidth=0.04, fillstyle=solid, dimen=outer](2.1976547,0.79919183){0.20277837}
	\rput[b](2.2161803,0.7065374){$\ast$}
	\psline[linecolor=blue, linewidth=0.05, arrowsize=0.013cm 4.26,arrowlength=1.61,arrowinset=0.0]{->}(0.9026179,-1.8440574)(0.8773216,-0.8459413)(0.8850577,-0.6547294)
	\pscircle[linecolor=black, linewidth=0.04, fillstyle=solid, dimen=outer](0.8976548,-1.8608081){0.20277837}
	\rput[b](0.9161803,-1.9534626){$\ast$}
	\rput[b](5.146917,1.7682155){$1$}
	\psline[linecolor=blue, linewidth=0.05, arrowsize=0.013cm 4.26,arrowlength=1.61,arrowinset=0.0]{->}(2.0931199,2.067818)(3.0625572,2.0812383)(3.3792336,2.042603)
	\pscircle[linecolor=black, linewidth=0.04, fillstyle=solid, dimen=outer](2.1776547,2.079192){0.20277837}
	\rput[b](2.1961803,1.9865375){$\ast$}
	\end{pspicture}
}
\end{center}\caption{Representation of the action of $\mathcal{P}_8$ on $x_0.$}\label{lattice}
\end{figure}

Similarly as in Theorem \ref{stefa}, we get the following Theorem.
\begin{theo}
Let $F(x_0)$ be a series as in Theorem \ref{stefa}.
Then, the series $(\mathcal{P}_n F)(\mathbf{x})$ satisfy
the identity
\begin{equation}\label{ballotchain}
(\mathcal{P}_n F)(x_0,\dots, x_n)=\sum_{k=1}^n
F^{(k)}(x_0)B^{\mathcal{B}}_{n, k}(x_0, x_1,\dots, x_n).
\end{equation}
\end{theo}
\section{Shift corolla operators and increasing trees}\label{Sec:shiftcorolla} Let $\phi(x)$ be a formal power series with coefficients in $\mathbb{R}$.
Consider the autonomous differential equation with initial condition the indeterminate $x$, 

\begin{equation}\label{autonomous1}
\begin{cases}y'=\phi(y)\\y(0)=x.\end{cases}
\end{equation}
It is the same autonomous differential as in Eq. (\ref{autonomous}), but having as initial condition the variable $x$. The solution has a similar combinatorial  interpretation in terms of increasing trees. It is a formal power series in two variables $\mathcal{A}^{\uparrow}_{\phi}(t,x)$   whose coefficients count $\phi$-enriched increasing trees with two sorts of vertices. The variable $t$ corresponding to internal vertices  and the variable $x$ to the leaves \cite{Leroux-Viennot}. The internal vertices are weighted using the coefficients of the formal power series $\phi(x)$. If $\phi[0]\neq 0$, the internal vertices could have indegree zero. When $\phi[0]=0$, all the internal vertices has positive indegree, and in that case the trees are called \emph{extended} \cite{Knut}. Expressing $\mathcal{A}^{\uparrow}_{\phi}(t,x)$ as an exponential generating function in $t$ with coefficients in  $\mathbb{C}[[x]]$,
$$\mathcal{A}^{\uparrow}_{\phi}(t,x)=\sum_{n=1}^{\infty}\mathcal{T}_n^{\phi}(x)\frac{t^n}{n!},$$
the series $\mathcal{T}_n^{\phi}(x)$ is the exponential generating function of increasing $\phi$-enriched trees having $n$ internal vertices, enumerated according with their number of leaves. 
This generating function does not give information  regarding neither the height of the internal vertices nor of the leaves. We obtain that by means of shift corolla operators.

\begin{defi}\normalfont 
 An operator of the form $\phi(x_{i+1})\partial_i$ is called a \emph{shift corolla operator} of type $i+1$. We generalize the operators $\mathcal{D}_n$ and $\Delta_n$,  $$\mathcal{D}_n^{\phi}:\mathcal{R}[[x_0,x_1,\dots,x_{n-1}]]\rightarrow\mathcal{R}[[x_0,x_1,\dots,x_{n-1},x_n]],$$
\begin{eqnarray}\mathcal{D}_n^{\phi}&=&\sum_{k=0}^{n-1}\phi(x_{k+1})\partial_k\\\Delta_n^\phi&=&\mathcal{D}_n^{\phi}\mathcal{D}_{n-1}^{\phi}\dots \mathcal{D}_1^{\phi}.\end{eqnarray}
\end{defi}
\begin{prop}\normalfont  \label{tree}Let \begin{equation}\mathcal{T}_n^{\phi}(\x)=\Delta_n^\phi x_0.\end{equation}
The series $\mathcal{T}_n^{\phi}(\x)=\mathcal{T}_n^{\phi}(x_1, x_2,\dots,x_{n})$,  enumerates the $\phi$-enriched increasing trees with $n$ internal vertices by the height of their leaves; the coefficient $\mathcal {T}_n^\phi[\kk]$ counts the number of such trees having $k_i$ leaves of height $i$, $i=0,1,\dots,n-1.$
\end{prop}
\begin{proof}
At a combinatorial level, $\mathcal{D}^{\phi}_1x_0=\phi(x_{1})\partial_0$ acts over $x_0$  by grafting a corolla of type $1$ in a singleton vertex of color zero, the ghost root taking the place of that vertex. The color $1$ of the leaves indicates their height. The configurations of  $\mathcal{D}^{\phi}_2\mathcal{D}^{\phi}_1 x_0$ are obtained by grafting a corolla of type $2$ over a corolla of type $1$, the ghost vertex taking the place of a vertex of color $1$. Again the color of leaves indicates their height. The ghost vertices are enumerated to keep track of the order in which the operators were applied. By applying the operator  $\mathcal{D}^{\phi}_3$ we graft a either a corolla of type $1$ or a corolla of type $2$ over the tree in the same way as before. The internal ghost vertices increase along any path from the root. The general result follows by induction. \end{proof}
We modify corolla operators in order to keep information about the height of internal vertices,  multiplying it by the parameter $q_i$, $\phi(x_{i+1})q_i\partial_i$. In that way we define the operators
\begin{eqnarray}\mathcal{D}_{n}^{\phi,q}&=&\sum_{k=0}^{n-1}\phi(x_{k+1})q_k\partial_k\\\Delta_n^{\phi,q}&=&\mathcal{D}_n^{\phi}\mathcal{D}_{n-1}^{\phi,q}\dots \mathcal{D}_1^{\phi,q}.
\end{eqnarray}

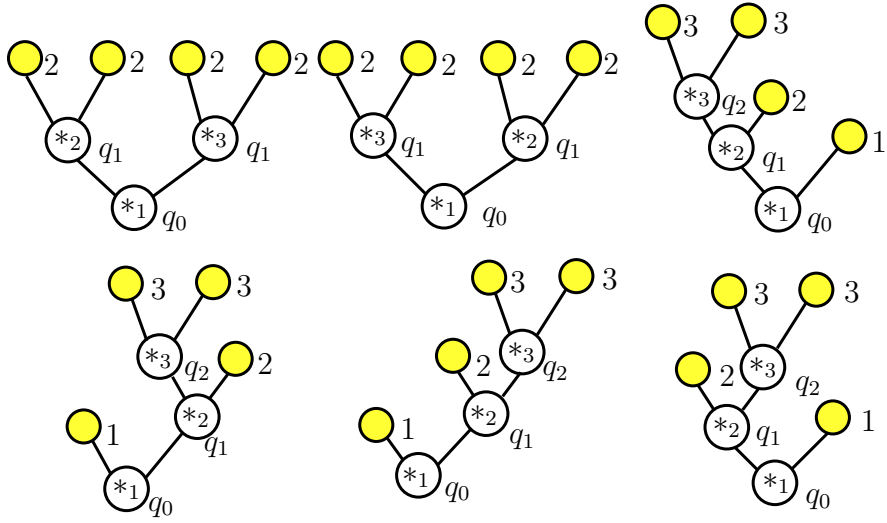
\begin{figure}\begin{center}% Generated with LaTeXDraw 2.0.8
% Thu Jul 02 08:22:18 VET 2015
% \usepackage[usenames,dvipsnames]{pstricks}
% \usepackage{epsfig}
% \usepackage{pst-grad} % For gradients
% \usepackage{pst-plot} % For axes
\scalebox{1} % Change this value to rescale the drawing.
{
\begin{pspicture}(0,-3.4889061)(11.903281,3.4489062)
\definecolor{color11815b}{rgb}{1.0,1.0,0.2}
\psline[linewidth=0.04cm](10.46,0.80890626)(11.08,1.5489062)
\psline[linewidth=0.04cm](6.71,1.8089063)(6.53,2.5089064)
\psline[linewidth=0.04cm](5.03,1.9089063)(5.4,2.5289063)
\psline[linewidth=0.04cm](1.44,-2.9310937)(1.12,-2.3510938)
\psline[linewidth=0.04cm](5.29,-2.6710937)(5.03,-2.3110936)
\psline[linewidth=0.04cm](5.7,-2.7310936)(6.2,-2.1910937)
\psline[linewidth=0.04cm](9.98,-2.7710938)(9.64,-2.431094)
\psline[linewidth=0.04cm](9.78,-1.9710938)(10.02,-1.6510937)
\pscircle[linewidth=0.04,dimen=outer,fillstyle=solid](10.05,-1.4010937){0.31}
\pscircle[linewidth=0.04,dimen=outer,fillstyle=solid](9.57,-2.2010937){0.31}
\psline[linewidth=0.04cm](10.4,-2.7910938)(10.94,-2.2510939)
\pscircle[linewidth=0.04,dimen=outer,fillstyle=solid](10.21,-2.9410937){0.31}
\pscircle[linewidth=0.04,dimen=outer,fillstyle=solid](6.85,-1.2010938){0.31}
\pscircle[linewidth=0.04,dimen=outer,fillstyle=solid](6.37,-2.0210938){0.31}
\pscircle[linewidth=0.04,dimen=outer,fillstyle=solid](5.47,-2.8410938){0.31}
\pscircle[linewidth=0.04,dimen=outer,fillstyle=solid](2.03,-1.2610937){0.31}
\pscircle[linewidth=0.04,dimen=outer,fillstyle=solid](2.53,-2.0610938){0.31}
\psline[linewidth=0.04cm](1.78,-2.9510937)(2.34,-2.2710938)
\pscircle[linewidth=0.04,dimen=outer,fillstyle=solid](1.59,-3.0210938){0.31}
\pscircle[linewidth=0.04,dimen=outer,fillstyle=solid,fillcolor=color11815b](1.34,2.7089062){0.24}
\pscircle[linewidth=0.04,dimen=outer,fillstyle=solid,fillcolor=color11815b](2.4,2.7089062){0.24}
\pscircle[linewidth=0.04,dimen=outer,fillstyle=solid,fillcolor=color11815b](3.54,2.7089062){0.24}
\pscircle[linewidth=0.04,dimen=outer,fillstyle=solid,fillcolor=color11815b](0.24,2.7089062){0.24}
\psline[linewidth=0.04cm](1.5,0.82890624)(0.92,1.3689063)
\psline[linewidth=0.04cm](1.94,0.86890626)(2.64,1.3889062)
\psline[linewidth=0.04cm](0.26,2.4689062)(0.66,1.7489063)
\psline[linewidth=0.04cm](0.96,1.8889064)(1.3,2.4889061)
\psline[linewidth=0.04cm](2.58,1.8489063)(2.4,2.5089064)
\psline[linewidth=0.04cm](2.96,1.7689062)(3.44,2.5289063)
\pscircle[linewidth=0.04,dimen=outer,fillstyle=solid,fillcolor=color11815b](5.47,2.7089062){0.24}
\pscircle[linewidth=0.04,dimen=outer,fillstyle=solid,fillcolor=color11815b](6.53,2.7089062){0.24}
\pscircle[linewidth=0.04,dimen=outer,fillstyle=solid,fillcolor=color11815b](7.67,2.7089062){0.24}
\pscircle[linewidth=0.04,dimen=outer,fillstyle=solid,fillcolor=color11815b](4.37,2.7089062){0.24}
\psline[linewidth=0.04cm](5.64,0.82890624)(4.98,1.4889063)
\psline[linewidth=0.04cm](6.04,0.86890626)(6.77,1.3689063)
\psline[linewidth=0.04cm](4.39,2.4689062)(4.69,1.9089063)
\psline[linewidth=0.04cm](7.13,1.8289063)(7.59,2.4889061)
\usefont{T1}{ptm}{m}{n}
\rput(4.7928123,2.6589065){$2$}
\usefont{T1}{ptm}{m}{n}
\rput(5.8728123,2.6189063){$2$}
\usefont{T1}{ptm}{m}{n}
\rput(6.932812,2.6189063){$2$}
\usefont{T1}{ptm}{m}{n}
\rput(8.032812,2.6389062){$2$}
\pscircle[linewidth=0.04,dimen=outer,fillstyle=solid,fillcolor=color11815b](11.2,1.6689063){0.24}
\pscircle[linewidth=0.04,dimen=outer,fillstyle=solid,fillcolor=color11815b](10.16,2.1889064){0.24}
\pscircle[linewidth=0.04,dimen=outer,fillstyle=solid,fillcolor=color11815b](9.86,3.2089062){0.24}
\pscircle[linewidth=0.04,dimen=outer,fillstyle=solid,fillcolor=color11815b](8.72,3.1889064){0.24}
\psline[linewidth=0.04cm](8.98,2.4289064)(8.78,2.9889061)
\psline[linewidth=0.04cm](9.34,2.4289064)(9.72,3.0289063)
\psline[linewidth=0.04cm](10.08,0.9089062)(9.74,1.2889062)
\psline[linewidth=0.04cm](9.4,1.6689063)(9.24,1.9489063)
\psline[linewidth=0.04cm](9.82,1.6689063)(10.06,1.9889063)
\usefont{T1}{ptm}{m}{n}
\rput(5.4428124,-2.8610938){$\ast_1$}
\usefont{T1}{ptm}{m}{n}
\rput(6.342812,-2.0410938){$\ast_2$}
\usefont{T1}{ptm}{m}{n}
\rput(6.8428125,-1.2010938){$\ast_3$}
\pscircle[linewidth=0.04,dimen=outer,fillstyle=solid,fillcolor=color11815b](7.57,-0.19109368){0.24}
\pscircle[linewidth=0.04,dimen=outer,fillstyle=solid,fillcolor=color11815b](6.41,-0.2110937){0.24}
\psline[linewidth=0.04cm](6.69,-0.97109365)(6.49,-0.4110937)
\psline[linewidth=0.04cm](7.05,-0.97109365)(7.43,-0.3710937)
\pscircle[linewidth=0.04,dimen=outer,fillstyle=solid,fillcolor=color11815b](5.93,-1.2510937){0.24}
\psline[linewidth=0.04cm](6.21,-1.7910937)(5.99,-1.4510937)
\psline[linewidth=0.04cm](6.59,-1.7910937)(6.83,-1.4710937)
\pscircle[linewidth=0.04,dimen=outer,fillstyle=solid,fillcolor=color11815b](4.91,-2.1710935){0.24}
\usefont{T1}{ptm}{m}{n}
\rput(6.332813,-1.3810937){$2$}
\usefont{T1}{ptm}{m}{n}
\rput(5.332813,-2.221094){$1$}
\usefont{T1}{ptm}{m}{n}
\rput(6.7928123,-0.2610937){$3$}
\usefont{T1}{ptm}{m}{n}
\rput(8.012813,-0.22109368){$3$}
\usefont{T1}{ptm}{m}{n}
\rput(1.5928124,-3.0210936){$\ast_1$}
\usefont{T1}{ptm}{m}{n}
\rput(2.5128126,-2.0610936){$\ast_2$}
\usefont{T1}{ptm}{m}{n}
\rput(2.0128124,-1.2810936){$\ast_3$}
\pscircle[linewidth=0.04,dimen=outer,fillstyle=solid,fillcolor=color11815b](3.04,-1.2910937){0.24}
\pscircle[linewidth=0.04,dimen=outer,fillstyle=solid,fillcolor=color11815b](2.74,-0.2710937){0.24}
\pscircle[linewidth=0.04,dimen=outer,fillstyle=solid,fillcolor=color11815b](1.58,-0.2910937){0.24}
\psline[linewidth=0.04cm](1.86,-1.0510937)(1.66,-0.4910937)
\psline[linewidth=0.04cm](2.22,-1.0510937)(2.6,-0.45109364)
\psline[linewidth=0.04cm](2.36,-1.8110937)(2.2,-1.5310937)
\psline[linewidth=0.04cm](2.7,-1.8110937)(2.94,-1.4910938)
\pscircle[linewidth=0.04,dimen=outer,fillstyle=solid,fillcolor=color11815b](1.02,-2.191094){0.24}
\usefont{T1}{ptm}{m}{n}
\rput(10.232813,-2.9410937){$\ast_1$}
\usefont{T1}{ptm}{m}{n}
\rput(9.572812,-2.2010937){$\ast_2$}
\usefont{T1}{ptm}{m}{n}
\rput(10.052813,-1.4010936){$\ast_3$}
\pscircle[linewidth=0.04,dimen=outer,fillstyle=solid,fillcolor=color11815b](10.76,-0.3710937){0.24}
\pscircle[linewidth=0.04,dimen=outer,fillstyle=solid,fillcolor=color11815b](9.6,-0.3910937){0.24}
\psline[linewidth=0.04cm](9.88,-1.1510937)(9.68,-0.5910938)
\psline[linewidth=0.04cm](10.24,-1.1510937)(10.62,-0.5510937)
\pscircle[linewidth=0.04,dimen=outer,fillstyle=solid,fillcolor=color11815b](9.12,-1.4310937){0.24}
\psline[linewidth=0.04cm](9.4,-1.9710938)(9.18,-1.6310937)
\pscircle[linewidth=0.04,dimen=outer,fillstyle=solid,fillcolor=color11815b](10.98,-2.0710936){0.24}
\usefont{T1}{ptm}{m}{n}
\rput(0.5928125,2.6189063){$2$}
\usefont{T1}{ptm}{m}{n}
\rput(1.7328124,2.6589065){$2$}
\usefont{T1}{ptm}{m}{n}
\rput(2.7928123,2.6589065){$2$}
\usefont{T1}{ptm}{m}{n}
\rput(3.9128125,2.6389062){$2$}
\usefont{T1}{ptm}{m}{n}
\rput(10.532812,2.1389062){$2$}
\usefont{T1}{ptm}{m}{n}
\rput(3.4328125,-1.3610936){$2$}
\usefont{T1}{ptm}{m}{n}
\rput(9.572812,-1.5210936){$2$}
\usefont{T1}{ptm}{m}{n}
\rput(1.4128125,-2.3010938){$1$}
\usefont{T1}{ptm}{m}{n}
\rput(11.452812,-2.1210938){$1$}
\usefont{T1}{ptm}{m}{n}
\rput(11.592813,1.5989063){$1$}
\usefont{T1}{ptm}{m}{n}
\rput(9.092813,3.1389062){$3$}
\usefont{T1}{ptm}{m}{n}
\rput(10.292813,3.1789062){$3$}
\usefont{T1}{ptm}{m}{n}
\rput(2.0128124,-0.3610937){$3$}
\usefont{T1}{ptm}{m}{n}
\rput(3.1528125,-0.3210937){$3$}
\usefont{T1}{ptm}{m}{n}
\rput(10.032812,-0.4210937){$3$}
\usefont{T1}{ptm}{m}{n}
\rput(11.232812,-0.4010937){$3$}
\pscircle[linewidth=0.04,dimen=outer,fillstyle=solid](10.25,0.69890624){0.31}
\pscircle[linewidth=0.04,dimen=outer,fillstyle=solid](6.89,1.6389062){0.31}
\pscircle[linewidth=0.04,dimen=outer,fillstyle=solid](5.81,0.73890626){0.31}
\pscircle[linewidth=0.04,dimen=outer,fillstyle=solid](9.63,1.5189062){0.31}
\usefont{T1}{ptm}{m}{n}
\rput(10.232813,0.6789063){$\ast_1$}
\usefont{T1}{ptm}{m}{n}
\rput(9.6328125,1.4789064){$\ast_2$}
\pscircle[linewidth=0.04,dimen=outer,fillstyle=solid](0.81,1.6189063){0.31}
\pscircle[linewidth=0.04,dimen=outer,fillstyle=solid](2.77,1.6389062){0.31}
\pscircle[linewidth=0.04,dimen=outer,fillstyle=solid](1.69,0.7189062){0.31}
\pscircle[linewidth=0.04,dimen=outer,fillstyle=solid](4.87,1.6789062){0.31}
\pscircle[linewidth=0.04,dimen=outer,fillstyle=solid](9.17,2.1989062){0.31}
\usefont{T1}{ptm}{m}{n}
\rput(0.8128125,1.6189063){$\ast_2$}
\usefont{T1}{ptm}{m}{n}
\rput(2.7528126,1.6589062){$\ast_3$}
\usefont{T1}{ptm}{m}{n}
\rput(1.6928126,0.71890634){$\ast_1$}
\usefont{T1}{ptm}{m}{n}
\rput(4.8628125,1.6789062){$\ast_3$}
\usefont{T1}{ptm}{m}{n}
\rput(5.8028126,0.7389063){$\ast_1$}
\usefont{T1}{ptm}{m}{n}
\rput(9.172812,2.1989064){$\ast_3$}
\usefont{T1}{ptm}{m}{n}
\rput(6.8828125,1.6389062){$\ast_2$}
\usefont{T1}{ptm}{m}{n}
\rput(10.791407,-3.1610937){$q_0$}
\usefont{T1}{ptm}{m}{n}
\rput(10.131406,-2.3210938){$q_1$}
\usefont{T1}{ptm}{m}{n}
\rput(10.651406,-1.6610937){$q_2$}
\usefont{T1}{ptm}{m}{n}
\rput(5.9914064,-3.0610938){$q_0$}
\usefont{T1}{ptm}{m}{n}
\rput(6.851406,-2.3610938){$q_1$}
\usefont{T1}{ptm}{m}{n}
\rput(7.291406,-1.4610938){$q_2$}
\usefont{T1}{ptm}{m}{n}
\rput(2.0514061,-3.2610939){$q_0$}
\usefont{T1}{ptm}{m}{n}
\rput(2.7914062,-2.4610937){$q_1$}
\usefont{T1}{ptm}{m}{n}
\rput(2.5314062,-1.4610938){$q_2$}
\usefont{T1}{ptm}{m}{n}
\rput(2.2314062,0.5389063){$q_0$}
\usefont{T1}{ptm}{m}{n}
\rput(1.3914063,1.4389062){$q_1$}
\usefont{T1}{ptm}{m}{n}
\rput(3.3514063,1.4589063){$q_1$}
\usefont{T1}{ptm}{m}{n}
\rput(6.4914064,0.5989063){$q_0$}
\usefont{T1}{ptm}{m}{n}
\rput(5.411406,1.5189062){$q_1$}
\usefont{T1}{ptm}{m}{n}
\rput(7.4714065,1.4989063){$q_1$}
\usefont{T1}{ptm}{m}{n}
\rput(10.811406,0.51890624){$q_0$}
\usefont{T1}{ptm}{m}{n}
\rput(10.211407,1.3189063){$q_1$}
\usefont{T1}{ptm}{m}{n}
\rput(9.671406,2.0789063){$q_2$}
\end{pspicture} }\end{center}\caption{Extended binary trees enumerated by the polynomial
$\mathcal{T}_{3}^{[2]}(x_1,x_2,x_3;q_1,q_2)=2x_2^4q_0q_1^2+4x_3^2 x_2 x_1q_0q_1q_2$.}\label{r-Bell}
\end{figure}

In a similar way as  in the proof of Proposition \ref{tree}  we get
\begin{prop}\normalfont  \label{tree1}Let \begin{equation}\mathcal{T}_n^{\phi}(\x,\mathbf{q})=\Delta_n^{\phi, q} x_0.\end{equation}
The formal power series $\mathcal{T}_n^{\phi}(\x,\mathbf{q})=\mathcal{T}_n^{\phi}(x_1, x_2,\dots,x_{n};q_0,q_1,\dots,q_{n-1})$,  enumerates the $\phi$-enriched increasing trees with $n$ internal vertices by the height of their leaves and internal vertices. The coefficient $$\mathcal {T}_n^{\phi}[V,\kappa](\mathbf{q}):=\mathcal {T}_n^{\phi}[V,\kappa](q_0,q_1,\dots,q_n),$$ a polynomial in $q_0,q_1,\dots,q_n$, counts the number of such trees having  leaves with labels in $V$, each leaf $v$ having height  $\kappa(v)$,  classified according with the height of their internal vertices. The coefficient of $q_0q_1^{j_1}\dots q_{n-1}^{j_{n-1}}$ of $\mathcal {T}_n^{\phi}[V,\kappa](q_0,q_1,\dots,q_n)$ gives us the number of those trees having  $j_r$ internal vertices of height $r$, $r=1,2,\dots,n-1$. 
\end{prop}

 From that we get the recursive formula
\begin{equation}
\mathcal{T}_n^{\phi}(\x,\mathbf{q})=\sum_{i=0}^{n-1}\phi(x_{n+1})q_i\partial_i \mathcal{T}_{n-1}^{\phi}(\x,\mathbf{q}).
\end{equation}
Le us denote by $B_{n,k}^{\phi}(\xx,\mathbf{q})$ the generating function of forests having  exactly $k$ increasing $\phi$-enriched trees as above, with $n$ internal vertices in total. The coefficient $B_{n,k}^{\phi}[V,\kappa](\mathbf{q})$ is a polynomial in $\mathbf{q}=(q_0,q_1,\dots,q_{n-1})$, whose coefficient $q_0^{k}q_1^{j_1}\dots q_{n-1}^{j_{n-1}}$ counts the number of such forest with $j_r$ internal vertices of height $r$, $r=0, 1,2,\dots,n-1.$ Again, the elements in the colored set $(V,\kappa)$ are labels for the leaves whose colors represent their height.

We have the following generalization of Theorem \ref{stefa}.
\begin{theo} \label{treechain}Let $F(x_0)$ be a series as in Theorem \ref{stefa}.
Then, the series $\Delta_n^{\phi,q} F(\xx;\mathbf{q})$ satisfy
the identity
\begin{equation}\label{gchain}
(\Delta_{n}^{\phi,q} F)(x_0,\dots, x_n; q_0,\dots, q_{n-1})=\sum_{k=1}^n
F^{(k)}(x_0)B_{n, k}^{\phi}(x_1,\dots, x_n; q_0,\dots, q_{n-1}).
\end{equation}
\end{theo}
\begin{proof} Completely similar to the proof of Theorem
\ref{stefa}.\end{proof}
The generating series $$Y_{n}^{\phi}(\xx,\mathbf{q})=\sum_{k=1}^nB_{n, k}^{\phi}(\xx; \mathbf{q})$$ counts forests of trees as above, with any number of trees from $1$ to $n$.

For $F(x_0)=e^{x_0}$, from Theorem \ref{treechain} we obtain the Rodrigues-like formula for $Y_n^{\phi}(\xx,\mathbf{q}).$
\begin{equation}\label{Bellgeneral}
Y_{n}^{\phi}(\xx,\mathbf{q})=e^{-x_0}(\Delta_{n}^{\phi,q}e^{x_0})(\xx,\mathbf{q}).
\end{equation} 
The following recursive formula can be obtained from
equation (\ref{Bellgeneral})  as it was done in \cite{stefa1} for
the ordinary Bell polynomials.
\begin{eqnarray}\nonumber Y_{1}^{\phi}(x_1, q_0)&=&\phi(x_1)q_0,\\
Y_{n}^{\phi}(\xx,\mathbf{q})&=&\phi(x_1)q_0Y_{n-1}^{\phi}(\xx,\mathbf{q})+\sum_{i=1}^{n-1}
\phi(x_{i+1})q_i\partial_{x_{i}}Y_{n-1}^{\phi}(\xx,\mathbf{q}). \end{eqnarray}

 However, a straightforward combinatorial proof can be given to the above recursive formula
 and to the following one, by using the combinatorial interpretations of $Y^{\phi}_n(\xx,\mathbf{q})$ and
 of $B_{n,k}^{\phi}(\xx,\mathbf{q})$ in terms of forests of increasing trees.
\begin{eqnarray}\nonumber B^{\phi}_{1,1}(x_1;q_0)&=&\phi(x_1)q_0,\\
B_{n,k}^{\phi}(\xx,\mathbf{q})&=&\phi(x_1)q_0
B_{n-1,k-1}^{\phi}(\xx,\mathbf{q})+\sum_{i=1}^{n-1}\phi(x_{i+1})
q_i\partial_{x_i}B_{n-1,k}^{\phi}(\xx,\mathbf{q}).
\end{eqnarray}
\begin{rem}\normalfont  \label{rem:inexpath}\begin{enumerate}
\item If $\phi[0]\neq 0$ the specializations $x_n\leftarrow 0$ and $q_n\leftarrow q^n$ in $\mathcal{T}_n^{\phi}(\xx,\mathbf{q})$ give us the polynomials $p_n^{\phi}(q)$ of Section \ref{sec:pathlength}.  
\item If $\phi[0]=0$, the $\phi$-enriched trees are called {\em extended}. If in addition $\phi(x)$ enumerates structures  without  symmetries except the identity (rigid structures), the substitution $x_n\leftarrow x^n$ and $q_n\leftarrow q^n$  gives us the series  $p_n^{\phi}(x,q)$, which is the ordinary generating function of increasing extended trees with unlabelled leaves, according to their internal and external pathlengths (sum of the heights of leaves \cite{Knut}).\end{enumerate}
\end{rem}

\begin{ex}\normalfont  Extended $r$-ary  plane trees.\\ Let $r$ be a positive integer. We consider here rooted trees where each internal node has exactly $r$ ordered children ($r$-ary plane trees).   In this case $\phi(x)=x^r$ and we are in the case 2 of Remark \ref{rem:inexpath}. We use the notation 
 $B_{n,k}^{[r]}(\xx)$ and $Y_{n}^{[r]}(\xx)$. The coefficient of $\xx^{\nn}/\nn!$ in $B_{n,k}^{[r]}(\xx)$ counts the number
forests with exactly $k$ trees, having $n$ internal vertices and $n_i$ leaves of depth
$i$. Since there are $n_i!$ ways of permuting the labels in the leaves of depth $i$,
obtaining a different labelled forest each time, the coefficient of  $\xx^n$ in
$B_{n,k}^{[r]}(\xx)$ counts the same kind of forests, but with unlabelled leaves (see
figure \ref{r-Bell}, and Table \ref{rarybell}). A similar interpretation is given to $Y_n^{[r]}(\xx).$
\end{ex}

%\begin{table}\begin{center}
%\begin{tabular}{|l|l|}\hline
%$\mathcal{T}_{2}^{[2]}(x_1,x_2)$&$2\,{x_{2}}^{2}\,{x_{1}}$\\\hline
%$\mathcal{T}_{3}^{[2]}(x_1,x_2,x_3)$& $2\,{x_{2}}^{4} +
%4\,{x_{3}}^{2}\,{x_{2}}\,{x_{1}}$\\\hline
%$\mathcal{T}_{4}^{[2]}(x_1,x_2,x_3,x_4)$&$12\,{x_{2}}^{3}\,{x_{3}}^{2} +
%4\,{x_{3}}^{4}\,{x_{1}} + 8\,{x_{4 }}^{2}\,{x_{3}}\,{x_{2}}\,{x_{1}},$\\\hline
%$\mathcal{T}_{5}^{[2]}(x_1,x_2,x_3,x_4,x_5)$&$40\,{x_{2}}^{2}\,{x_{3}}^{4} +
%32\,{x_{2}}^{3}\,{x_{4}}^{2}\,{x_{ 3}} + 24\,{x_{3}}^{3}\,{x_{4}}^{2}\,{x_{1}} +
%8\,{x_{4}}^{4}\,{x _{2}}\,{x_{1}} +$\\
%$\mbox{ }$&$16\,{x_{5}}^{2}\,{x_{4}}\,{x_{3}}\,{x_{2}}\,{x_{ 1}}$\\\hline
%$\mathcal{T}_{2}^{[3]}(x_1,x_2)$&$3\,{x_{2}}^{3}\,{x_{1}}^{2}$\\\hline
 %$\mathcal{T}_{3}^{[3]}(x_1,x_2,x_3)$&$6\,{x_{2}}^{6}\,{x_{1}} +
%9\,{x_{3}}^{3}\,{x_{2}}^{2}\,{x_{1}}^{2 }$\\\hline
%$\mathcal{T}_{4}^{[3]}(x_1,x_2,x_3,x_4)$&$6x_2^9+54x_2^5x_3^3x_1+
%18x_3^6x_2x_1^2+27x_4^3x_3^2x_2^2x_1^2$\\\hline
%$\mathcal{T}_{5}^{[3]}(x_1,x_2,x_3,x_4,x_5)$&$108\,{x_{2}}^{8}\,{x_{3}}^{3} +
%306\,{x_{2}}^{4}\,{x_{3}}^{6}\,{x _{1}} +
%216\,{x_{2}}^{5}\,{x_{4}}^{3}\,{x_{3}}^{2}\,{x_{1}} + 18 \,{x_{3}}^{9}\,{x_{1}}^{2} +$\\
%$\mbox{}$&$162\,{x_{3}}^{5}\,{x_{4}}^{3}\,{x_{2 }}\,{x_{1}}^{2}  +
%54\,{x_{4}}^{6}\,{x_{3}}\,{x_{2}}^{2}\,{x_{1}}^{2} + 81
%\,{x_{5}}^{3}\,{x_{4}}^{2}\,{x_{3}}^{2}\,{x_{2}}^{2}\,{x_{1}}^{2}$\\\hline
%\end{tabular}\end{center}
%\caption{The polynomials $\mathcal{T}^{[r]}_n(\xx)=B^{[r]}_{n,1}(\xx)$, for $n=2,3,4,5$
%and $r=2,3$.} \end{table}
\begin{table}[h]\begin{center}\begin{tabular}{|l|l|}\hline$B_{5,1}^{[2]}(\xx)$&$40\,{x_{2}}^{2}\,{x_{3}}^{4} +
32\,{x_{2}}^{3}\,{x_{4}}^{2}\,{x_{ 3}} + 24\,{x_{3}}^{3}\,{x_{4}}^{2}\,{x_{1}} +
8\,{x_{4}}^{4}\,{x _{2}}\,{x_{1}} +$\\
$\mbox{ }$&$16\,{x_{5}}^{2}\,{x_{4}}\,{x_{3}}\,{x_{2}}\,{x_{ 1}}$\\\hline
 $B_{5,2}^{[2]}(\xx)$&$20\,{x_{3}}^{4}\,{x_{1}}^{3} +
 40\,{x_{4}}^{2}\,{x_{3} }\,{x_{2}}\,{x_{1}}^{3} +
40\,{x_{2}}^{6}\,{x_{1}} + 140\,{x_{2}} ^{3}\,{x_{3}}^{2}\,{x_{1}}^{2}$\\\hline
$B_{5,3}^{[2]}(\xx)$&$80\, {x_{1}}^{4}\,{x_{2}}^{4} +
40\,{x_{1}}^{5}\,{x_{3}}^{2}\,{x_{2}}$
 \\\hline
$B_{5,4}^{[2]}(\xx)$& $20\,{x_{1}}^{7}\,{x_{2}}^{2}$\\\hline
$B_{5,5}^{[2]}(\xx)$&${x_{1}}^{10}$\\\hline\end{tabular}
\end{center}\label{rarybell}
\caption{Generalized partial Bell polynomials for $n=5$ and and $r=2$.}
\end{table}

\section{Enumeration of Cayley's trees}\label{chen}
For a (non rooted) tree $\mathfrak{a}$ on the set of vertices $[n]$ define its weight as follows $$w(\mathfrak{a})=\prod_{i=1}^{n}x_{d_{\mathfrak{a}}(i)},$$ where $d_{\mathfrak{a}}(i)$ is the degree in $\mathfrak{a}$ of the vertex $i$. The weight of a rooted tree $T$ is then defined as $$w(T)=\prod_{i=1}^n x_{\mathrm{d}_T(i)}.$$ Denote by $\mathfrak{A}_n$  (resp. $\mathcal{A}_n$) the set of trees (resp. rooted trees) on $n$ vertices. Recall that for rooted trees, $d_T(i)$ denotes the in-degree of vertex $i$, the edges oriented towards the root.

Define the generating functions $\mathfrak{A}_n(\x)=\sum_{\mathfrak{a}\in\mathfrak{A}_n}w(\mathfrak{a})$ and $\mathcal{A}_n(\x)=\sum_{T\in \mathcal{A}_n}w(T)$.

\begin{prop}\normalfont  \label{Prop:cayley}
	
	We have the following identities
	\begin{eqnarray}
	\mathcal{A}_n(x_0,x_1,\dots,x_{n-1})&=&\Delta_{n-1}x_0^n\label{Eq:cayley1}\\
	\mathfrak{A}_n(x_1,x_2,\dots,x_{n-1})&=&\Delta_{n-2} x_1^n\label{Eq:cayley2}.
	\end{eqnarray}
\end{prop}
\begin{proof}
	The  coefficient $\mathcal{A}_n[V,\kappa]$ is the number of pairs $(T,h)$, $T$ being a rooted tree on $\mathcal{A}_n$, $n=|V|$, and $h:[n]\rightarrow V$ a bijection assigning to each vertex of $T$ a label in $V$ whose color is equal to the in-degree of the vertex, $\kappa(h(i))=\mathrm{d}(i)$, $i=1,2, \dots,n.$ The structures enumerated by  $\mathcal{D}_n\mathcal{A}_n[V,\kappa]$ are colored labeled trees with a dart of type $(i,i+1)$ for some $i=0,1,\dots,n-1.$ Now we can use an argument similar to that in \cite{Joyal}  used to prove Cayley's theorem on the number of rooted trees. In a darted tree (Fig \ref{Fig:endo} (a)), replace the ghost vertex of color $i$ by the real vertex of color $i+1$, and mark it with an ingoing  arrow. Consider the subtree of $T$ whose root is that marked vertex. The color of marked vertex is equal to its in-degree plus one. The same property is satisfied by all the subtrees whose roots are on the path from the root of $T$ to the marked vertex.  Then we have a totally ordered set of such trees (Fig. \ref{Fig:endo} (b)). The total order can be replaced by a permutation, and from that we get an endofunction of the same weight, using the ingoing arrow to connect the roots of the trees in the same external cycle (Fig. \ref{Fig:endo} (c) and (d)). This construction is clearly reversible. By Corollary \ref{Cor:functions} the sum of the weights of the endofunctions is equal to $\Delta_n x_0^n$, and we have  $\mathcal{D}_n\mathcal{A}_n(\x)=\Delta_n x_0^n.$ Hence $\mathcal{D}_n(\mathcal{A}_n(\x)-\Delta_{n-1} x_0^n)=0$ and we get that $\mathcal{A}_n(\x)-\Delta_{n-1} x_0^n$ is a constant. It has to be zero, because both polynomials have zero constant term. This proves Eq. (\ref{Eq:cayley1}). To prove Eq. (\ref{Eq:cayley2}) observe that a dart operator acts in a tree by choosing a root and rising its degree by one. The rest of the vertices have colors that indicates their total degree, which is equal to their in-degree plus one. Then $(\mathcal{D}\mathfrak{A}_n)(x_1,x_2,\dots,x_n)$ is equal to the generating function of the rooted trees with its variables shifted by one, $S\mathcal{A}_n(x_0,x_1,\dots,x_{n-1})=\mathcal{A}(x_1,x_2,\dots,x_n)=S\Delta_{n-1}x_0^n=\Delta_{n-1}x_1^n$.   By a similar argument  as above $$\mathcal{D}_n\mathfrak{A}_n(x_1,x_2,\dots,x_n)= \Delta_{n-1}x_1^n\Rightarrow \mathfrak{A}(x_1,x_2,\dots,x_{n-1})=\Delta_{n-2}x_1^n.$$
\end{proof}
For example
$$\mathcal{A}(x_0,x_1,x_2,x_3,x_4)=\Delta_4x^5_0=5 x_4 x_0^4+60 x_2^2 x_0^3+80 x_1 x_3 x_0^3+360 x_1^2 x_2 x_0^2+120 x_1^4 x_0,$$ and 
$$\mathfrak{A}(x_1,x_2,x_3,x_4)=\Delta_{3}x_1^5=60 x_1^2x_2^3+60x_1^3x_2x_3+5x_1^ 4x_4.$$

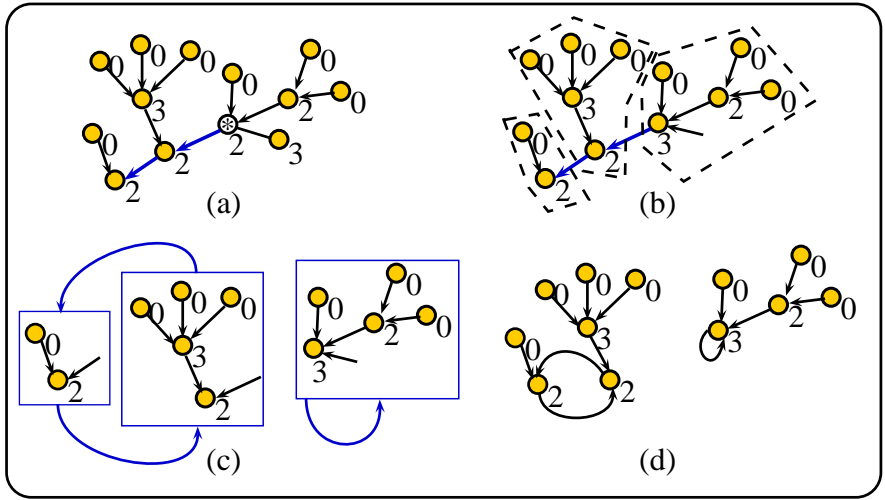
\begin{figure}
	\begin{center}
		% Generated with LaTeXDraw 2.0.8
		% Sun Jul 05 11:25:16 VET 2015
		% \usepackage[usenames,dvipsnames]{pstricks}
		% \usepackage{epsfig}
		% \usepackage{pst-grad} % For gradients
		% \usepackage{pst-plot} % For axes
		\scalebox{1} % Change this value to rescale the drawing.
		{
			\begin{pspicture}(0,-3.3)(11.66,3.3)
			\definecolor{color11247b}{rgb}{1.0,0.8,0.0}
			\definecolor{color11003}{rgb}{0.0,0.0,0.8}
			\psbezier[linewidth=0.036,arrowsize=0.05291667cm 2.0,arrowlength=1.4,arrowinset=0.4]{->}(7.1,-1.8690909)(7.1,-2.36)(8.08,-2.310909)(8.08,-1.82)
			\psline[linewidth=0.036cm,arrowsize=0.05291667cm 2.0,arrowlength=1.4,arrowinset=0.4]{->}(7.24,-0.58)(7.64,-1.02)
			\pscircle[linewidth=0.04,dimen=outer,fillstyle=solid,fillcolor=color11247b](7.74,-1.02){0.14}
			\usefont{T1}{ptm}{m}{n}
			\rput(7.967656,-1.17){3}
			\pscircle[linewidth=0.04,dimen=outer,fillstyle=solid,fillcolor=color11247b](8.04,-1.72){0.14}
			\usefont{T1}{ptm}{m}{n}
			\rput(8.278594,-1.87){2}
			\pscircle[linewidth=0.04,dimen=outer,fillstyle=solid,fillcolor=color11247b](7.72,-0.3){0.14}
			\usefont{T1}{ptm}{m}{n}
			\rput(7.9570312,-0.45){0}
			\pscircle[linewidth=0.04,dimen=outer,fillstyle=solid,fillcolor=color11247b](8.38,-0.38){0.14}
			\usefont{T1}{ptm}{m}{n}
			\rput(8.617031,-0.53){0}
			\psline[linewidth=0.036cm,arrowsize=0.05291667cm 2.0,arrowlength=1.4,arrowinset=0.4]{->}(8.3,-0.5)(7.84,-0.94)
			\psline[linewidth=0.036cm,arrowsize=0.05291667cm 2.0,arrowlength=1.4,arrowinset=0.4]{->}(7.74,-0.4)(7.74,-0.92)
			\psline[linewidth=0.036cm,arrowsize=0.05291667cm 2.0,arrowlength=1.4,arrowinset=0.4]{->}(7.78,-1.14)(8.04,-1.62)
			\pscircle[linewidth=0.04,dimen=outer,fillstyle=solid,fillcolor=color11247b](7.18,-0.52){0.14}
			\usefont{T1}{ptm}{m}{n}
			\rput(7.3970313,-0.59){0}
			\pscircle[linewidth=0.036,dimen=outer,fillstyle=solid,fillcolor=color11247b](7.08,-1.78){0.14}
			\usefont{T1}{ptm}{m}{n}
			\rput(7.318594,-1.93){2}
			\pscircle[linewidth=0.036,dimen=outer,fillstyle=solid,fillcolor=color11247b](6.78,-1.16){0.14}
			\usefont{T1}{ptm}{m}{n}
			\rput(7.017031,-1.31){0}
			\psline[linewidth=0.036cm,arrowsize=0.05291667cm 2.0,arrowlength=1.4,arrowinset=0.4]{->}(6.86,-1.26)(7.02,-1.72)
			\psbezier[linewidth=0.036,arrowsize=0.05291667cm 2.0,arrowlength=1.4,arrowinset=0.4]{->}(7.96,-1.6509091)(7.74,-1.3)(7.2323403,-1.2)(7.06,-1.7)
			\psline[linewidth=0.036cm,arrowsize=0.05291667cm 2.0,arrowlength=1.4,arrowinset=0.4]{->}(7.04,2.48)(7.44,2.04)
			\pscircle[linewidth=0.04,dimen=outer,fillstyle=solid,fillcolor=color11247b](8.68,1.7){0.14}
			\usefont{T1}{ptm}{m}{n}
			\rput(8.767656,1.43){3}
			\pscircle[linewidth=0.04,dimen=outer,fillstyle=solid,fillcolor=color11247b](7.54,2.04){0.14}
			\usefont{T1}{ptm}{m}{n}
			\rput(7.7676563,1.89){3}
			\pscircle[linewidth=0.04,dimen=outer,fillstyle=solid,fillcolor=color11247b](7.18,0.96){0.14}
			\usefont{T1}{ptm}{m}{n}
			\rput(7.418594,0.81){2}
			\pscircle[linewidth=0.04,dimen=outer,fillstyle=solid,fillcolor=color11247b](6.88,1.58){0.14}
			\usefont{T1}{ptm}{m}{n}
			\rput(7.117031,1.43){0}
			\pscircle[linewidth=0.04,dimen=outer,fillstyle=solid,fillcolor=color11247b](7.84,1.34){0.14}
			\usefont{T1}{ptm}{m}{n}
			\rput(8.078594,1.19){2}
			\pscircle[linewidth=0.04,dimen=outer,fillstyle=solid,fillcolor=color11247b](8.74,2.38){0.14}
			\usefont{T1}{ptm}{m}{n}
			\rput(8.977032,2.23){0}
			\pscircle[linewidth=0.04,dimen=outer,fillstyle=solid,fillcolor=color11247b](7.52,2.76){0.14}
			\usefont{T1}{ptm}{m}{n}
			\rput(7.7570314,2.61){0}
			\pscircle[linewidth=0.04,dimen=outer,fillstyle=solid,fillcolor=color11247b](9.78,2.7){0.14}
			\usefont{T1}{ptm}{m}{n}
			\rput(10.017032,2.55){0}
			\pscircle[linewidth=0.04,dimen=outer,fillstyle=solid,fillcolor=color11247b](8.18,2.68){0.14}
			\usefont{T1}{ptm}{m}{n}
			\rput(8.417031,2.53){0}
			\pscircle[linewidth=0.04,dimen=outer,fillstyle=solid,fillcolor=color11247b](10.18,2.14){0.14}
			\usefont{T1}{ptm}{m}{n}
			\rput(10.417031,1.99){0}
			\psline[linewidth=0.036cm,arrowsize=0.05291667cm 2.0,arrowlength=1.4,arrowinset=0.4]{->}(10.1,2.12)(9.62,2.06)
			\psline[linewidth=0.036cm,arrowsize=0.05291667cm 2.0,arrowlength=1.4,arrowinset=0.4]{->}(9.72,2.58)(9.58,2.16)
			\psline[linewidth=0.036cm,arrowsize=0.05291667cm 2.0,arrowlength=1.4,arrowinset=0.4]{->}(9.42,2.0)(8.8,1.72)
			\psline[linewidth=0.036cm,arrowsize=0.05291667cm 2.0,arrowlength=1.4,arrowinset=0.4]{->}(8.74,2.26)(8.72,1.8)
			\psline[linewidth=0.048cm,linecolor=color11003,arrowsize=0.05291667cm 2.0,arrowlength=1.4,arrowinset=0.4]{->}(8.58,1.62)(7.98,1.34)
			\psline[linewidth=0.036cm,arrowsize=0.05291667cm 2.0,arrowlength=1.4,arrowinset=0.4]{->}(8.1,2.56)(7.64,2.12)
			\psline[linewidth=0.036cm,arrowsize=0.05291667cm 2.0,arrowlength=1.4,arrowinset=0.4]{->}(7.54,2.66)(7.54,2.14)
			\psline[linewidth=0.036cm,arrowsize=0.05291667cm 2.0,arrowlength=1.4,arrowinset=0.4]{->}(7.58,1.9)(7.8,1.4)
			\psline[linewidth=0.036cm,arrowsize=0.05291667cm 2.0,arrowlength=1.4,arrowinset=0.4]{->}(6.96,1.48)(7.12,1.02)
			\psline[linewidth=0.048cm,linecolor=color11003,arrowsize=0.05291667cm 2.0,arrowlength=1.4,arrowinset=0.4]{->}(7.74,1.26)(7.3,0.96)
			\psline[linewidth=0.036cm,arrowsize=0.05291667cm 2.0,arrowlength=1.4,arrowinset=0.4]{->}(9.28,1.54)(8.78,1.66)
			\pscircle[linewidth=0.04,dimen=outer,fillstyle=solid,fillcolor=color11247b](6.98,2.54){0.14}
			\usefont{T1}{ptm}{m}{n}
			\rput(7.197031,2.47){0}
			\psline[linewidth=0.036,linestyle=dashed,dash=0.16cm 0.16cm](6.7,2.64)(7.56,3.1)(8.64,2.72)(8.26,1.94)(8.24,0.96)(7.68,1.06)(7.38,1.54)(6.7,2.68)
			\psline[linewidth=0.036,linestyle=dashed,dash=0.16cm 0.16cm](6.64,1.74)(7.08,1.82)(7.76,0.68)(7.22,0.52)(6.96,0.84)(6.62,1.76)
			\psline[linewidth=0.036,linestyle=dashed,dash=0.16cm 0.16cm](8.66,2.64)(8.46,2.18)(8.48,1.38)(9.02,0.92)(10.78,2.0)(9.92,2.98)(8.68,2.6)
			\pscircle[linewidth=0.04,dimen=outer,fillstyle=solid,fillcolor=color11247b](9.48,2.04){0.14}
			\usefont{T1}{ptm}{m}{n}
			\rput(9.718594,1.89){2}
			\usefont{T1}{ptm}{m}{n}
			\rput(8.675157,0.61){(b)}
			\pscircle[linewidth=0.04,dimen=outer,fillstyle=solid,fillcolor=color11247b](9.48,-1.06){0.14}
			\usefont{T1}{ptm}{m}{n}
			\rput(9.707656,-1.21){3}
			\pscircle[linewidth=0.04,dimen=outer,fillstyle=solid,fillcolor=color11247b](9.54,-0.38){0.14}
			\usefont{T1}{ptm}{m}{n}
			\rput(9.777031,-0.53){0}
			\pscircle[linewidth=0.04,dimen=outer,fillstyle=solid,fillcolor=color11247b](10.58,-0.06){0.14}
			\usefont{T1}{ptm}{m}{n}
			\rput(10.817031,-0.21){0}
			\pscircle[linewidth=0.04,dimen=outer,fillstyle=solid,fillcolor=color11247b](10.98,-0.62){0.14}
			\usefont{T1}{ptm}{m}{n}
			\rput(11.2170315,-0.77){0}
			\psline[linewidth=0.036cm,arrowsize=0.05291667cm 2.0,arrowlength=1.4,arrowinset=0.4]{->}(10.9,-0.64)(10.42,-0.7)
			\psline[linewidth=0.036cm,arrowsize=0.05291667cm 2.0,arrowlength=1.4,arrowinset=0.4]{->}(10.52,-0.18)(10.38,-0.6)
			\psline[linewidth=0.036cm,arrowsize=0.05291667cm 2.0,arrowlength=1.4,arrowinset=0.4]{->}(10.22,-0.76)(9.6,-1.04)
			\psline[linewidth=0.036cm,arrowsize=0.05291667cm 2.0,arrowlength=1.4,arrowinset=0.4]{->}(9.54,-0.5)(9.52,-0.96)
			\pscircle[linewidth=0.04,dimen=outer,fillstyle=solid,fillcolor=color11247b](10.28,-0.72){0.14}
			\usefont{T1}{ptm}{m}{n}
			\rput(10.518594,-0.87){2}
			\psbezier[linewidth=0.036,arrowsize=0.05291667cm 2.0,arrowlength=1.4,arrowinset=0.4]{->}(9.34,-1.06)(9.08,-1.3082956)(9.41913,-1.7)(9.56,-1.1610029)
			\usefont{T1}{ptm}{m}{n}
			\rput(8.675157,-2.81){(d)}
			\psline[linewidth=0.036cm,arrowsize=0.05291667cm 2.0,arrowlength=1.4,arrowinset=0.4]{->}(1.86,-0.82)(2.26,-1.26)
			\pscircle[linewidth=0.04,dimen=outer,fillstyle=solid,fillcolor=color11247b](2.36,-1.26){0.14}
			\usefont{T1}{ptm}{m}{n}
			\rput(2.5876563,-1.41){3}
			\pscircle[linewidth=0.04,dimen=outer,fillstyle=solid,fillcolor=color11247b](2.66,-1.96){0.14}
			\usefont{T1}{ptm}{m}{n}
			\rput(2.8985937,-2.11){2}
			\pscircle[linewidth=0.04,dimen=outer,fillstyle=solid,fillcolor=color11247b](2.34,-0.54){0.14}
			\usefont{T1}{ptm}{m}{n}
			\rput(2.5770311,-0.69){0}
			\pscircle[linewidth=0.04,dimen=outer,fillstyle=solid,fillcolor=color11247b](3.0,-0.62){0.14}
			\usefont{T1}{ptm}{m}{n}
			\rput(3.2370312,-0.77){0}
			\psline[linewidth=0.036cm,arrowsize=0.05291667cm 2.0,arrowlength=1.4,arrowinset=0.4]{->}(3.4,-1.68)(2.8,-1.96)
			\psline[linewidth=0.036cm,arrowsize=0.05291667cm 2.0,arrowlength=1.4,arrowinset=0.4]{->}(2.92,-0.74)(2.46,-1.18)
			\psline[linewidth=0.036cm,arrowsize=0.05291667cm 2.0,arrowlength=1.4,arrowinset=0.4]{->}(2.36,-0.64)(2.36,-1.16)
			\psline[linewidth=0.036cm,arrowsize=0.05291667cm 2.0,arrowlength=1.4,arrowinset=0.4]{->}(2.4,-1.4)(2.62,-1.9)
			\pscircle[linewidth=0.04,dimen=outer,fillstyle=solid,fillcolor=color11247b](1.8,-0.76){0.14}
			\usefont{T1}{ptm}{m}{n}
			\rput(2.0170312,-0.83){0}
			\pscircle[linewidth=0.04,dimen=outer,fillstyle=solid,fillcolor=color11247b](4.1,-1.3){0.14}
			\usefont{T1}{ptm}{m}{n}
			\rput(4.1876564,-1.61){3}
			\pscircle[linewidth=0.04,dimen=outer,fillstyle=solid,fillcolor=color11247b](4.16,-0.62){0.14}
			\usefont{T1}{ptm}{m}{n}
			\rput(4.3970313,-0.77){0}
			\pscircle[linewidth=0.04,dimen=outer,fillstyle=solid,fillcolor=color11247b](5.2,-0.3){0.14}
			\usefont{T1}{ptm}{m}{n}
			\rput(5.4370313,-0.45){0}
			\pscircle[linewidth=0.04,dimen=outer,fillstyle=solid,fillcolor=color11247b](5.6,-0.86){0.14}
			\usefont{T1}{ptm}{m}{n}
			\rput(5.8370314,-1.01){0}
			\psline[linewidth=0.036cm,arrowsize=0.05291667cm 2.0,arrowlength=1.4,arrowinset=0.4]{->}(5.52,-0.88)(5.04,-0.94)
			\psline[linewidth=0.036cm,arrowsize=0.05291667cm 2.0,arrowlength=1.4,arrowinset=0.4]{->}(5.14,-0.42)(5.0,-0.84)
			\psline[linewidth=0.036cm,arrowsize=0.05291667cm 2.0,arrowlength=1.4,arrowinset=0.4]{->}(4.84,-1.0)(4.22,-1.28)
			\psline[linewidth=0.036cm,arrowsize=0.05291667cm 2.0,arrowlength=1.4,arrowinset=0.4]{->}(4.16,-0.74)(4.14,-1.2)
			\psline[linewidth=0.036cm,arrowsize=0.05291667cm 2.0,arrowlength=1.4,arrowinset=0.4]{->}(4.68,-1.48)(4.2,-1.36)
			\pscircle[linewidth=0.04,dimen=outer,fillstyle=solid,fillcolor=color11247b](4.9,-0.96){0.14}
			\usefont{T1}{ptm}{m}{n}
			\rput(5.1385937,-1.11){2}
			\pscircle[linewidth=0.036,dimen=outer,fillstyle=solid,fillcolor=color11247b](0.7,-1.72){0.14}
			\usefont{T1}{ptm}{m}{n}
			\rput(0.93859375,-1.87){2}
			\pscircle[linewidth=0.036,dimen=outer,fillstyle=solid,fillcolor=color11247b](0.4,-1.1){0.14}
			\usefont{T1}{ptm}{m}{n}
			\rput(0.63703126,-1.25){0}
			\psline[linewidth=0.036cm,arrowsize=0.05291667cm 2.0,arrowlength=1.4,arrowinset=0.4]{->}(0.48,-1.2)(0.64,-1.66)
			\psline[linewidth=0.036cm,arrowsize=0.05291667cm 2.0,arrowlength=1.4,arrowinset=0.4]{->}(1.26,-1.42)(0.82,-1.72)
			\psframe[linewidth=0.018,linecolor=color11003,dimen=outer](1.4,-0.8)(0.18,-2.06)
			\psframe[linewidth=0.02,linecolor=color11003,dimen=outer](3.44,-0.28)(1.54,-2.34)
			\psframe[linewidth=0.02,linecolor=color11003,dimen=outer](6.04,-0.12)(3.86,-1.98)
			\psbezier[linewidth=0.036,linecolor=color11003,arrowsize=0.05291667cm 2.0,arrowlength=1.4,arrowinset=0.4]{->}(4.0,-1.96)(4.0,-2.76)(5.0,-2.76)(4.98,-1.98)
			\psbezier[linewidth=0.036,linecolor=color11003,arrowsize=0.05291667cm 2.0,arrowlength=1.4,arrowinset=0.4]{->}(0.7,-2.04)(0.7,-2.84)(2.56,-3.14)(2.56,-2.34)
			\psbezier[linewidth=0.036,linecolor=color11003,arrowsize=0.05291667cm 2.0,arrowlength=1.4,arrowinset=0.4]{<-}(0.74,-0.84)(0.74,-0.04)(2.54,0.5)(2.54,-0.3)
			\usefont{T1}{ptm}{m}{n}
			\rput(2.9051561,-2.81){(c)}
			\psline[linewidth=0.036cm,arrowsize=0.05291667cm 2.0,arrowlength=1.4,arrowinset=0.4]{->}(1.32,2.46)(1.72,2.02)
			\pscircle[linewidth=0.04,dimen=outer,fillstyle=solid](2.96,1.68){0.14}
			\usefont{T1}{ptm}{m}{n}
			\rput(3.0785937,1.41){2}
			\pscircle[linewidth=0.04,dimen=outer,fillstyle=solid,fillcolor=color11247b](1.82,2.02){0.14}
			\usefont{T1}{ptm}{m}{n}
			\rput(2.0476563,1.87){3}
			\pscircle[linewidth=0.04,dimen=outer,fillstyle=solid,fillcolor=color11247b](1.46,0.94){0.14}
			\usefont{T1}{ptm}{m}{n}
			\rput(1.6985937,0.79){2}
			\pscircle[linewidth=0.04,dimen=outer,fillstyle=solid,fillcolor=color11247b](1.16,1.56){0.14}
			\usefont{T1}{ptm}{m}{n}
			\rput(1.3970313,1.41){0}
			\pscircle[linewidth=0.04,dimen=outer,fillstyle=solid,fillcolor=color11247b](2.12,1.32){0.14}
			\usefont{T1}{ptm}{m}{n}
			\rput(2.3585937,1.17){2}
			\pscircle[linewidth=0.04,dimen=outer,fillstyle=solid,fillcolor=color11247b](3.02,2.36){0.14}
			\usefont{T1}{ptm}{m}{n}
			\rput(3.2570312,2.21){0}
			\pscircle[linewidth=0.04,dimen=outer,fillstyle=solid,fillcolor=color11247b](1.8,2.74){0.14}
			\usefont{T1}{ptm}{m}{n}
			\rput(2.0370312,2.59){0}
			\pscircle[linewidth=0.04,dimen=outer,fillstyle=solid,fillcolor=color11247b](4.06,2.68){0.14}
			\usefont{T1}{ptm}{m}{n}
			\rput(4.2970314,2.53){0}
			\pscircle[linewidth=0.04,dimen=outer,fillstyle=solid,fillcolor=color11247b](2.46,2.66){0.14}
			\usefont{T1}{ptm}{m}{n}
			\rput(2.6970313,2.51){0}
			\pscircle[linewidth=0.04,dimen=outer,fillstyle=solid,fillcolor=color11247b](4.46,2.12){0.14}
			\usefont{T1}{ptm}{m}{n}
			\rput(4.697031,1.97){0}
			\psline[linewidth=0.036cm,arrowsize=0.05291667cm 2.0,arrowlength=1.4,arrowinset=0.4]{->}(4.38,2.1)(3.9,2.04)
			\psline[linewidth=0.036cm,arrowsize=0.05291667cm 2.0,arrowlength=1.4,arrowinset=0.4]{->}(4.0,2.56)(3.86,2.14)
			\psline[linewidth=0.036cm,arrowsize=0.05291667cm 2.0,arrowlength=1.4,arrowinset=0.4]{->}(3.7,1.98)(3.08,1.7)
			\psline[linewidth=0.036cm,arrowsize=0.05291667cm 2.0,arrowlength=1.4,arrowinset=0.4]{->}(3.02,2.24)(3.0,1.78)
			\psline[linewidth=0.048cm,linecolor=color11003,arrowsize=0.05291667cm 2.0,arrowlength=1.4,arrowinset=0.4]{->}(2.86,1.6)(2.26,1.32)
			\psline[linewidth=0.036cm,arrowsize=0.05291667cm 2.0,arrowlength=1.4,arrowinset=0.4]{->}(2.38,2.54)(1.92,2.1)
			\psline[linewidth=0.036cm,arrowsize=0.05291667cm 2.0,arrowlength=1.4,arrowinset=0.4]{->}(1.82,2.64)(1.82,2.12)
			\psline[linewidth=0.036cm,arrowsize=0.05291667cm 2.0,arrowlength=1.4,arrowinset=0.4]{->}(1.86,1.88)(2.08,1.38)
			\psline[linewidth=0.036cm,arrowsize=0.05291667cm 2.0,arrowlength=1.4,arrowinset=0.4]{->}(1.24,1.46)(1.4,1.0)
			\psline[linewidth=0.048cm,linecolor=color11003,arrowsize=0.05291667cm 2.0,arrowlength=1.4,arrowinset=0.4]{->}(2.02,1.24)(1.58,0.94)
			\psline[linewidth=0.036cm](3.52,1.5)(3.06,1.64)
			\pscircle[linewidth=0.04,dimen=outer,fillstyle=solid,fillcolor=color11247b](1.26,2.52){0.14}
			\usefont{T1}{ptm}{m}{n}
			\rput(1.4770312,2.45){0}
			\pscircle[linewidth=0.04,dimen=outer,fillstyle=solid,fillcolor=color11247b](3.76,2.02){0.14}
			\usefont{T1}{ptm}{m}{n}
			\rput(3.9985938,1.87){2}
			\usefont{T1}{ptm}{m}{n}
			\rput(2.9051561,0.61){(a)}
			\pscircle[linewidth=0.04,dimen=outer,fillstyle=solid,fillcolor=color11247b](3.64,1.48){0.14}
			\usefont{T1}{ptm}{m}{n}
			\rput(3.8676562,1.29){3}
			\psdots[dotsize=0.16,dotstyle=asterisk](2.96,1.68)
			\psframe[linewidth=0.036,framearc=0.1,dimen=outer](11.66,3.3)(0.0,-3.3)
			\end{pspicture} }\end{center}\caption{The identity $\mathcal{D}\mathcal{A}_n(\x)=\Delta_nx_0^n$}\label{Fig:endo}
\end{figure}
Eq. (\ref{Eq:cayley1}) was first stated in \cite{Chen} in the language of grammars, and proved using Pr\"ufer codes (see also \cite{Dumont} for other formulas on similar enumeration problems).  Eq. (\ref{Eq:cayley2}) is, to the best of our knowledge, new. Observe that our proof of Eq. (\ref{Eq:cayley1}) relies on similar arguments as given by Joyal \cite{Joyal} in his classical proof of Cayley's formula for the number of trees (see also \cite{Labelle}). In this case, the polynomial obtained gives information about the degrees of the vertices of the rooted trees, and is equivalent, as it was pointed out by Chen \cite{Chen}, to the Lagrange inversion formula.   

We can give a visual proof of identity (\ref{chenfaa}) by considering the dart operators $x_{i+1}\partial_{x_i}$ and the corolla operators $x_1y_{i+1}\partial_{y_{i}}$,
\begin{equation}\label{Eq:chendartcoro}
\scalebox{1} % Change this value to rescale the drawing.
{
	\begin{pspicture}(0,-0.648125)(11.562813,0.648125)
	\definecolor{color1218b}{rgb}{0.8,0.8,0.8}
	\definecolor{color1223b}{rgb}{0.4,0.4,0.4}
	\pscircle[linewidth=0.04,dimen=outer,fillstyle=solid,fillcolor=color1218b](3.8109374,-0.3403125){0.17}
	\psdots[dotsize=0.24,dotstyle=asterisk](3.8009374,-0.3403125)
	\psline[linewidth=0.02cm](3.8109374,0.2296875)(3.8109374,-0.1903125)
	\pscircle[linewidth=0.04,dimen=outer,fillstyle=solid,fillcolor=color1223b](3.8209374,0.2496875){0.14}
	\usefont{T1}{ptm}{m}{n}
	\rput(4.072344,-0.4203125){${}_i$}
	\usefont{T1}{ptm}{m}{n}
	\rput(4.302344,0.2596875){${}_{i+1}$}
	\pscircle[linewidth=0.04,dimen=outer,fillstyle=solid](9.9609375,-0.27450013){0.18}
	\psdots[dotsize=0.26,dotstyle=asterisk](9.9609375,-0.27450013)
	\psline[linewidth=0.02cm](10.220938,0.3896875)(10.020938,-0.1103125)
	\pscircle[linewidth=0.04,dimen=outer,fillstyle=solid](10.250937,0.4396875){0.13}
	\psline[linewidth=0.02cm](9.560938,0.1296875)(9.830937,-0.1703125)
	\pscircle[linewidth=0.04,dimen=outer,fillstyle=solid,fillcolor=color1223b](9.540937,0.1696875){0.14}
	\usefont{T1}{ptm}{m}{n}
	\rput(1.7923437,-0.1003125){$x_{i+1}\partial_{x_i}=$}
	\usefont{T1}{ptm}{m}{n}
	\rput(7.2723436,-0.1203125){$x_1y_{i+1}\partial_{y_i}=$}
	\usefont{T1}{ptm}{m}{n}
	\rput(10.662344,0.4596875){${}_{i+1}$}
	\usefont{T1}{ptm}{m}{n}
	\rput(9.782344,0.1396875){${}_1$}
	\usefont{T1}{ptm}{m}{n}
	\rput(10.232344,-0.3403125){${}_i$}
	\end{pspicture} 
}\end{equation}
The dark vertices are associated to the $x$'s and the white ones to the $y$'s.
The reader may check that the combinatorial structures obtained after applying the operator $\Gamma_n=\mathcal{F}_n \mathcal{F}_{n-1}\dots \mathcal{F}_1$ to $y_0$ (sum of dart and corolla operators as in Eq. (\ref{Eq:chendartcoro})) the resulting structures are increasing trees in two kinds of colors, as
in Fig. \ref{Fig:chentrees} (a). It is not difficult to see that the weight of them is equal to
\begin{equation}
\label{Faa}\sum_{\pi\in\Pi[n]}y_{|\pi|}\prod_{B\in\pi}x_{|B|},
\end{equation}
which is the coefficient $z_n=Q(P(t))[n]$ (Fig. \ref{Fig:chentrees} (b)). 
\begin{figure}
	\begin{center}
		\scalebox{1} {
			\begin{pspicture}(0,-2.46875)(10.880938,2.46875)
			\definecolor{color7656b}{rgb}{0.8,0.8,0.8}
			\definecolor{color7673b}{rgb}{0.4,0.4,0.4}
			\psline[linewidth=0.036cm](6.767686,0.15716977)(6.783264,-0.70640266)
			\psline[linewidth=0.036cm](5.773916,-0.67782825)(6.514944,-0.9173635)
			\psline[linewidth=0.036cm](6.200112,-0.16757642)(6.630879,-0.7837255)
			\pscircle[linewidth=0.04,dimen=outer,fillstyle=solid](5.596961,-0.49207732){0.3}
			\rput(0.0, 0.0){\psdots[dotsize=0.24,dotangle=1.6585509,dotstyle=asterisk](5.555858,-0.44208297)}
			\usefont{T1}{ptm}{m}{n}
			\rput{1.6585509}(-0.013336386,-0.16717103){\rput(5.7257795,-0.5559867){1}}
			\rput{2.5645797}(0.007995463,-0.2727382){\pscircle[linewidth=0.04,dimen=outer,fillstyle=solid,fillcolor=color7656b](6.0962777,0.0422293){0.3}}
			\rput(0.0, 0.0){\psdots[dotsize=0.24,dotangle=2.5645797,dotstyle=asterisk](6.0150156,0.06861964)}
			\usefont{T1}{ptm}{m}{n}
			\rput{2.5645797}(0.004709378,-0.27755004){\rput(6.187432,-0.04311069){3}}
			\rput{1.7392882}(0.013385256,-0.2052538){\pscircle[linewidth=0.04,dimen=outer,fillstyle=solid,fillcolor=color7656b](6.767661,0.33827746){0.3}}
			\rput(0.0, 0.0){\psdots[dotsize=0.24,dotangle=1.7392882,dotstyle=asterisk](6.68618,0.38582632)}
			\usefont{T1}{ptm}{m}{n}
			\rput{1.7392882}(0.011098001,-0.20871954){\rput(6.8588023,0.2516722){8}}
			\rput{6.3480153}(-0.06073111,-0.75054014){\pscircle[linewidth=0.04,dimen=outer,fillstyle=solid,fillcolor=color7673b](6.7369123,-0.9228548){0.27}}
			\usefont{T1}{ptm}{m}{n}
			\rput{-2.4062877}(0.05843174,0.25260183){\rput(6.0239363,-1.2451172){$x_3$}}
			\psline[linewidth=0.036cm](9.360937,-1.1853125)(8.280937,-1.6653125)
			\psline[linewidth=0.036cm](9.700937,-1.0453125)(10.360937,-0.9853125)
			\psline[linewidth=0.036cm](9.880938,-0.3653125)(9.580937,-0.9453125)
			\pscircle[linewidth=0.04,dimen=outer,fillstyle=solid](10.020938,-0.1253125){0.3}
			\psdots[dotsize=0.24,dotstyle=asterisk](9.920938,-0.0753125)
			\usefont{T1}{ptm}{m}{n}
			\rput(10.101875,-0.1753125){4}
			\pscircle[linewidth=0.04,dimen=outer,fillstyle=solid,fillcolor=color7656b](10.580937,-0.9653125){0.3}
			\psdots[dotsize=0.24,dotstyle=asterisk](10.500937,-0.8953125)
			\usefont{T1}{ptm}{m}{n}
			\rput(10.67625,-1.0153126){7}
			\pscircle[linewidth=0.04,dimen=outer,fillstyle=solid,fillcolor=color7673b](9.490937,-1.0753125){0.27}
			\usefont{T1}{ptm}{m}{n}
			\rput(9.819531,-1.3153125){2}
			\psline[linewidth=0.036cm](8.040937,-0.4653125)(8.060938,-1.5053124)
			\psline[linewidth=0.036cm](6.9209375,-1.0653125)(7.8009377,-1.6053125)
			\psline[linewidth=0.036cm](7.6609373,0.7746875)(7.9809375,-0.1853125)
			\psline[linewidth=0.036cm](9.120937,0.1146875)(8.220938,-0.1853125)
			\psline[linewidth=0.036cm](8.780937,0.5546875)(8.120937,-0.1853125)
			\pscircle[linewidth=0.04,dimen=outer,fillstyle=solid](7.6209373,1.0146875){0.3}
			\psdots[dotsize=0.24,dotstyle=asterisk](7.5409374,1.0646875)
			\usefont{T1}{ptm}{m}{n}
			\rput(7.719531,0.9446875){2}
			\psline[linewidth=0.036cm](8.260938,0.8146875)(8.100938,-0.1253125)
			\pscircle[linewidth=0.04,dimen=outer,fillstyle=solid,fillcolor=color7656b](8.300938,1.0546875){0.3}
			\psdots[dotsize=0.24,dotstyle=asterisk](8.240937,1.0846875)
			\usefont{T1}{ptm}{m}{n}
			\rput(8.410469,0.9846875){5}
			\pscircle[linewidth=0.04,dimen=outer,fillstyle=solid,fillcolor=color7656b](9.300938,0.2346875){0.3}
			\psdots[dotsize=0.24,dotstyle=asterisk](9.200937,0.2846875)
			\usefont{T1}{ptm}{m}{n}
			\rput(9.376094,0.1446875){9}
			\pscircle[linewidth=0.04,dimen=outer,fillstyle=solid,fillcolor=color7656b](8.900937,0.7546875){0.3}
			\psdots[dotsize=0.24,dotstyle=asterisk](8.820937,0.8246875)
			\usefont{T1}{ptm}{m}{n}
			\rput(9.01625,0.7246875){6}
			\pscircle[linewidth=0.04,dimen=outer,fillstyle=solid,fillcolor=color7673b](8.050938,-0.2753125){0.27}
			\usefont{T1}{ptm}{m}{n}
			\rput(8.421875,-0.3753125){4}
			\usefont{T1}{ptm}{m}{n}
			\rput(7.128594,-0.9553125){3}
			\pscircle[linewidth=0.04,dimen=outer,fillstyle=solid](8.060938,-1.6853125){0.3}
			\usefont{T1}{ptm}{m}{n}
			\rput(8.468594,-1.8353125){3}
			\usefont{T1}{ptm}{m}{n}
			\rput(7.3123436,-1.7553124){$y_3$}
			\usefont{T1}{ptm}{m}{n}
			\rput(8.802343,-0.9153125){$x_2$}
			\usefont{T1}{ptm}{m}{n}
			\rput(7.382344,-0.2353125){$x_4$}
			\psline[linewidth=0.036cm](3.6809375,0.9103125)(3.4809375,0.4103125)
			\psline[linewidth=0.036cm](3.2609375,-0.0696875)(3.0609374,-0.5696875)
			\psline[linewidth=0.036cm](2.9209375,-0.9296875)(2.7209375,-1.4296875)
			\psline[linewidth=0.036cm](2.1209376,-1.1896875)(2.3909376,-1.4896874)
			\pscircle[linewidth=0.04,dimen=outer,fillstyle=solid](2.5809374,-1.6896875){0.3}
			\psdots[dotsize=0.24,dotstyle=asterisk](2.5209374,-1.6396875)
			\usefont{T1}{ptm}{m}{n}
			\rput(2.6878126,-1.7596875){1}
			\pscircle[linewidth=0.04,dimen=outer,fillstyle=solid](3.0009375,-0.7496875){0.3}
			\pscircle[linewidth=0.04,dimen=outer,fillstyle=solid](3.3809376,0.1903125){0.3}
			\pscircle[linewidth=0.04,dimen=outer,fillstyle=solid](3.7809374,1.1703125){0.3}
			\pscircle[linewidth=0.04,dimen=outer,fillstyle=solid,fillcolor=color7656b](2.0209374,-1.0696875){0.3}
			\psline[linewidth=0.036cm](1.5609375,-0.5896875)(1.8309375,-0.8896875)
			\psline[linewidth=0.036cm](2.9209375,0.6703125)(3.1909375,0.3703125)
			\pscircle[linewidth=0.04,dimen=outer,fillstyle=solid,fillcolor=color7656b](2.8209374,0.7903125){0.3}
			\psline[linewidth=0.036cm](2.3609376,1.2703125)(2.6309376,0.9703125)
			\pscircle[linewidth=0.04,dimen=outer,fillstyle=solid,fillcolor=color7656b](1.4609375,-0.4296875){0.3}
			\psline[linewidth=0.036cm](1.0009375,0.0703125)(1.2709374,-0.2296875)
			\pscircle[linewidth=0.04,dimen=outer,fillstyle=solid,fillcolor=color7673b](2.2109375,1.4203125){0.27}
			\pscircle[linewidth=0.04,dimen=outer,fillstyle=solid,fillcolor=color7673b](0.9109375,0.2003125){0.27}
			\psline[linewidth=0.036cm](2.5209374,-0.2696875)(2.7909374,-0.5696875)
			\psline[linewidth=0.036cm](1.9609375,0.2703125)(2.2309375,-0.0296875)
			\pscircle[linewidth=0.04,dimen=outer,fillstyle=solid,fillcolor=color7656b](2.4209375,-0.1696875){0.3}
			\psline[linewidth=0.036cm](1.4009376,0.8903125)(1.6709375,0.5903125)
			\psline[linewidth=0.036cm](0.8009375,1.5503125)(1.0709375,1.2503124)
			\pscircle[linewidth=0.04,dimen=outer,fillstyle=solid,fillcolor=color7656b](1.2609375,1.0503125){0.3}
			\pscircle[linewidth=0.04,dimen=outer,fillstyle=solid,fillcolor=color7656b](1.8209375,0.4303125){0.3}
			\pscircle[linewidth=0.04,dimen=outer,fillstyle=solid,fillcolor=color7673b](0.6909375,1.7003125){0.27}
			\psdots[dotsize=0.24,dotstyle=asterisk](3.2809374,0.2403125)
			\usefont{T1}{ptm}{m}{n}
			\rput(3.461875,0.1403125){4}
			\psdots[dotsize=0.24,dotstyle=asterisk](1.9409375,-1.0396875)
			\usefont{T1}{ptm}{m}{n}
			\rput(2.1085937,-1.1596875){3}
			\psdots[dotsize=0.24,dotstyle=asterisk](2.9209375,-0.6996875)
			\usefont{T1}{ptm}{m}{n}
			\rput(3.0995312,-0.8196875){2}
			\psdots[dotsize=0.24,dotstyle=asterisk](2.3609376,-0.1396875)
			\usefont{T1}{ptm}{m}{n}
			\rput(2.5304687,-0.2396875){5}
			\psdots[dotsize=0.24,dotstyle=asterisk](1.7409375,0.5003125)
			\usefont{T1}{ptm}{m}{n}
			\rput(1.93625,0.4003125){6}
			\psdots[dotsize=0.24,dotstyle=asterisk](2.7409375,0.8603125)
			\usefont{T1}{ptm}{m}{n}
			\rput(2.91625,0.7403125){7}
			\psdots[dotsize=0.24,dotstyle=asterisk](1.3809375,-0.3796875)
			\usefont{T1}{ptm}{m}{n}
			\rput(1.5496875,-0.5196875){8}
			\psdots[dotsize=0.24,dotstyle=asterisk](1.1609375,1.1003125)
			\usefont{T1}{ptm}{m}{n}
			\rput(1.3360938,0.9603125){9}
			\usefont{T1}{ptm}{m}{n}
			\rput(3.9323437,1.9003125){$y_3$}
			\usefont{T1}{ptm}{m}{n}
			\rput(4.128594,0.8803125){3}
			\usefont{T1}{ptm}{m}{n}
			\rput(2.5995312,1.3803124){2}
			\usefont{T1}{ptm}{m}{n}
			\rput(1.061875,1.6003125){4}
			\usefont{T1}{ptm}{m}{n}
			\rput(1.3085938,0.1603125){3}
			\usefont{T1}{ptm}{m}{n}
			\rput(2.2023437,2.0403125){$x_2$}
			\usefont{T1}{ptm}{m}{n}
			\rput(0.48234376,2.2803125){$x_4$}
			\usefont{T1}{ptm}{m}{n}
			\rput(0.44234374,0.7603125){$x_3$}
			\usefont{T1}{ptm}{m}{n}
			\rput(3.0460937,-2.241875){(a)}
			\usefont{T1}{ptm}{m}{n}
			\rput(8.7760935,-2.241875){(b)}
			\end{pspicture} 
		}\end{center}\caption{Increasing tree counted by $\Gamma_9y_0=\mathcal{F}_9 \mathcal{F}_{8}\dots \mathcal{F}_1y_0$ and equivalent partition tree.}\label{Fig:chentrees}
	\end{figure}
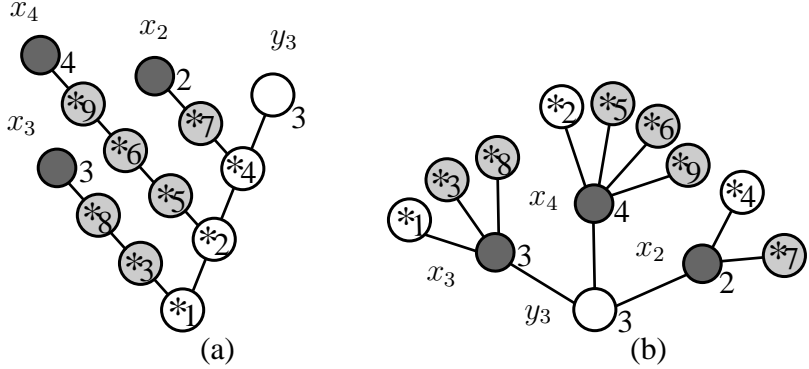

\section{One-parameter groups, generalized exponential polynomials, and enumeration of enriched trees} \label{sec:onepgroup} In this section we extend the kind of operators studied in previous sections to an arbitrary number of variables.  By considering the group infinitesimaly generated by a differential operator, we define generalized exponential polynomials and establish interesting relationships with different versions of the Lagrange inversion formula. In this way we generalized formula (\ref{Eq:cayley1}) to a wide family of enriched trees. 

Let $\mathcal{R}=\mathbb{C}[[X]]$ be the ring of formal power series in an arbitrary set of variables $X$  (either finite or infinite). The elements of $\mathcal{R}$ are of the form 
\begin{equation}\label{infinite}
H(X)=\sum_{\mathbf{k}\in \mathbb{N}_X}H[\mathbf{k}]\prod_{k_x\neq 0}\frac{x^{k_x}}{k_x!}=\sum_{\mathbf{k}\in \mathbb{N}_X}H[\mathbf{k}]\frac{X^{\mathbf{k}}}{\mathbf{k}!}
\end{equation}
In the sum of above,  $\mathbb{N}_X$ is the set of tuples $\mathbf{k}=(k_x)_{x\in X}$ such that $k_x=0$ for almost every $x\in X$. The conventions for powers and multiple factorials are similar to that at the beginning of Section \ref{Sec:combinaformal}. The combinatorics of the coefficients of the series regarding the operations of sum, product and partial derivatives is also similar. These combinatorial operations are interpreted in the same way by using colored sets of the form $(V,\kappa)$, $\kappa:V\rightarrow X$ being a coloration. Here we identify the set of colors with the set of variables itself, and define $H[V,\kappa]:=H[\mathbf{k}]$, $k_x=|\kappa^{-1}(x)|$.
\begin{defi}\normalfont 
A family $\mathbf{G}=\{G_s\}_{s\in S}$ of formal power series in $\mathcal{R}$ is said to be summable if for every $\mathbf{k}$, the coefficient $G_s[\mathbf{k}]$ is zero for almost every $s\in S$. Equivalently, for every colored set $(V,\kappa)$, we have that $G_s[V,\kappa]=0$ for almost every $s\in S$. 
\end{defi}

Observe that a family $\{G_s(X)\}_{s\in S}$ is summable if and only if the sum $\sum_{s\in S} G_s(X)$ is  well defined as a formal power series. The notion of summability is very important for the study of series in an infinite number of variables. It is the essential condition in order to perform the operation of substitution. We have the following Lemma.
\begin{lem}\normalfont  \label{prod}
	Assume that $\mathbf{G}(X)= \{G_s(X)\}_{s\in S}$ is a summable family. Then for any other, not necessarily summable family $\mathbf{H}(X)=\{H_s(X)\}_{s\in S}$, the product $\{G_s(X)H_s(X)\}_{s\in S}$ is also summable. 
\end{lem}
\begin{proof}
	We have to prove that for any colored set $(V,\kappa)$, the set of elements $s\in S$ such that $(G_s.H_s)[V,\kappa]\neq 0$ is finite. By the combinatorial definition of product of formal power series (Eq. (\ref{combinatorialproduct})) we have that
	\begin{equation*}
	\{s|(G_s.H_s)[V,\kappa]\neq 0\}\subseteq \{s|\exists V_1\subseteq V,\; G_s[V_1,\kappa|_{V_1}]\neq 0\}= \bigcup_{V_1\subseteq V}\{s|G_s[V_1,\kappa|_{V_1}]\neq 0\}.
	\end{equation*}
\noindent The result follows from the summability of $\mathbf{G}$ and because the union in the right hand side of the equality is over a finite set. 
\end{proof}

\begin{theo} \label{substi}\normalfont
	Let $\mathbf{G}(X)=\{G_x(X)\}_{x\in X}$ be a summable family indexed by the set of variables $X$, and such that for every $x\in X$, $G_x(X)$ has zero constant term. Then, \begin{enumerate}
		\item The family $$\{\mathbf{G}^{\mathbf{k}}(X)\}_{\mathbf{k}\in \mathbb{N}_X},\; \mathbf{G}^{\mathbf{k}}(X)=\prod_{k_x\neq 0}G_x(X)^{k_x}$$
			is summable.
			\item Let $H(X)$ be an arbitrary series. Then, the substitution \begin{equation}
			H(G_x(X))_{x\in X}=\sum_{\mathbf{k}\in \mathbb{N}_X}H[\mathbf{k}]\prod_{k_x\neq 0}\frac{G_x(X)^{k_x}}{k_x!}=\sum_{\mathbf{k}\in \mathbb{N}_X}H[\mathbf{k}]\frac{\mathbf{G}^{\mathbf{k}}(X)}{\mathbf{k}!}
			\end{equation} is a well defined formal power series. Conversely, if $H(G_x)_{x\in X}$ is well defined for every  formal power series $H(X)$, then $ \{G_x(X)\}_{x\in X}$ is summable. 
	\end{enumerate}
	\end{theo}
	\begin{proof}
		The first part of item 2 is a consequence of the Lemma \ref{prod} and item 1. The converse part of item 2 follows by taking $H(X)=\sum_{x\in X}x$. 
		Then, what remains to prove is that the set $\{\mathbf{k}|\mathbf{G}^{\mathbf{k}}[V,\kappa]\neq 0\}$ is finite for every colored set $(V,\kappa)$. By the same argument used in Lemma \ref{prod}, the summability of $\mathbf{G}$ implies that the set $$S_1=\{x|\exists V_1\subseteq V: G_x[V_1,\kappa|_{V_1}]\neq 0\}$$ is finite. Hence, by the definition of product of series we have that \begin{equation}\label{cond1}
		\{x|\exists k\geq 1: G^k_x[V,\kappa]\neq 0\}\subseteq S_1.
		\end{equation}
		Since $G_x[\emptyset]=0$ for every $x\in X$, we also have
		\begin{equation}\label{cond2}
k>|V|\Rightarrow G_x^k[V,\kappa]=0.	
		\end{equation} 
		By Eq. (\ref{cond1}) and condition in Eq. (\ref{cond2}), we obtain that
		$$\{\mathbf{k}|\mathbf{G}^{\mathbf{k}}[V,\kappa]\neq 0\}\subseteq \{\mathbf{k}|k_x=0,\mbox{ if } x\notin S_1\}\cap \{\mathbf{k}|\forall x: 1\leq k_x\leq |V|\},$$
		which is a finite set.
	\end{proof}

	The coefficient of the substitution $H(G_x)_{x\in X}$ on a colored set $(V,\kappa)$ has the following combinatorial interpretation (see \cite{Nava-yo}):
	\begin{equation}\label{substitution}
	H(G_x)_{x\in X}[V,\kappa]=\sum_{(\pi,\hat{\kappa})}H[\pi,\hat{\kappa}]\prod_{B\in\pi}G_{\hat{\kappa}(B)}[B,\kappa|_B].
	\end{equation}
	\noindent Where the sum of above ranges over all colored partitions $(\pi, \hat{\kappa})$, $\pi$ a partition of $V$, and   $\hat{\kappa}:\pi\rightarrow X$ a (external) coloring of the blocks of $\pi$. The summability of $\{G_x(X)\}_{x\in X}$ ensures that the sum has only a finite number of non vanishing terms.
\begin{defi}\normalfont 
	For a summable family $\boldsymbol{\phi}(X)=\{\phi_x(X)\}_{x\in S}$, indexed by a subset $S$ of $X$, define the operator $\mathscr{D}^{\boldsymbol\phi}$  by $$\mathscr{D}^{\boldsymbol\phi}:=\sum_{x\in S }\phi_x(X)\partial_x.$$ 
\end{defi}
\noindent By Lemma \ref{prod}, this is a well defined operator on $\mathcal{R}$. It is easy to check that it is a derivation. Moreover, it preserves summability,
\begin{lem}\normalfont  \label{preserv}
	Let $\{G_z(X)\}_{z\in Z}$ be a summable family. Then for every non negative integer $n$ we have that $\{(\mathscr{D}^{\boldsymbol{\phi}})^n G_z(X)\}_{z\in Z}$ is also summable.
\end{lem}
Define $e^{t\mathscr{D}^{\boldsymbol{\phi}}}$, by
\begin{equation}\label{parametergroup}e^{t\mathscr{D}^{\boldsymbol{\phi}}}H(X)=\sum_{n=0}^{\infty}(
\mathscr{D}^{\phi})^n H(X) \frac{t^n}{n!}\in \mathcal{R}[[t]].\end{equation}
By abuse of language we call $e^{t\mathscr{D}^{\boldsymbol{\phi}}}$ the group infinitesimaly generated by $\mathscr{D}^{\boldsymbol{\phi}}$. The series in (\ref{parametergroup}) is not necessarily well defined as an element of $\mathcal{R}$ for particular values of $t$.  

 Define the family $\{F^{\boldsymbol{\phi}}_x(t)\}_{x\in X}$ of formal power series in $\mathcal{R}[[t]]=\mathbb{C}[[X,t]]$ by
\begin{equation}\label{group}
F^{\boldsymbol{\phi}}_x(t):=e^{t\mathscr{D}^{\phi}}x =\sum_{n=0}^{\infty}(\mathscr{D}^{\phi})^n x \frac{t^n}{n!}, \; x\in X.
\end{equation}
 \begin{rem}\normalfont By representing $\mathscr{D}^{\boldsymbol{\phi}}$ as sum of corolla operators, and using grafting techniques as in Prop.\ref{tree} we obtain that the series $F^{\boldsymbol{\phi}}_x(t)$ has the following combinatorial interpretation. The coefficients $(\mathscr{D}^{\boldsymbol{\phi}})^nx=\mathcal{T}_{n,x}^{\boldsymbol{\phi}}(X)\in \mathbb{C}[[X]]$ in the expansion of $F^{\boldsymbol{\phi}}_x(t)$
 have as configurations increasing trees enriched with the family $\{\phi_z(X)\}_{z\in X}$. More precisely, it counts the number of increasing trees with $n$ internal vertices colored with colors in $X$, the root colored with $x$ and leaves weighted with their colors (in X). The fiber of an internal vertex colored with color $z$ is enriched (weighted) with the corresponding series $\phi_z(X)$.
  \end{rem} 
\begin{lem}\normalfont  
	The family of formal power series $\{F^{\boldsymbol{\phi}}_x(t)\}_{x\in X}$ defined by Eq. (\ref{group}) is summable.
\end{lem}
\begin{proof}
The family $\{x\}_{x\in X}$ is clearly summable. By Lemma \ref{preserv}, $\{(\mathscr{D}^{\boldsymbol{\phi}})^n x\}_{x\in X}$ is summable for every $n\in \mathbb{N}$. Then, $\{F^{\boldsymbol{\phi}}_x(t)\}_{x\in X}$ is also summable.   
\end{proof}
For simplicity, when clear from the context, we suppress the super index ${\boldsymbol{\phi}}$ from the family. This summable family completely  characterizes the one parameter operator group $e^{t\mathscr{D}^{\boldsymbol\phi}}$.

\begin{cor}\normalfont  \label{important}
	Let $\{F_x(t)\}_{x\in X}$ be as defined in Eq. (\ref{group}). Then, for any series $H(X)\in \mathcal{R}$ we have
	\begin{equation}e^{t\mathscr{D}^{\boldsymbol\phi}}H(X)=H(F_x(t))_{x\in X}
	\end{equation}
	
\end{cor}
\begin{proof}
	Since $\mathscr{D}^{\boldsymbol{\boldsymbol\phi}}$ is a derivation, $e^{t\mathscr{D}^{\boldsymbol\phi}}$ is a multiplicative map. Then,
	$$e^{t\mathscr{D}^{\boldsymbol{\phi}}}X^{\mathbf{k}}=\prod_{k_x\neq 0}F_x(t)^{k_x},$$
	which is summable by Theorem \ref{substi}. We have
	$$e^{t\mathscr{D}^{\boldsymbol{\phi}}}H(X)=\sum_{\mathbf{k}\in \mathbb{N}_X}H[\mathbf{k}]\frac{e^{t\mathscr{D}^{\boldsymbol{\phi}}}X^{\mathbf{k}}}{\mathbf{k}!}=\sum_{\mathbf{k}\in \mathbb{N}_X}H[\mathbf{k}]\frac{\prod_{k_x\neq 0}F_x(X)^{k_x}}{\mathbf{k}!}=H(F_x(t))_{x\in X}.$$
\end{proof}
\begin{cor}\normalfont  \label{Cor:gendifequation}
	The family $\{F_x(t)\}_{x\in X}$ satisfies the system of differential equations
	\begin{equation}
	\begin{cases}y'_x(t)=
	\phi_x(y_z(t))_{z\in X},\; x\in X\\
	y_x(0)=x.
	\end{cases}
	\end{equation}
\end{cor}
\begin{proof}
	The derivative $F_x(t)$ is readily computed by interchanging with the one parameter group operator
	$$F'_x(t)=\frac{d}{dt}e^{t\mathscr{D}^{\boldsymbol{\phi}}}x=e^{t\mathscr{D}^{\boldsymbol{\phi}}}   \mathscr{D}^{\boldsymbol{\phi}}x=e^{t\mathscr{D}^{\boldsymbol{\phi}}}\phi_x(X)=\phi_x(F_z(t))_{z\in X}.$$
	The latter equality obtained from Corollary \ref{important}.	
\end{proof}

\subsection{Generalized exponential polynomials} 
Let $F_x^+(t):=F_x(t)-x$. Define the generalized exponential polynomials $Y^x_n(X)$ by the generating function
\begin{equation*}
Y_x(t)=\sum_{n=0}^{\infty}Y^x_n(X)\frac{t^n}{n!}:=e^{F_x^+(t)}.
\end{equation*}
The generating function $Y^x_n(X)$ counts the number of forests of non-singleton increasing trees with $n$ internal nodes.

By the definition and Corollary \ref{Cor:gendifequation},	the generating function $Y_x(t)$ satisfies the following system of differential equation 
\begin{equation}
\begin{cases}Y'_x(t)=\phi_x(F_z)_{z\in X}.Y_x(t),\; x\in X \\ Y_x(0)=1.\end{cases}
	\end{equation}
As we have that $e^{t\mathscr{D}^{\boldsymbol{\phi}}}e^x=e^{F_x(t)}=e^x e^{F^+_x(t)}$, we obtain 
\begin{equation} 
Y_x(t)=e^{F^+_x(t)}=e^{-x}e^{t\mathscr{D}^{\boldsymbol{\phi}}}e^x,
\end{equation}
\noindent and the Rodrigues-like formula
\begin{equation}\label{Eq:Rodrigues}
Y_n^x(X)=e^{-x}(\mathscr{D}^{\boldsymbol{\phi}})^ne^x.
\end{equation}

\noindent From that we get the recursion
\begin{prop}\normalfont  
The exponential polynomials $Y_n^x(X)$ satisfy the recursive formula
\begin{equation}
Y_{n+1}^x(X)=\phi_x(X)Y_n^x(X)+\mathscr{D}^{\boldsymbol{\phi}}Y_n(X).
\end{equation}
\end{prop}
\begin{proof}
	By Eq. (\ref{Eq:Rodrigues})  Hence,
	$$Y_{n+1}^x(X)=e^{-x}\mathscr{D}^{\boldsymbol{\phi}}(\mathscr{D}^{\boldsymbol{\phi}})^ne^x=e^{-x}\mathscr{D}^{\boldsymbol{\phi}}e^xY_n^x(X)=e^{-x}(e^x\phi_x(X)Y_n(X)+e^x\mathscr{D}^{\boldsymbol{\phi}}Y_n^x(X)).$$
\end{proof}
\begin{ex}\normalfont  
	Let $X=\{x\}$, the differential operators are of the form $\phi(x)D$, $D$ being the ordinary derivative $Df(x)=f'(x)$. The series $F_x(t)= \mathcal{A}_{\phi}^{\uparrow}(x,t)=e^{t\phi(x)D}x$ is the solution to the differential equation (\ref{autonomous1}). 
	\begin{enumerate}
		\item The operator $xD$ gives us the differential equation $y'=y$, $y(0)=x$ that has as solution $F_x(t)=xe^t$. The exponential polynomials are Touchard's (\cite{Touchard})
		$$\sum_{n=0}^{\infty}T_n(x)\frac{t^n}{n!}=e^{x(e^t-1)}.$$
		Touchard polynomials enumerate the set partitions according with the number of blocks. The recursive formula becomes the classical
		$$T_{n+1}(x)=x(T_{n}(x)+T'_n(x)).$$
		\item More generally, the operator $x^rD$ has associated the differential equation $y'=y^r$, $y(0)=x$. It has as solution
		$$F_x(t)=\frac{x}{\sqrt[r-1]{1-(r-1)x^{r-1}t}},$$
		the generating function for the $r$-ary plane increasing trees, the leaves weighted by $x$. The corresponding exponential polynomials count the forests of such trees. They satisfy the recursion
		$$T^{(r)}_{n+1}=x^r(T^{(r)}_{n}(x)+DT_n^{(r)}(x)).$$
	\end{enumerate}
	\end{ex}
	\begin{ex}\normalfont  
		For the set $X=\{x_0, x_1, x_2,\dots;q_0,q_1,\dots\}$, the operator $$\mathscr{D}^{\phi,q}=\sum_{k=0}^{\infty}\phi(x_{k+1})q_k\partial_{x_k}$$ extends $\mathcal{D}_n^{\phi,q}$ to $\mathbb{C}[[X]].$ Clearly, $\phi(x_{k+1})q_k$ is a summable family.
		
		The series $$F_{x_0}^{\phi}(t)=e^{t\mathscr{D}^{\phi,q}}x_0=\sum_{k=0}^{\infty} (\mathscr{D}^{\phi,q})^k x_{0}=:\mathcal{A}^{\uparrow}_{\phi}(t,\mathbf{x},\mathbf{q})$$
		is the exponential generating function of the increasing trees $\mathcal{T}_n^{\phi}(\mathbf{x},\mathbf{q})$,$$\mathcal{A}^{\uparrow}_{\phi}(t,\mathbf{x},\mathbf{q})=\sum_{n=0}^{\infty}\mathcal{T}_n^{\phi}(\mathbf{x},\mathbf{q})\frac{t^n}{n!}.$$

		According to Corollary \ref{Cor:gendifequation}, it is the solution to the differential equation
		$$\begin{cases}(F_{x_0}^{\phi})'(t)=\frac{\partial}{\partial t}\mathcal{A}^{\uparrow}_{\phi}(t,\mathbf{x},\mathbf{q})=q_0\phi(F_{x_1}^{\phi}(t))=q_0\phi(\mathcal{A}^{\uparrow}_{\phi}(t,S\mathbf{x},S\mathbf{q})),\\F_{x_0}^{\phi}(0)=x_0.\end{cases}$$
since, by the combinatorial interpretation of $\mathcal{T}_n^{\phi}(\mathbf{x},\mathbf{q})$ in Proposition \ref{tree1} it is easy to see that 
$$F_{x_n}^{\phi}(t)=\mathscr{A}_{\phi}^{\uparrow}(t,S^n\mathbf{x},S^n\mathbf{q}),$$
$S$	being the shift operator:
\begin{eqnarray*}S(x_0,x_1,x_2,\dots)&=&(x_1,x_2,x_3,\dots),\\ S(q_0,q_1,q_2,\dots)&=&(q_1,q_2,q_3,\dots).	
	\end{eqnarray*}\end{ex}
	
\subsection{Lagrange inversion}
For a formal power series $F(t)\in \mathcal{R}[[t]]$ with non-zero constant term, define the generating series of $F$-enriched trees $\mathcal{A}_F(t)$, by the implicit formula (see \cite{BLL, Joyal})
$$\mathcal{A}_F(t)=tF(\mathcal{A}_F(t)).$$
 The configurations of $\mathcal{A}_F[n]\in \mathbb{C}[[X]]$ are rooted trees weighted with the coefficients of $F(X)$ as follows,
\begin{equation}
\mathcal{A}_F[n]=\sum_{T}\prod_{k=1}^nF[T^{-1}(v)]
\end{equation} 
Where the sum is over all rooted trees with vertices in $[n]$.

As was pointed out in \cite{Chen}, formula (\ref{Eq:cayley1}) is equivalent to the Lagrange inversion formula. That can be stated in terms of the enumeration of $F$-enriched trees as follows.
\begin{prop}\normalfont  
	\label{Lagrange}The coefficient $\mathcal{A}_F[n]$ is given by the formula
	\begin{equation}
	\mathcal{A}_F[n]=\frac{d^{n-1}}{dt^{n-1}}F^{n}(t)|_{t=0}.
	\end{equation}
\end{prop}
Consider the infinite set $X=\{x_0, x_1, x_2, \dots\}$, and let \begin{equation}F_0(t)=\sum_{k=0}^{\infty}x_k\frac{t^k}{k!}.\end{equation} The $F_0(t)$-enriched trees are the same trees enumerated by Eq. (\ref{Eq:cayley1}). The weight of each vertex $v$ being the variable $x_{d_T(v)}$. 
Let $\phi_k(X)=x_{k+1}$, a clearly summable family. The derivation $$\mathscr{D}:=\mathscr{D}^{\boldsymbol\phi}=\sum_{k=0}^{\infty}x_{k+1}\partial_k$$ is the extension of $\mathcal{D}_n$ to an infinite number of variables. 
We have that $$F_{x_0}(t)=e^{t\mathscr{D}}x_0=\sum_{k=0}^{\infty}\mathscr{D}^n x_0\frac{t^n}{n!}=F_0(t)$$
By Corollary \ref{important}, $F_0^n(t)=e^{t\mathscr{D}}x_0^n.$
Lagrange inversion gives us Formula (\ref{Eq:cayley1}) 
$$\mathcal{A}_{F_0}[n]=\frac{d^{n-1}}{dt^{n-1}}F_0^{n}(t)|_{t=0}=\frac{d^{n-1}}{dt^{n-1}}e^{t\mathscr{D}}x_0^n|_{t=0}=e^{t\mathscr{D}}\mathscr{D}^{n-1}x_0^n|_{t=0}=\mathscr{D}^{n-1}x_0^n.$$ Conversely, assuming $\mathcal{A}_{F_0}[n]=\mathscr{D}^{n-1}x_0^n$ and going backwards we obtain $$\mathcal{A}_{F_0}[n]=\frac{d^{n-1}}{dt^{n-1}}F_0^{n}(t)|_{t=0}.$$
Since the variables $x_0, x_1, x_2,\dots$ are algebraically independent,  they can be replace by the coefficients of any other series, and we obtain the Lagrange inversion formula. 

Lagrange inversion formula for enriched trees series of the form $\mathcal{A}_{H(F_x(t))_{x\in X}}(t)$, $F_x(t)=e^{t\mathscr{D}}x$, gives us a powerful computational device.

\begin{cor}\label{Cor:generalchen}\normalfont    
Let $H(X)$ be a formal power series in $\mathcal{R}[[X]]$. The coefficient $\mathcal{A}_{H(F_x(t))_{x\in X}}[n]$ is given by
\begin{equation}
\mathcal{A}_{H(F_x(t))_{x\in X}}[n]=(\mathscr{D}^{\boldsymbol{\phi}})^{n-1}H^n(X)
\end{equation}
In particular we have
\begin{equation}
\mathcal{A}_{F_x(t)}[n]=(\mathscr{D}^{\boldsymbol{\phi}})^{n-1}x^n,\; x\in X.
\end{equation}
\end{cor}
\begin{proof}By Corollary \ref{important}, $H^n(F_x(t))_{x\in X}=e^{t\mathscr{D}^{\boldsymbol{\phi}}}H^n(X).$
	By the Lagrange inversion formula
	\begin{eqnarray*}
	\mathcal{A}_{H(F_x(t))_{x\in X}}[n]=\frac{d^{n-1}}{dt^{n-1}}H^n(F_x(t))_{x\in X}|_{t=0}&=&\frac{d^{n-1}}{dt^{n-1}}e^{t\mathscr{D}^{\boldsymbol{\phi}}}H^n(X)|_{t=0}\\&=&e^{t\mathscr{D}^{\boldsymbol{\phi}}}(\mathscr{D}^{\boldsymbol{\phi}})^{n-1}H^n(X)|_{t=0}.
	\end{eqnarray*}\end{proof}
	There are two simple and important cases of  Corollary \ref{Cor:generalchen}, $H(X)=e^x$ and $H(X)=\frac{1}{1-x}$. In the first one, we enumerate rooted trees enriched with assemblies (unordered tuples) of configurations enumerated by $F_x(t)$, $\mathcal{A}_{e^{F_x(t)}}[n]$. In the second, rooted trees enriched with tuples of structures enumerated by $F_x(t)$. $\mathcal{A}_{\frac{1}{1-F_x(t)}}[n]$. However, the most interesting examples of enriched trees are obtained by enriching with assemblies of non-empty configurations of $F_x(t)$, and with tuples of non-empty configurations of $F_x(t)$. Respectively  $\mathcal{A}_{e^{F_x^+(t)}}[n]$ and $\mathcal{A}_{\frac{1}{1-F_x^+(t)}}[n]$. 
	\begin{theo}
		We have the formulas
		\begin{eqnarray}\label{enrichedwforests}
		\mathcal{A}_{e^{F_x^+(t)}}[n]&=&e^{-nx}(\mathscr{D}^{\boldsymbol\phi})^{n-1}e^{nx}\\\label{enrichedwlinearforests}
		(\mathscr{D}^{\boldsymbol\phi})^{n-1}(1-x)^{-n}&=&\sum_{k=1}^{n-1}\frac{c_{n,k}(X)}{(1-x)^{n+k}}
		\end{eqnarray}
		\noindent Where in Eq. (\ref{enrichedwlinearforests}), the formal power series $c_{n,k}(X)=\mathcal{A}^{\underline{k}}_{\frac{1}{1-F_x^+(t)}}[n]$ enumerates the configurations of $\mathcal{A}_{\frac{1}{1-F_x^+(t)}}[n]$ enriched with tuples of $F_x^+(t)$-configurations having $k$ blocks in total
		$$\mathcal{A}_{\frac{1}{1-F_x^+(t)}}[n]=\sum_{k=1}^{n-1}\mathcal{A}^{\underline{k}}_{\frac{1}{1-F_x^+(t)}}[n].$$
			
	\end{theo}
	\begin{proof}
		To prove Eq. (\ref{enrichedwforests}), observed that the configurations of  $\mathcal{A}_{e^{F_x(t)}}[n]$ are rooted trees with $n$ vertices whose fibers are enriched (weighted) with the coefficients of the series $e^{x+F_x^+(t)}=e^xe^{F_x^+(t)}$. Since there are $n$ vertices we have
		$$\mathcal{A}_{e^{F_x(t)}}[n]=e^{nx}\mathcal{A}_{e^{F_x^+(t)}}[n]=(\mathscr{D}^{\boldsymbol\phi})^{n-1}e^{nx}.$$
		To prove Eq. (\ref{enrichedwlinearforests}), observe that the configurations of $\frac{1}{1-F_x(t)}$ are tuples of sets, $$(B_1,B_2,\dots,B_k),$$ each of them weighted with the coefficient $F_x[B_i]$, $i=1,\dots,k$. If some $B_i=\emptyset$, its weight is equal to $F_x[\emptyset]=x$. The fibers of each tree enumerated by $\mathcal{A}_{\frac{1}{1-F_x(t)}}[n]$ are enriched with this weighted tuples. Denote by $\mathcal{A}^{\underline{k}}_{\frac{1}{1-F_x(t)}}[n]$ the configurations of $\mathcal{A}_{\frac{1}{1-F_x(t)}}[n]$ having exactly $k$ nonempty blocks in its fibers. If the fiber of a vertex $v$ has $k_v$ nonempty blocks, we can put as many empty configurations we want between  them (in $k_v+1$ available positions, $k=\sum_{v=1}^n k_v$). Since each of the empty sets in the tuple has weight $x$, the total weight added by the empty configurations is equal to $(1-x)^{-k_v-1}$, for each vertex $v$. Taking the product over all the $n$ vertices of the tree we have that the total weight added by the empty configurations is equal to $(1-x)^{-n-k}$. Then we get
		$$\mathcal{A}^{\underline{k}}_{\frac{1}{1-F_x(t)}}[n]=\frac{\mathcal{A}^{\underline{k}}_{\frac{1}{1-F_x^+(t)}}[n]}{(1-x)^{n+k}}.$$
	\end{proof}
\subsection{Examples and applications}
\begin{ex}\normalfont  
	For the univariate case, $X=\{x\}$, for a series $H(x)$, the trees enriched with $H(F_x(t))$ are counted by 
	$$\mathcal{A}_{H(F_x(t))}[n]=(\phi(x)D)^{n-1}H^n(x).$$ In particular, we have the following formulas for
	\begin{enumerate}
		\item The rooted trees enriched with the series $F_x(t)= \mathcal{A}_{\phi}^{\uparrow}(x,t)=e^{t\phi(x)D}x$ 
		$$\mathcal{A}_{F_x(t)}[n]=(\phi(x)D)^{n-1}x^n.$$
		As examples of this kind of trees we have
		\begin{enumerate}
			\item For the operator $xD$, $F_x(t)=xe^t.$ The coefficient $\mathcal{A}_{xe^t}[n]$ counts the number of rooted trees with $n$  vertices, each vertex having weight $x$. We deduce Cayley's formula 
			$$ \mathcal{A}_{xe^t}[n]=(xD)^{n-1}x^n=n^{n-1}x^n.$$
			\item For the operator $x^2D$,  $$F_x(t)=\frac{x}{1-xt}=\sum_{n=0}^{\infty}n!x^{n+1}\frac{t^n}{n!}.$$ The $n\mathrm{th}$ coefficient of this series enumerates $n$-linear orders each of them having weight $x^{n+1}$. Then, the coefficient $\mathcal{A}_{F_x(t)}[n]$ counts the number of plane trees with $n$  vertices, each vertex and arc having weight $x$. A few computations give us the formula
			$$\mathcal{A}_{\frac{x}{1-xt}}[n]=(x^2D)^{n-1}x^n=n(n+1)\dots(2n-2)x^{2n-1}=n^{(n-1)}x^{2n-1}.$$
		\end{enumerate}

		\item\label{forests} The rooted trees enriched with $e^{F^+_x(t)}=e^{\mathcal{A}_{\phi}^{\uparrow,+}(x,t)},$ (Forests of increasing trees, excluding the empty  tree with weight $x$.) 
		$$\mathcal{A}_{e^{F_x(t)}}[n]=e^{-nx}(\phi(x)D)^{n-1}e^{nx}.$$
		\item\label{orderesforests} The rooted trees enriched with $\frac{1}{1-F_x(t)}$ (linearly ordered forests of increasing trees, including the empty trees  with weight $x$)
		$$\mathcal{A}_{\frac{1}{1-F_x(t)}}[n]=(\phi(x)D)^{n-1}(1-x)^{-n}=\sum_{k=1}^{n-1}\mathcal{A}^{\underline{k}}_{\frac{1}{1-F^+_x(t)}}[n](1-x)^{-(n+k)}.$$ 
	
	\end{enumerate}
	 \end{ex}
	 From formulas in Items \ref{forests} and \ref{orderesforests},  we get a couple of interesting results.
	 \begin{prop}\label{touchardtrees}\normalfont  
	 	The number of trees, enriched with set partitions, each block weighted with $x$, is given by the formula
	 	\begin{equation} \label{tochardtrees}
	 	\mathcal{A}_{e^{x(e^t-1)}}[n]=T_{n-1}(nx),
	 	\end{equation}
	 	\noindent where $T_n(x)$ is the Touchard polynomial.
	 	In particular, the number of set partitions enriched trees is given by $T_{n-1}(n).$ 
	 \end{prop}
	 	\begin{proof}
	 		 By Item \ref{forests}, $\mathcal{A}_{e^{x(e^t-1)}}[n]=e^{nx}(xD)^{n-1}e^{nx}.$ From general Rodrigues-like formula (\ref{Eq:Rodrigues})  it is easy to see that $e^{nx}(xD)^{k}e^{nx}$  is equal to $T_k(nx)$ for arbitrary $k$.
	 	\end{proof}
	 	 See \cite{Husimi}, \cite{Greene}, \cite{Kreweras}, \cite{Gessel} where these numbers are interpreted in terms of hypertrees. The bijection between rooted hypertrees (see definition in \cite{Gessel}) and the present interpretation is easy.
	 \begin{prop}\normalfont The number of trees on $n$ vertices, enriched with partitions whose blocks are linearly ordered, each block weighted with $x$, is given by the sum of numerators in the expansion of
	 $(xD)^{n-1}(1-x)^{-n}$ in simple fractions.\end{prop}
 For example:
 \begin{eqnarray*}(xD)^2\frac{1}{(1-x)^3}&=&\frac{12 x^2}{(1-x)^5}+\frac{3 x}{(1-x)^4}\\(xD)^{3}\frac{1}{(1-x)^4}&=&\frac{120 x^3}{(1-x)^7}+\frac{60 x^2}{(1-x)^6}+\frac{4 x}{(1-x)^5}\\
 (xD)^{4}\frac{1}{(1-x)^5}&=&\frac{1680 x^4}{(1-x)^9}+\frac{1260 x^3}{(1-x)^8}+\frac{210 x^2}{(1-x)^7}+\frac{5 x}{(1-x)^6}.
\end{eqnarray*}

		  \begin{ex}\normalfont  
		  	Consider the set of variables $X=\{x_0,x_1,x_2,\dots\}$. We have
		  	\begin{eqnarray*}
		  	e^{-nx_0}\mathscr{D}^{n-1}e^{nx_0}&=&Y_{n-1}(nx_1,nx_2,\dots,nx_{n-1})\\
		  	e^{-nx_1}\mathscr{D}^{n-1}e^{nx_1}&=&Y_{n-1}(nx_2,nx_3,\dots,nx_{n})
		 \end{eqnarray*}
		 The Bell polynomials $Y_{n-1}(nx_1,nx_2,\dots,nx_{n-1})$ and $Y_{n-1}(nx_2,nx_3,\dots, nx_{n})$, respectively enumerate: \begin{enumerate} \item Rooted trees enriched with set partitions according with the size of their blocks. \item Rooted hypertrees by the size of their hyper edges (see \cite{Gessel}).\end{enumerate}
		 The sum of the numerators of the decomposition in simple fractions of \begin{equation*}
		 \mathscr{D}^{n-1}\frac{1}{(1-x_0)^n}
		 \end{equation*}
		 \noindent gives us the generating function of rooted trees enriched with partitions where the blocks are linearly ordered, enumerated by the size of blocks.
		 \end{ex}
		 \begin{ex}
		 	\normalfont We now modify the operator $\mathscr{D}$ in order to enumerate plane trees and obtain another kind of Lagrange inversion. Let $$\mathscr{D}^{\natural}=\sum_{k=0}^{\infty}(k+1)x_{k+1}\partial_k.$$
		  In this case we obtain $F_0^{\natural}(t)$ a generic ordinary formal power series
		  $$F_0^{\natural}(t)=e^{t\mathscr{D}^{\natural}}x_0=\sum_{k=0}^{\infty}x_k t^k.$$
		  The coefficients of the ordinary generating function of the trees $\mathcal{A}_{F_0^{\natural}}(t)$ enumerates the unlabeled rooted plane trees according with the degrees of its vertices.
		  $$\mathbf{Pf}_n(x_0,x_1,\dots, x_{n-1}):=\frac{\mathcal{A}_{F_0^{\natural}}[n]}{n!}=\frac{1}{n!}(\mathscr{D}^{\natural})^{n-1} x_0^n.$$ For example
		  $$\mathbf{Pf}_5(x_0,x_1,\dots, x_{4})=\frac{\mathcal{A}_{F_0^{\natural}}[5]}{5!}=x_4 x_0^4+2 x_2^2 x_0^3+4 x_1 x_3 x_0^3+6 x_1^2 x_2 x_0^2+x_1^4 x_0.$$
		  Making $x_0=1=h_0$ and $x_n=h_n$, $h_n$ being the homogeneous symmetric function, we get that $\mathbf{Pf}_n(1,h_1,\dots, h_{n-1})$ is the Frobenius character of the linear span of parking functions on $n-1$ elements (as a representation of the symmetric group $\mathfrak{G}_{n-1}$) in terms of the homogeneous symmetric functions \cite{G-H}. 
		  \end{ex}
		  \begin{ex}\normalfont  
		  	Let $X=\{x_1,x_2,x_3,\dots;y_0,y_1,y_2,	\dots\}$. The operator 
		  	\begin{equation}
		  	\mathscr{F}=\sum_{k=1}^{\infty}x_{i+1}\partial_{x_i}+x_1\sum_{k=0}^{\infty}y_{i+1}\partial_{y_i}
		  	\end{equation}
		  	\noindent is the extension of $\mathcal{F}_n$ to the the bicolored infinite number of variables of above. By Eq. (\ref{Faa})
		  	$$F_{y_0}(t)=e^{t\mathscr{F}}y_0=P(Q(t)),$$
		  	\noindent where $P(t)=\sum_{k=0}^{\infty}y_k \frac{t^k}{k!}$ and $Q(t)=\sum_{k=1}^{\infty}x_k\frac{t^k}{k!}$. It is easy to check that   $$F_{x_n}(t)=e^{t\mathscr{F}}x_n=Q^{(n)}(t).$$ By the same argument to obtain  (\ref{Faa}) we get $$F_{y_n}(t)=e^{t\mathscr{F}}y_n=P^{(n)}(Q(t)).$$
		  	Now we can use the one parameter group to deduce Fa\`a di Bruno formula,
		  	\begin{eqnarray*}
		  	\frac{d^n}{dt^n}P(Q(t))=\frac{d^n}{dt^n}e^{t\mathscr{F}}y_0=e^{t\mathscr{F}}\mathscr{F}^ny_0&=&e^{t\mathscr{F}}\sum_{\pi\in \Pi[n]}y_{|\pi|}\prod_{B\in\pi}x_{|B|}\\&=&\sum_{\pi\in\Pi[n]}P^{(|\pi|)}(Q(t))\prod_{B\in\pi}Q^{(|B|)}(t).
		  	\end{eqnarray*}
		  	The configurations of $\mathcal{A}_{P(Q(t))}[n]$ are trees enriched with partitions as in Proposition \ref{touchardtrees}, each partition weighted by $y_{|\pi|}\prod_{B}x_{|B|}$. Its inventory is given by the formula
		  	\begin{equation*}
		  	\mathcal{A}_{P(Q(t))}[n]=\mathscr{F}^{n-1}y_0^n.
		  	\end{equation*}
		  \end{ex}

\noindent

% % % % % % % % % % % % % % % % % % % % % % % % % % % % % % % % %

\bibliographystyle{amsplain}

\end{document}